\documentclass{article}
\usepackage{geometry}
\usepackage{amsmath,amsthm,amsfonts,amssymb,bbm,empheq}
\usepackage{paralist,graphics,epsfig,graphicx,epstopdf,mathrsfs}
\usepackage{float,color,ulem,comment,tabulary,booktabs}
\usepackage{multirow}
\usepackage{epsf,epsfig,subfigure}
\usepackage[colorlinks=true]{hyperref}
\hypersetup{urlcolor=blue,linkcolor=blue,citecolor=blue,colorlinks=true}

\def\R{\mathbb{R}}

\newcommand{\dd }{{\mathrm{d}}}

\newcommand\doubleRule{\toprule\toprule}

\makeatletter\@addtoreset{equation}{section}
\newfloat{figure}{H}{lof}
\newfloat{table}{H}{lot}
\floatname{figure}{\figurename}
\floatname{table}{\tablename}
\newtheorem{theorem}{Theorem}[section]

\newtheorem{assumption}[theorem]{Assumption}

\newtheorem{definition}[theorem]{Definition}

\newtheorem{lemma}[theorem]{Lemma}

\newtheorem{proposition}[theorem]{Proposition}
\newtheorem{remark}[theorem]{Remark}

\begin{document}
	\title{\textbf{A cell-cell repulsion model on a hyperbolic Keller-Segel equation}}
	\author{\textsc{Xiaoming Fu\thanks{The research of this author is supported by China Scholarship Council.} , Quentin Griette and Pierre Magal}
		\\
		{\small \textit{Univ. Bordeaux, IMB, UMR 5251, F-33400 Talence, France}} \\
		{\small \textit{CNRS, IMB, UMR 5251, F-33400 Talence, France.}}\\
		{\small \textit{}} }
	\maketitle
	
	\begin{abstract}
		In this work, we discuss a cell-cell repulsion population dynamic model based on a hyperbolic Keller-Segel equation with two populations. This model can well describe the cell growth and dispersion in the cell co-culture experiment in the work of Pasquier et al. \cite{Pasquier2011}.  With the notion of solutions integrated along the characteristics, we prove  the existence and uniqueness of the solution and the segregation property of the two species. From a numerical perspective, we can also observe that the model admits a competitive exclusion (the results are different from the corresponding ODE model). More importantly, our model shows the complexity of the short term (6 days) co-cultured cell distribution depending on the initial distribution of each species.  Through numerical simulations, the impact of the initial distribution on the population ratio lies in the initial total cell number and our study shows that the population ratio is not impacted by the law of initial distribution. We also find that a fast dispersion rate gives a short-term advantage while the vital dynamic contributes to a long-term population advantage.  
	\end{abstract}

	\section{Introduction}
	
	In many recent biological experiments, the co-culture of multiple types of cells has been used for a better understanding of cell-cell interactions. This is a typical case in the context of studying cancer cells where the interaction between cancer cells and normal cells plays a crucial role in tumor development as well as in the resistance of cells to chemotherapeutic drugs. The goal of this work is to introduce a  mathematical model taking care of the cell growth together with the spatial segregation property between two types of cells. Such a phenomenon was observed by Pasquier et al. \cite{Pasquier2012}. They studied the protein transfer between two types of human breast cancer cell. Over a 7-day cell co-culture, the spatial competition was observed between these two types of cells and a clear boundary was formed between them on day 7 (see Figure \ref{Figure1}). Segregation property in cell co-culture was also studied recently by Taylor et al. \cite{Taylor2017}. They compared the experimental results with an individual-based model. They found the heterotypic repulsion and homotypic cohesion can account for the cell segregation and the border formation. A similar segregation property is also found in the mosaic pattern between nections and cadherins in the experiments of Katsunuma et al. \cite{Katsunuma2016}. 
	\begin{figure}[H]
		\begin{center}
			\includegraphics[width=0.6\textwidth]{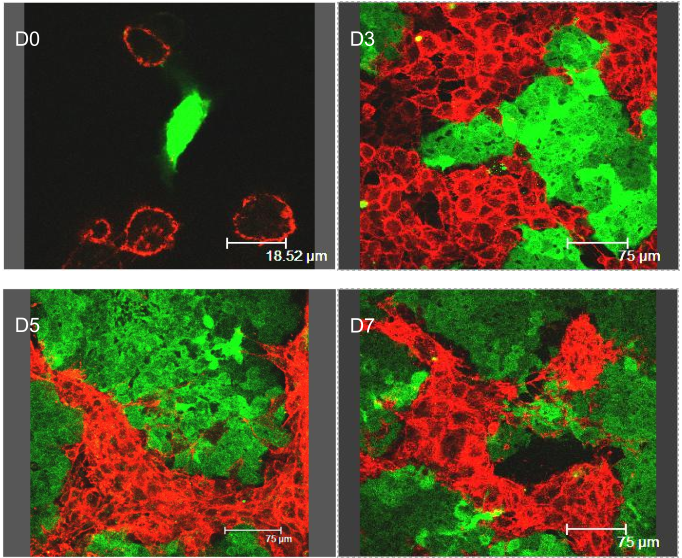}
		\end{center}
		\caption{\textit{Direct immunodetection of P-gp transfers in co-cultures of sensitive (MCF-7) and resistant (MCF-7/Doxo) variants of the human breast cancer cell line.}}\label{Figure1}
	\end{figure}

	The early attempts to explain the segregation property by continuum equations date back to 1970s. Shigesada, Kawasaki and Teramoto \cite{Shigesada1979} studied segregation with a nonlinear diffusion model and they found the spatial segregation acts to stabilize the coexistence of two similar species, relaxing the competition among different species. Lou and Ni \cite{Lou1996} generalized the model and studied the steady state problem for the self/cross-diffusion model. For the nonlinear diffusion model, Bertsch et al. \cite{Bertsch12} in their work proved the existence of segregated solutions when the reaction term is of Lotka-Volterra type.     
	
	Here instead of using nonlinear diffusion models, we will focus on a (hyperbolic) Keller-Segel model. Such models have been used to describe the attraction and the repulsion of cell populations known as chemotaxis models. Theoretical and mathematical modeling of chemotaxis can date to the pioneering works of Patlak \cite{Patlak1953} in the 1950s and Keller and Segel \cite{Keller1971} in the 1970s. It has become an important model in the description of tumor growth or embryonic development. We refer to the review papers of Horstmann \cite{Horstmann2003} and Hillen and Painter \cite{Hillen2009} for a detailed introduction about the Keller--Segel model. To our best knowledge a model taking care of segregation property and cell-cell repulsion of Keller-Segel type has not been studied.  
	
	As we will explain in the paper, our model can also be regarded as a nonlocal advection model. Recently, implementing nonlocal advection models for the study of cell-cell adhesion and repulsion has attracted a lot of attention. As pointed out by many biologists, cell-cell interactions do not only exist in a local scope, but a long-range interaction should be taken into account to guide the mathematical modeling.  Armstrong, Painter and Sherratt \cite{Armstrong2006} in their early work purposed a model (APS model) under the principle of the local diffusion plus the nonlocal attraction driven by the adhesion forces to describe the phenomenon of cell mixing, full/partial engulfment and complete sorting in the cell sorting problem. Based on the APS model, Murakawa and Togashi \cite{Murakawa2015} thought that the population pressure should come from the cell volume size instead of the linear diffusion. Therefore, they changed the linear diffusion term into a nonlinear diffusion in order to capture the sharp fronts and the segregation in cell co-culture. Carrillo et al. \cite{Carrillo2019} recently purposed a new assumption on the adhesion velocity field and their model showed a good agreement in the experiments in the work of Katsunuma et al.  \cite{Katsunuma2016}. The idea of the long-range attraction and short-range repulsion can also be seen in the work of Leverentz, Topaz and Bernoff \cite{Leverentz2009}. They considered a nonlocal advection model to study the asymptotic behavior of the solution. By choosing a Morse-type kernel which follows the attractive-repulsive interactions, they found the solution can asymptotically spread, contract (blow-up), or reach a steady-state.  Burger, Fetecau and Huang \cite{Burger2014} considered a similar nonlocal adhesion model with nonlinear diffusion, they studied the well-posedness of the model and proved the existence of a compact supported, non-constant steady state. Dyson et al. \cite{Dyson2010}  established the local existence of a classical solution for a nonlocal cell-cell adhesion model in spaces of uniformly continuous functions. For Turing and Turing-Hopf bifurcation due to the nonlocal effect, we refer to Ducrot et al. \cite{Ducrot2018} and Song et al. \cite{Song2019}. We also refer the readers to Mogliner et al. \cite{Mogilner2003}, Eftimie et al. \cite{Eftimie2007}, Ducrot and Magal \cite{Ducrot2014}, Fu and Magal \cite{Fu2018} for more topics about nonlocal advection equations. For the derivation of such models, readers can refer to the work of Bellomo et al. \cite{Bellomo2012} and Morale, Capasso and Oelschl\"ager \cite{Morale2005}.

	In this article, we consider a two-dimensional bounded domain (a flat circular petri dish). We use the notion of solution integrated along the characteristics.  Thanks to the appropriate boundary condition of the pressure equation (see Equation \eqref{eq1.4}), we deduce that the characteristics stay in the domain for any positive time. The positivity of solutions, the segregation property and a conservation law follow from the notion of solution as well. The main goal in this article is to investigate the complexity of the short-term (6 days) co-cultured cell distribution depending on the initial distribution of each species.  Through the numerical simulations, we investigate the impact of the initial population number (as well as the law of initial distributions) on the population ratio. In the above mentioned literature, the numerical simulation are restricted to a rectangular domain with periodic boundary conditions. It is worth mentioning that here the domain is circular with no flux boundary condition for the pressure which requires a finite volume method (see Appendix \ref{App:NumScheme}).  
	
	The plan of the paper is the following. In Section 2, we present the model for the single-species case and we prove the local existence and uniqueness of solutions by considering the solution integrated along the characteristics and we prove a conservation law. Section 3 is devoted to the numerical analysis of the model. In Section 3.1, we consider the model homogeneous in space which corresponds to an ODE model that has been previously studied by Zeeman \cite{Zeeman1995}. In Section 3.2, we investigated the competitive exclusion principle and the impact of the initial distribution on population ratio.  The spatial competition due to the dispersion coefficients and cell kinetics is considered in Section 3.3. Section 4 is devoted to discussion and conclusion and all the mathematical details are presented in the Appendix.

	\section{Mathematical modeling}
	\subsection{Single species model}
	Let us consider the following one-species model 
	\begin{equation}\label{eq1.3}
		\begin{cases}
			\partial_{t}u(t,x) -d\,\mathrm{div }\bigl(u(t,x)\nabla P(t,x)\bigr)=
			u(t,x)h(u(t,x))&\text{ in }  (0,T]\times\Omega\\
			u(0,x)=u_0(x) &\text{ on } \overline{\Omega},
		\end{cases}
	\end{equation}
	where $ P $ satisfies the following elliptic equation
	\begin{equation}\label{eq1.4}
		\begin{cases}
			\big(I-\chi\Delta \big)P(t,x) = u(t,x) &\text{ in }  (0,T]\times\Omega\\
			\nabla P(t,x)\cdot \mathbf{\nu}(x) =0 &\text{ on }  [0,T]\times\partial\Omega,
		\end{cases}
	\end{equation}	
	We denote $ \Omega\subset \R^2 $ to be the unit open disk centered at $ \mathbf{0}=(0,0) $ with radius $ r=1 $, i.e., $\Omega = B_{\R^2}(\mathbf{0},1)$. Here 
	$ \mathbf{\nu} $ is the outward normal vector, $ d $ is the dispersion coefficient, $ \chi $ is the sensing coefficient.	
	The divergence, gradient and Laplacian are taken with respect to $ x $.  System \eqref{eq1.3}-\eqref{eq1.4} can be regarded as a hyperbolic Keller-Segel equation (with chemotactic repulsion) on a bounded domain.  
	
	\begin{remark} Equation \eqref{eq1.4} can be derived from the following parabolic equation (which is the classical case in the Keller-Segel equation \cite{Horstmann2003}) as  $\varepsilon$ goes to $0$:
		\begin{equation}
			\varepsilon \partial_t P(t,x) = \chi \Delta P(t,x) +u(t,x) -P(t,x).
		\end{equation}	
		The process of letting $\varepsilon \to 0$ corresponds to the assumption that the dynamics of the chemorepellent is fast compared to the evolution of the cell density. In the case of chemoattractant a variant of such a model was considered by Perthame and Dalibard \cite{Perthame2009}, Calvez and Dolak-Stru\ss \cite{Calvez2008}.  
	\end{remark}	
	
	\begin{remark} As we mentioned in the introduction, Equation \eqref{eq1.4} can be regarded as a non-local integral equation by using Green's representation 
			\begin{equation*}
				P(t,x)=\int_\Omega \kappa(x,y)u(t,y) dy. 
			\end{equation*}
	\end{remark}
	\subsubsection{The invariance of domain $ \Omega $ and the well-posedness of the model}
	We remark that in System \eqref{eq1.3}-\eqref{eq1.4} we do not impose any boundary condition directly on $ u $. Instead, the boundary condition here is induced by $  \nabla P\cdot \nu =0 $. If we consider the associated characteristics flow of \eqref{eq1.3}-\eqref{eq1.4}
	\begin{equation}\label{eq2.3}
		\begin{cases}
			\frac{\partial\phantom{t}}{\partial t}\Pi(t,s;x)=-d\,\nabla P(t,\Pi(t,s;x))\\
			\Pi(s,s;x)=x\in \Omega.
		\end{cases}
	\end{equation}
	We can prove (see Appendix \ref{App:Invariance}) the characteristics can not leave the domain $ \Omega $ (see Figure \ref{Figure2} for an illustration). In fact, we can prove for any $ t>0 $, the mapping $ x\mapsto \Pi(t,0;x) $ is a bijection from $ \Omega $ to itself (see Lemma \ref{lem2.7}).	We consider the solution along the characteristics
	\[ w(t,x):=u(t,\Pi(t,0;x))\quad x\in \Omega,t>0. \]
	Taking any $ x\in \Omega $, there exists $ y\in \Omega $ such that $ x=\Pi (t,0;y) $, and since  
	\[ w(t,y)=w(t,\Pi(0,t;x))=u(t,x), \]
	we can reconstruct the solution $ u(t,\cdot) $ from $ w(t,\cdot) $ and $ \{\Pi(t,s,\cdot)\}_{t,s\in[0,T]} $ on $ \Omega $. 
	\begin{figure}[H]
		\begin{center}
			\includegraphics[width=0.38\textwidth]{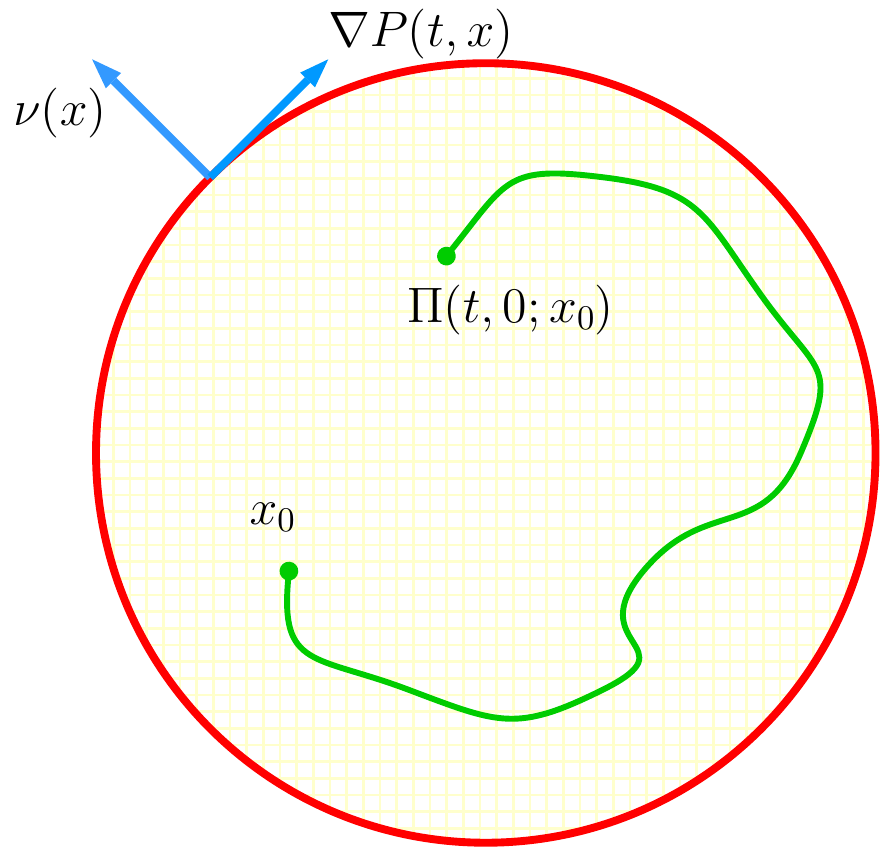}
		\end{center}
		\caption{\textit{An illustration for the invariance of domain $ \Omega $. The green curve represents a trajectory of the characteristics. }}\label{Figure2}
	\end{figure}
	
	\begin{assumption}\label{ass:P}
		Assume the vector field $ (t,x)\mapsto\nabla P(t,x)  $ is continuous in $ [0,T]\times\overline{\Omega} $ and lipschitzian with respect to $ x $ in $ [0,T]\times\overline{\Omega} $.
	\end{assumption}
	\begin{remark}
		Assumption \ref{ass:P} is a sufficient condition for the existence and uniqueness of the characteristic flow $ \lbrace\Pi(t,s;\cdot)\rbrace_{t,s\in[0,T]} $ in \eqref{eq2.3}. 
	\end{remark}
	\begin{definition}
		For any bounded open domain $ \Omega $. If $ u: \Omega \to \R $ is bounded and continuous, we write 
		\[ \|u \|_{C(\overline{\Omega})} :=\sup_{x\in \Omega} |u(x)|. \]
		For any $ \alpha\in (0,1] $, the $ \alpha^{th} $--H\"older norm of $ u:\Omega \to \R $ is 
		\[\|u \|_{C^{0,\alpha}(\overline{\Omega})}:=\|u\|_{C(\overline{\Omega})}+[u]_{C^{0,\alpha}(\overline{\Omega})},  \]
		where
		\[ [u]_{C^{0,\alpha}(\overline{\Omega})} := \sup_{x,y\in \Omega\atop x\neq y} \bigg\lbrace \frac{|u(x)-u(y)|}{|x-y|^{\alpha}} \bigg\rbrace. \]
		The H\"older space $ C^{k,\alpha}(\overline{\Omega}) $ consists of all functions $ u\in C^{k}(\overline{\Omega}) $ for which the norm 
		\[ \| u\|_{C^{k,\alpha}(\overline{\Omega})}:= \sum_{|\alpha|\leq k} \| D^{\alpha}u \|_{C(\overline{\Omega})} + \sum_{|\alpha|= k} [D^\alpha u]_{C^{0,\alpha}(\overline{\Omega})}  \]
		is finite.
	\end{definition}
	\begin{lemma}\cite[Theorem 6.30 and 6.31]{Gilbarg2001}\label{lem:elliptic}
		Let $\Omega\subset\R^2 $ to be a unit open disk. Consider the following elliptic equation
		\begin{equation}\label{5.1}
			\begin{cases}
				(I-\chi\Delta)P(x)=u(x)\quad &x\in\Omega\\
				\nabla P(x)\cdot\mathbf{\nu}(x)=0\quad &x\in\partial \Omega,
			\end{cases}
		\end{equation}
		where $ \mathbf{\nu} $ is the outward unit normal vector on $ \partial \Omega $. Then for all $ u\in C^{0,\alpha}(\overline{\Omega}) $, the elliptic problem \eqref{5.1} has a unique solution $ P \in C^{2,\alpha}(\overline{\Omega}) $. Moreover,
		\[ \|P\|_{C^{2,\alpha}(\overline{\Omega})} \leq C \|u\|_{C^{0,\alpha}(\overline{\Omega})}, \]
		where $ C=C(\alpha,\chi,\Omega) $.
	\end{lemma} 
	The following theorem tells us if we choose our initial value $ u_0 $ to be smooth enough, then Assumption \ref{ass:P} can be satisfied and the existence and the uniqueness of solutions follow. 
	
	\begin{theorem}[\textbf{Existence and uniqueness of the solution along the characteristic}]\label{thm:existence}
		Let $ u_0\in W^{1,\infty}(\Omega)\cap C_+^0(\overline{\Omega}) $. Then for some $ T>0$ there exists a unique non-negative  solution $u\in C\left([0,T];C_+^{0}(\overline{\Omega})\right)$ to \eqref{eq1.3}-\eqref{eq1.4} which satisfies $u(t=0, x)=u_0(x)$. Moreover for any $t\in [0,T]$, we have $u(t,\cdot)\in W^{1,\infty}(\Omega)$ and $\sup_{t\in[0,T]}\|u(t,\cdot) \|_{W^{1,\infty}(\Omega)}<\infty $.
	\end{theorem}
The proof the above theorem will be detailed in Appendix \ref{app:existence}.	
\begin{remark}\label{rem:regularityofP}  Since for any $t\in [0,T]$ and for any $\alpha\in(0,1)$, we have $u(t,\cdot)\in W^{1,\infty}(\Omega)\hookrightarrow C^{0,\alpha}(\overline{\Omega})$, we deduce from Lemma \ref{lem:elliptic} that $P(t,\cdot)\in C^{2,\alpha}(\overline{\Omega})$. Therefore, $(t,x)\to \nabla P(t,x) $ is continuous (since $P \in C([0,T];C^1(\overline{\Omega}))$) and lipchitzian with respect to $ x $ which implies Assumption \ref{ass:P}.
	\end{remark}
	\subsubsection{Conservation law on a volume}
	If the reaction term $ h\equiv 0 $ in System \eqref{eq1.3}-\eqref{eq1.4}, the boundary condition implies the conservation law for $ u $. This can be seen through the solution along the characteristics. 
	In fact, we have the following conservation law.
	\begin{theorem}\label{thm:conservation}
		For each volume $A\subset \Omega $ and each $0\leq s\leq t$ we have
		\begin{equation*}
			\int_{\Pi(t,s;A)}u(t,x)\dd x=\int_A \exp\left(\int_s^t
			h\left(u\left(l,\Pi(l,s;z)\right)\right)\dd  l\right)u(s,z)\dd  z.
		\end{equation*}
		In particular, if we have $ h=0 $, then for any  $0\leq s\leq t$
		\begin{equation*}
			\int_{\Pi(t,s;A)}u(t,x)\dd x=\int_A u(s,z)\dd  z.
		\end{equation*}
		This means the total number of cell in the volume $ A $ is constant along the volumes $\Pi(t,s;A)$.
	\end{theorem}
	Before proving Theorem \ref{thm:conservation}, we need the following lemma.
	\begin{lemma}\label{lem2.7}
		Let $ T >0 $ and $\left\{ \Pi(t,s;x)\right\} _{t,s\in \left[ 0,T \right] }$ to be the characteristic flow generated by \eqref{eq2.3}. Then the map $x\mapsto \Pi(t,s;x)$ is continuously differentiable and one has the determinant of Jacobi matrix: 
		\begin{equation}\label{eq2.6}
			\det J_{\Pi}(t,s;x)=\exp \left( \int_{s}^{t}\frac{d}{\chi} \left(u(l,\Pi(l,s;x))-P(l,\Pi(l,s;x)) \right)\dd l\right).
		\end{equation} 
		where $ J_{\Pi}(t,s;x) $ is the Jacobian matrix of $ \Pi (t,s;x) $ with respect to $ x $ at point  $ (t,s;x) $.
	\end{lemma}
	\begin{proof}
		From Theorem \ref{thm:existence} and Remark \ref{rem:regularityofP}, the mapping $ (t,x)\to P(t,x) $ is $ C([0,T];C^{1}(\overline{\Omega})) $ and $P(t,\cdot) \in C^{2,\alpha}(\overline{\Omega})$ for any $\alpha\in (0,1)$ if $ u_0 \in W^{1,\infty}(\Omega)$. This ensures the characteristics $ x\to \Pi (t,s;x) $ is continuously differentiable. Taking the partial derivative of Equation \eqref{eq2.3} with respect to $ x $ yields
		\begin{equation*}
			\begin{cases}
				\partial_{t}J_{\Pi}(t,s;x)= -d\,J_{\nabla P}(t,\Pi(t,s;x)) J_{\Pi}(t,s;x) \\ 
				J_{\Pi}(s,s;x)=\mathrm{Id},
			\end{cases}
		\end{equation*}
		where $ J_{\nabla P}(t,\Pi(t,s;x)) $ is the Jacobian matrix of $ \nabla P(t,x) $ with respect to $ x $ at point  $ (t,\Pi(t,s;x)) $.
		For any  matrix-valued $ C^1 $ function $ A:t\mapsto A(t) $, the Jacobian formula reads as follows 
		\begin{equation*}
			\frac{\dd  }{\dd  t} \det A(t) = \det A(t)\times \text{Trace}\left(A^{-1}(t) \frac{d}{d t}A(t)\right).
		\end{equation*}
		Hence, we obtain
		\begin{align*}
			\frac{\dd  }{\dd  t} \det J_{\Pi}(t,s;x)  &= \det J_{\Pi}(t,s;x) \times \text{Trace}\left( J_{\Pi}(t,s;x)^{-1} J_{\nabla P}(t,\Pi(t,s;x))  J_{\Pi}(t,s;x)\right)\\
			&= \det J_{\Pi}(t,s;x) \times \text{Trace}\left( J_{\nabla P}(t,\Pi(t,s;x))\right)\\
		\end{align*}
		and since $ \text{Trace}\left( J_{\nabla P}(t,\Pi(t,s;x))\right)= (\Delta P)(t,\Pi(t,s;x))=-\frac{1}{\chi} \left(u(t,\Pi(t,s;x))-P(t,\Pi(t,s;x))\right) $. Therefore, we have
		\begin{equation*}
			\begin{cases}
				\dfrac{\dd  }{\dd  t} \det J_{\Pi}(t,s;x)  = \det J_{\Pi}(t,s;x) \times \dfrac{d}{\chi}\big[u(t,\Pi(t,s;x))-P(t,\Pi(t,s;x))\big]\\
				\det J_{\Pi}(s,s;x) = 1.
			\end{cases}
		\end{equation*}
		Therefore the result follows.
	\end{proof}
	\begin{proof}[Proof of Theorem \ref{thm:conservation}]
		Let $\left\{ \Pi(t,s;x)\right\} _{t,s\in \left[ 0,T \right] }$ to be the characteristic flow generated by \eqref{eq2.3}. Given any measurable set $ A\subset \Omega $ and any $ 0\leq s\leq t $, we integrate $ u(t,x) $ over the volume $ \Pi(t,s;A) $ with respect to $ x $ 
		\begin{equation}\label{eq2.7}
			\int_{\Omega} \mathbbm{1}_{\Pi(t,s;A)}(x)u(t,x)\dd x= \int_{\Omega}\mathbbm{1}_{A}(z)u(t,\Pi(t,s;z))\det J_{\Pi}(t,s;z)\dd z,
		\end{equation}
		where we changed the variable $ x $ to $ \Pi(t,s;z) $ on the right-hand-side. 
		
		For the right-hand-side, we will prove in Appendix \ref{app:existence} that 
		\[ 
		u(t,\Pi(t,s;z))=u(s,z)\exp\left(\int_s^t h(u(l,\Pi(l,s;z)))+\frac{d}{\chi}\left(P(l,\Pi(l,s;z))-u(l,\Pi(l,s;z))\right)\dd l \right). 
		\]
		Combining with \eqref{eq2.6} we obtain that 
		\[ u(t,\Pi(t,s;z))\det J_{\Pi}(t,s;z) = u(s,z)\exp\left(\int_s^t h(u(l,\Pi(l,s;z)))\dd l \right). \]
		Substituting the above equation into \eqref{eq2.7} gives us
		\begin{equation*}
			\int_{\Omega} \mathbbm{1}_{\Pi(t,s;A)}(x)u(t,x)\dd x= \int_{\Omega}\mathbbm{1}_{A}(z) u(s,z)\exp\left(\int_s^th(u(l,\Pi(l,s;z)))\dd l \right)\dd z,
		\end{equation*}
		which is equivalent to 
		\begin{equation*}
			\int_{\Pi(t,s;A)}u(t,x)\dd x=\int_A \exp\left(\int_s^t
			h\left(u\left(l,\Pi(l,s;z)\right)\right)\dd  l\right)u(s,z)\dd  z.
		\end{equation*}
		The result follows.
	\end{proof}
	\subsection{Multi-species model}
	\subsubsection{Multi-species ODE model}
	Let us consider the corresponding two species model without the spatial variable $ x $ that is $ u_i=u_i(t) $ for $ i=1,2 .$
	\begin{equation}\label{eq:odemodel}
		\begin{cases}
			\dfrac{\dd u_i}{\dd t} =u_i h_i(u_1,u_2)\quad i=1,2,\\
			u_i(0)= u_{i,0}\in\R_+.
		\end{cases}
	\end{equation}
	We adopt the Lotka-Volterra model by letting
	\begin{equation}\label{eq:h}
		h_i(u_1,u_2)= b_i-\delta_i-\sum_{j=1}^{2} a_{ij} u_j,\quad i=1,2. 
	\end{equation}
	where $b_i>0,\, i=1,2$ are the growth rate, $a_{ij}\geq 0,\, i\neq j$ represent the mutual competition between the species, $ a_{ii} $ is the competition among the same species  and $ \delta_i $ is the additional mortality rate caused by drug treatment. In Section 2.2.1 we will always assume $ \delta_i =0$ for $ i=1,2 $ (when $ \delta_i>0 $, one can regard $ b_i-\delta_i $ as a whole).
	If we consider \eqref{eq:odemodel} in the absence of the other species, we can rewrite \eqref{eq:h} as  
	\begin{equation*}
		h_i(u_1,u_2)= b_i- a_{ii} u_i,\quad i=1,2. 
	\end{equation*}
	We always assume that for each $ i $, $ a_{ii}>0 $ meaning that each species alone exhibits logistic growth. 
	This model has been considered by many authors (for example, see \cite{Murray2003,Zeeman1995}). Here we give a short summary of some qualitative properties of the solution to \eqref{eq:odemodel} in order to compare with the PDE model.
	\medskip
	 
	\noindent\textbf{Equilibrium and stability for \eqref{eq:odemodel}-\eqref{eq:h}}
	
	One can easily compute the system has the following equilibrium
	\[ E_0 =(0,0),\quad E_1= \left(P_1,0\right),\quad E_2= \left(0,P_2\right),\quad E^*=(u_1^*,u_2^*) \]
	where 
	\begin{equation}\label{eq:equilibrium}
		P_1:=\frac{b_1}{a_{11}},\quad P_2:=\frac{b_2}{a_{22}},\quad E^* = \left(\dfrac{a_{22} b_1-a_{12} b_2}{a_{11} a_{22}-a_{12} a_{21}}, \dfrac{a_{21} b_1-a_{11} b_2}{a_{12} a_{21}-a_{11} a_{22}}\right)
	\end{equation} 
	The solution $ E^* $ is only of relevance when $ a_{12} a_{21}\neq a_{11} a_{22} $ and $ (u_1^*,u_2^*) $ is strictly positive, which is equivalent to say 
	\[ \begin{cases}
	\dfrac{a_{12}}{a_{11}}>\dfrac{P_1}{P_2}\\
	\dfrac{a_{21}}{a_{22}}>\dfrac{P_2}{P_1}
	\end{cases} \text{ or }\quad \begin{cases}
	\dfrac{a_{12}}{a_{11}}<\dfrac{P_1}{P_2}\\
	\dfrac{a_{21}}{a_{22}}<\dfrac{P_2}{P_1}
	\end{cases}.\]
	
	
	A scheme of the qualitative behavior of the phase trajectory is given in Figure \ref{Figure:competition} by numerical simulations.
	
	\begin{figure}[H]
		\begin{center}
			\includegraphics[width=0.33\textwidth]{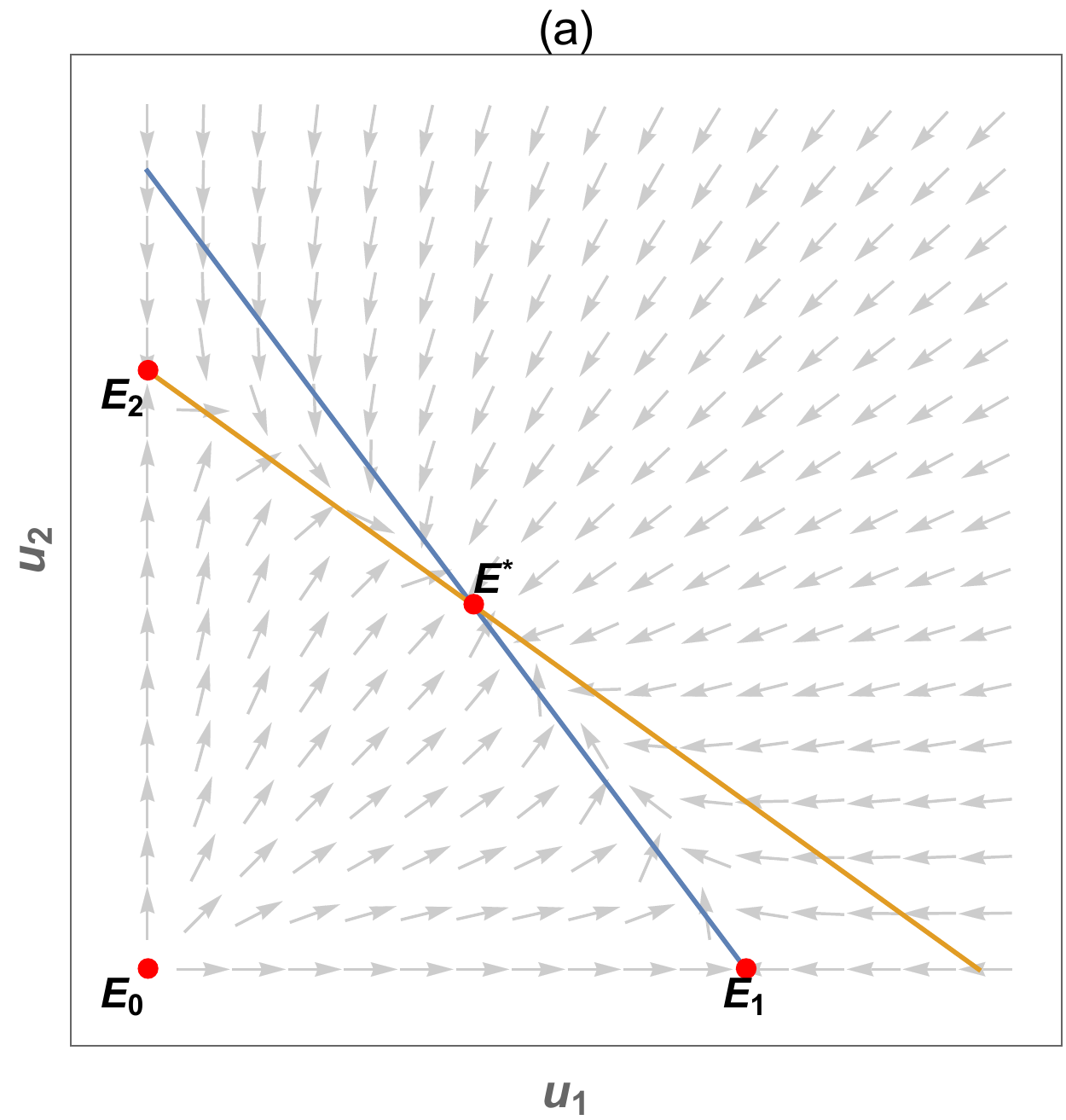}\,\,\includegraphics[width=0.33\textwidth]{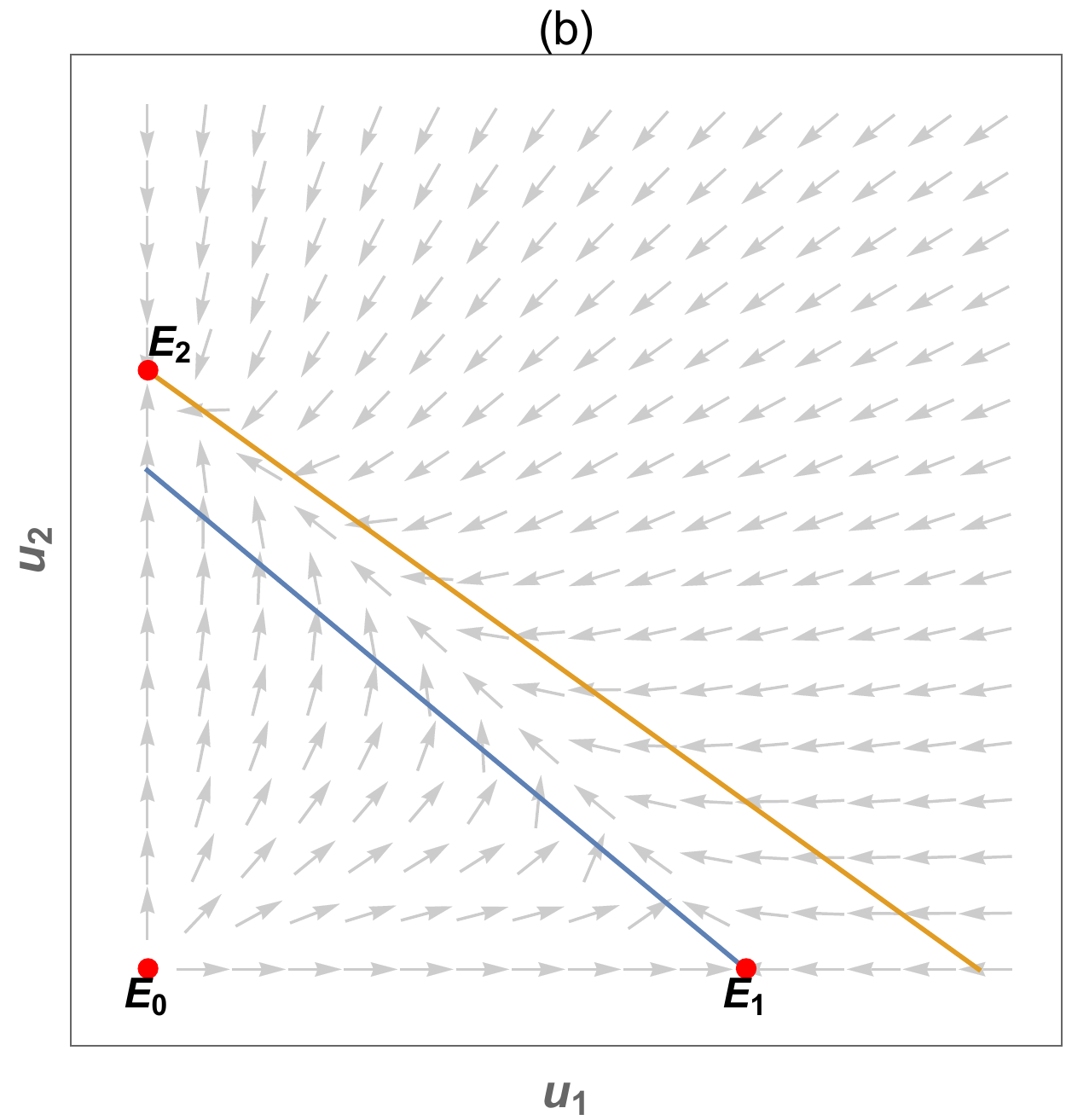}\\
			\includegraphics[width=0.33\textwidth]{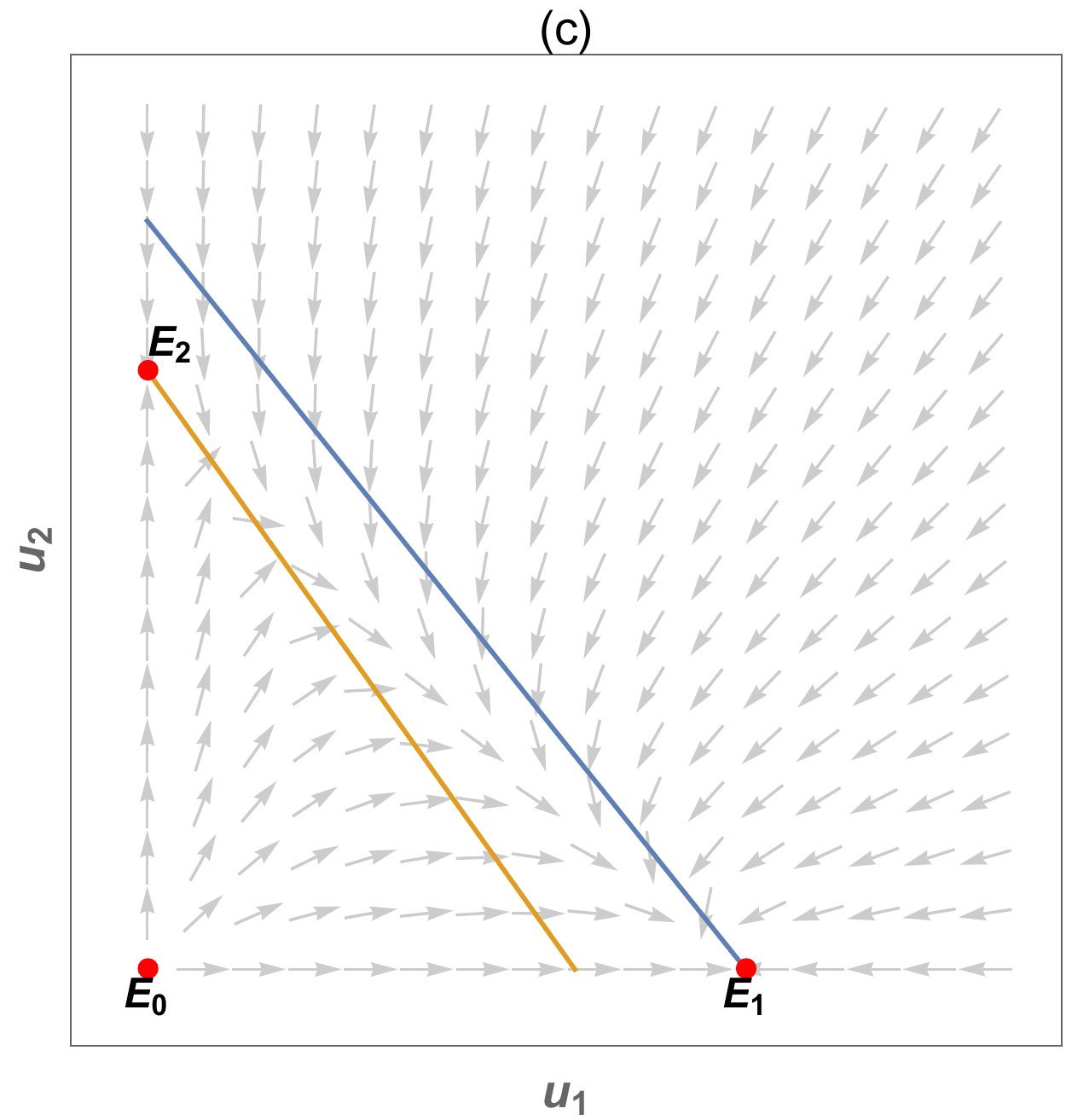}\,\,
			\includegraphics[width=0.33\textwidth]{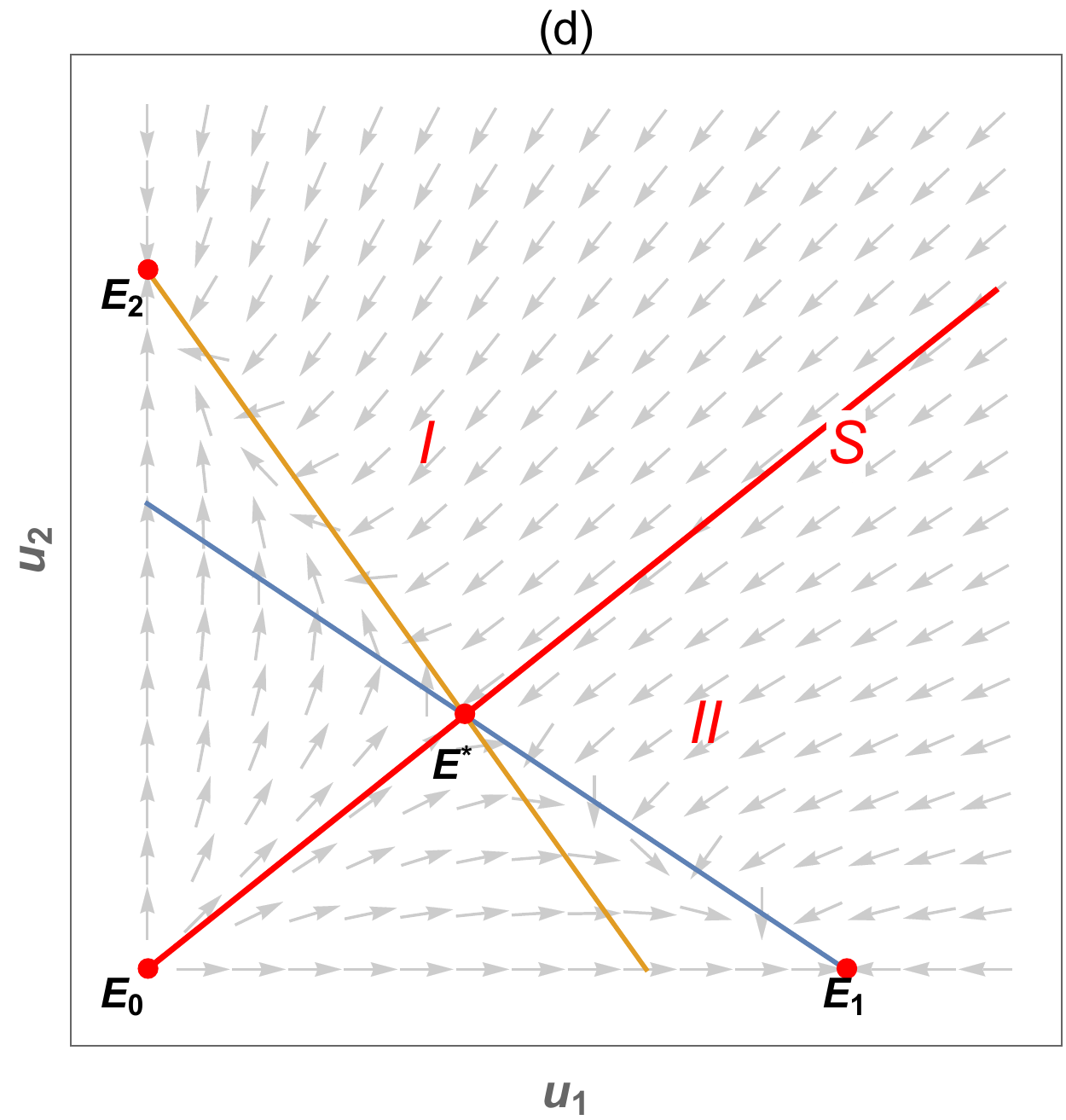}
		\end{center}
		\caption{\textit{A scheme of the qualitative behavior of the phase trajectory for various cases. {\rm\textbf{(a)}} $ a_{12}/a_{11}  < P_1/P_2,\, a_{21}/a_{22}  < P_2/P_1.$ Only the positive steady	state $ E^* $ is stable and all trajectories tend to it. {\rm\textbf{(b)}}	$ a_{12}/a_{11}> P_1/P_2, a_{21}/a_{22} < P_2/P_1.$ Only one stable steady state $ E_2 $ exists with the whole positive quadrant its domain of attraction. {\rm\textbf{(c)}}	$ a_{12}/a_{11} < P_1/P_2, a_{21}/a_{22}  > P_2/P_1.$ Only one stable steady state $ E_1 $ exists with the whole positive quadrant its domain of attraction. {\rm\textbf{(d)}} $ a_{12}/a_{11}   > P_1/P_2, a_{21}/a_{22} > P_2/P_1.$ $ E_1 $ and $ E_2 $ are stable steady	states, each of which has a domain of attraction namely \textbf{I} and \textbf{II}, separated by a separatrix \textbf{S} which is the stable manifold of equilibria $ E^* $. }}\label{Figure:competition}
	\end{figure}
	We adapt the main stability results from Zeeman \cite{Zeeman1995} where the author considered a general $ n $--species extinction case, Murray \cite[Chapter 3.5]{Murray2003} and Hirsch \cite[Chapter 11]{Hirsch2012}  to system \eqref{eq:odemodel}-\eqref{eq:h} for the following fours cases (i)-(iv) and discuss their biological implications.
	\begin{proposition}\label{prop:ode}
		For system \eqref{eq:odemodel}-\eqref{eq:h}, suppose for each $ i=1,2 $, $ b_i>0,a_{ii}>0 $ and $ a_{ij}\geq 0 $ for any $ i\neq j $. Let $ P_1=a_{11}/b_1,P_2=a_{22}/b_2 $ be the equilibrium for each species alone and assume the initial value $ (u_{1,0},u_{2,0})  $ lies strictly in the first quadrant that is $ u_{1,0}>0 $ and $ u_{2,0}>0 $. Then for the following four cases we have
		\begin{itemize}
			\item[\rm(i).]  $ a_{12}/a_{11}  < P_1/P_2,\, a_{21}/a_{22}  < P_2/P_1.$ This case corresponds to Figure \ref{Figure:competition} (a).  The system \eqref{eq:odemodel} has four positive equilibrium, namely $ E_0,E_1,E_2$ and $ E^* $. In such case, only $ E^* $ is global globally asymptotic stable in the region $ \{(u_1,u_2)\in \R^2 \,|\, u_1>0,u_2>0\} $. 
			\item[\rm(ii).] $ a_{12}/a_{11}> P_1/P_2, a_{21}/a_{22} < P_2/P_1.$ This case corresponds to Figure \ref{Figure:competition} (b). The system \eqref{eq:odemodel} has three positive equilibrium, namely $ E_0,E_1$ and $ E_2 $. Only $ E_2 $ is globally stable in the positive quadrant excepted for the axis $u_1=0$. 
			\item[\rm(iii).] 	$ a_{12}/a_{11} < P_1/P_2, a_{21}/a_{22}  > P_2/P_1.$ This case corresponds to Figure \ref{Figure:competition} (c). The analysis of the stability is similar to the case (ii). Only $ E_1 $ is globally stable in the positive quadrant excepted for the axis $u_2=0$. 
			\item[\rm(iv).]  $ a_{12}/a_{11}   > P_1/P_2, a_{21}/a_{22} > P_2/P_1.$ This case corresponds to Figure \ref{Figure:competition} (d). In this case, system \eqref{eq:odemodel} has four equilibrium, where $ E_1 $ and $ E_2 $ are stable while $ E^* $ is a saddle point. The steady states $ E_1 $ and $ E_2 $ have two non-overlapping domains of attraction, separated by the stable manifold \textbf{S} of equilibria $ E^* $. 
		\end{itemize}
	\end{proposition}
	\begin{remark}
		Although among the four cases, {\rm (ii)} and {\rm (iii)} always lead to the principle of exclusion and so do (iv) due to the natural perturbation in population levels, we still have the case {\rm (i)} where the two species can coexist in the long term. As we further develop our PDE model for \eqref{eq:odemodel}, we can show numerically that the principle of exclusion dominates even when case {\rm (i)} is satisfied and this is a evident difference compared to ODE model \eqref{eq:odemodel}.
	\end{remark}
	
	\subsubsection{Multi-species PDE model}
	We study a two species population dynamic model on a unit open disk $ \Omega\subset \R^2 $ given as follows
	\begin{equation}		\label{eq2.1}
		\begin{cases}
			\begin{aligned}
				&\partial _{t}u_1(t,x) -d_1\,\mathrm{div }\bigl(u_1(t,x)\nabla P(t,x)\bigr)=
				u_1(t,x)h_1((u_1,u_2)(t,x))\\
				&\partial _{t}u_2(t,x) -d_2\,\mathrm{div }\bigl(u_2(t,x)\nabla P(t,x)\bigr)=
				u_2(t,x)h_2((u_1,u_2)(t,x))\\
				&\big(I-\chi\Delta \big)P(t,x) = u_1(t,x)+u_2(t,x)
			\end{aligned}
			&\text{in } [0,T]\times\Omega\\
			\nabla P(t,x)\cdot \mathbf{\nu}(x) =0\,& \text{on }  [0,T)\times\partial\Omega, 
		\end{cases}
	\end{equation}
	where $ \mathbf{\nu} $ is the outward normal vector, $ d_i $ is the dispersion coefficient, $ \chi $ is the sensing coefficient. Recall the function $ h_i $ is of form
	\begin{equation*}
		h_i(u_1,u_2)= b_i-\delta_i-\sum_{j=1}^{2} a_{ij} u_j,\quad i=1,2. 
	\end{equation*}
	System \eqref{eq2.1} is supplemented with initial distribution
	\begin{equation}\label{eq2.2}
		\textbf{u}_0(\cdot):=\left(
		u_1(0,\cdot),
		u_2(0,\cdot)\right)\in C^1(\overline{\Omega})^2.  
	\end{equation}

	\subsubsection{Segregation property}\label{Sect.segregation}
	From the mono-layer cell populations co-culture experiments, we can see that once the two cell populations confront each other, they will stop growing, thus, forming the separated islets. We can prove that our model \eqref{eq2.1} preserves such segregation property. 
	\begin{theorem}\label{THM3.1} Suppose $ \mathbf{u}=(u_1,u_2)(t,x) $ is the solution of \eqref{eq2.1}-\eqref{eq2.2} and assume $ d_1=d_2=d $  in \eqref{eq2.1}.
		Then  for any initial distribution with $ u_1(0,x)u_2(0,x)=0 $ for all $ x\in \Omega$, we have $  u_1(t,x)u_2(t,x)=0 $ for any $ t>0 $ and $ x\in \Omega$.
	\end{theorem}
	\begin{proof}
		We argue by contradiction, assume there exist $ t^*>0, x^*\in \Omega$ such that \[  u_1(t^*,x^*)u_2(t^*,x^*)>0. \] 
		Suppose the characteristic flow satisfies the following equation
		\begin{equation*}
			\begin{cases}
				\frac{\partial\phantom{t}}{\partial t}\Pi(t,s;x)=-d\,\nabla P(t,\Pi(t,s;x))\\
				\Pi(s,s;x)=x\in \Omega.
			\end{cases}
		\end{equation*}		
		Since $ x\to \Pi(t,s;x) $ is invertible from $ \Omega $ to itself, there exists some $ x_0\in \Omega $ such that $ \Pi(t^*,0;x_0)=x^* $. Then  for any $ i=1,2, $
		we have
		\begin{equation}\label{eq2.12}
			\quad u_i(t^*,\Pi(t^*,0;x_0))
			= u_i\left(0,x_0\right)e^{\int_{0}^{t^*}h_i\left((u_1,u_2)(l,\Pi(l,0;x_0))\right)+\frac{d}{\chi} \left(P(l,\Pi(l,0;x_0))-(u_1+u_2)(l,\Pi(l,0;x_0)) \right)\dd l}>0,
		\end{equation}
		which implies 
		\[ u_i\left(0, x_0\right)>0,\quad i=1,2. \]
		This is a contradiction.
	\end{proof}
	For the one dimensional case $ N=1 $, suppose $ u_1,u_2 $ are solutions to \eqref{eq2.1}-\eqref{eq2.2}, we give an illustration (see Figure \ref{Figure3}) of the segregation for the solutions integrated along the characteristics $ u_i(t,\Pi(t,0;x)) $ for $ i=1,2 $. In fact, if there exists for some $ x_0 $ such that $ u_i(0,x_0)=0 $ for $ i=1,2.$ Then from Equation \eqref{eq2.12} we obtain 
	\[ u_1(t,\Pi(t,0;x_0))= u_2(t,\Pi(t,0;x_0))=0,\,\forall t>0. \]
	Therefore, the characteristics $t\mapsto \Pi(t,0;x_0) $ forms a segregation barrier for the two cell populations.
	
	\begin{figure}[H]
		\begin{center}
			\includegraphics[width=0.45\textwidth]{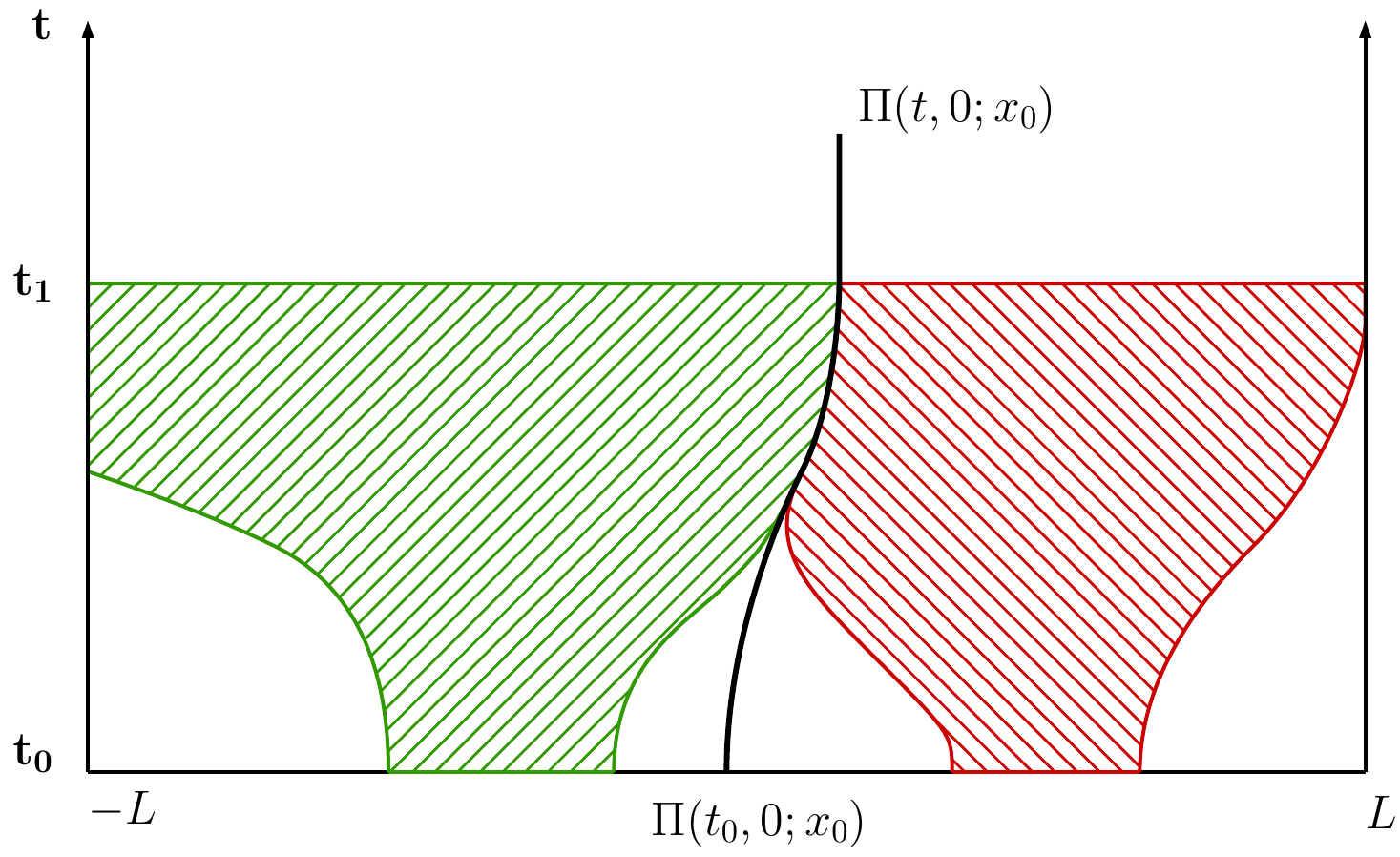}\,\,\includegraphics[width=0.45\textwidth]{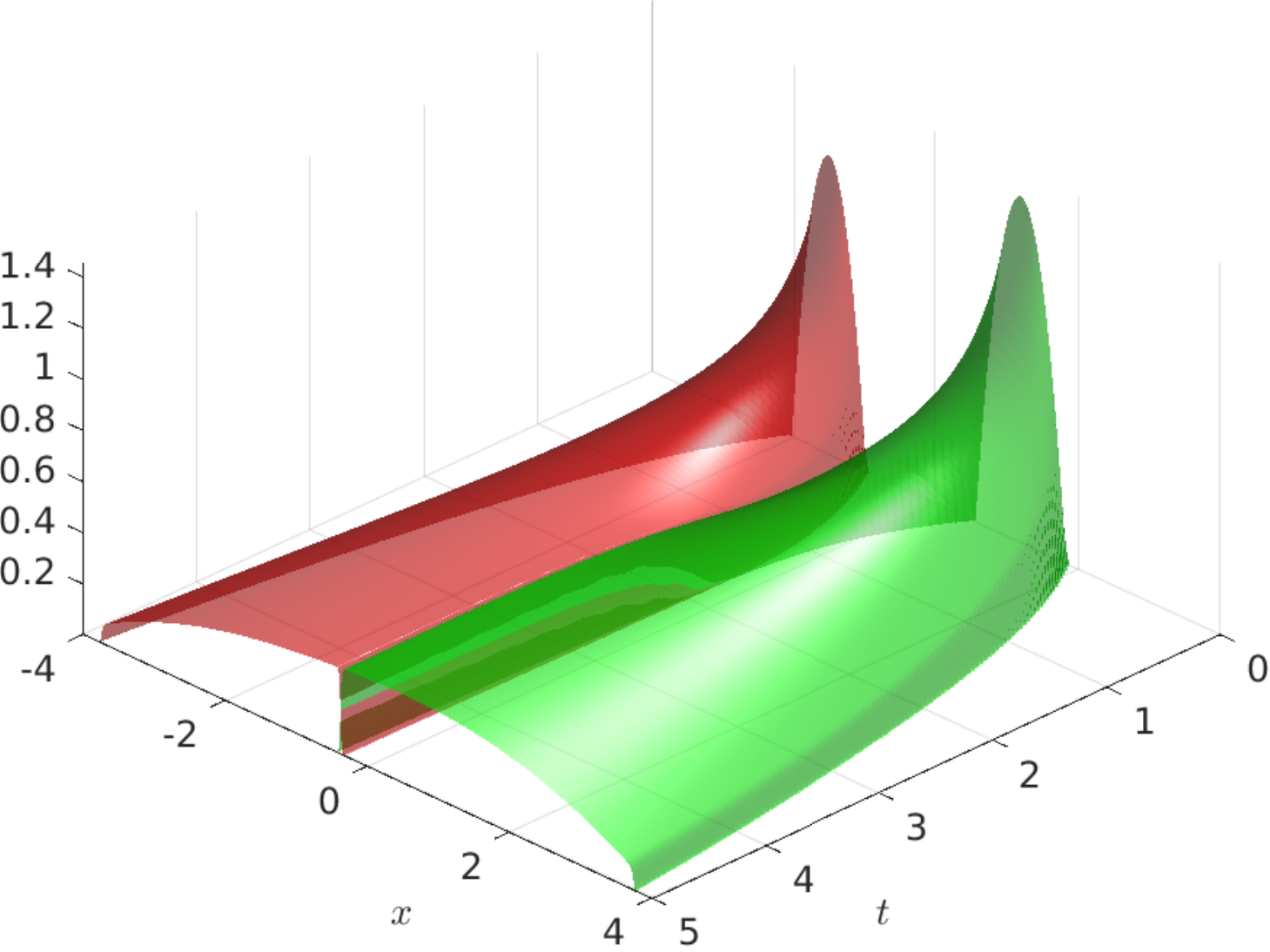}
		\end{center}
		\caption{\textit{In this figure we illustrate the notion of segregation with a one dimensional bounded domain. Figure (a) shows the characteristic $t\mapsto \Pi(t,0;x_0) $ forms a segregation ``wall". Figure (b) shows the temporal-spatial evolution of the two species.}}
		\label{Figure3}
	\end{figure}
	\subsubsection{Conservation law on a volume}
	If we assume that $ d_1=d_2=d $ in system \eqref{eq2.1}, we have the following similar conservation law for  two species case. Suppose volume $A\subset \Omega $ and each $0\leq s\leq t$: 
	\begin{equation*}
		\int_{\Pi(t,s;A)}u_i(t,x)\dd x=\int_A \exp\left[\int_s^t
		h_i\left((u_1,u_2)\left(l,\Pi(l,s;z)\right)\right)\dd  l\right]u_i(s,z)\dd  z,i=1,2.
	\end{equation*}
	Therefore, if we have $ h_i=0 $ for any  $0\leq s\leq t$
	\begin{equation*}
		\int_{\Pi(t,s;A)}u_i(t,x)\dd x=\int_A u_i(s,z)\dd  z,i=1,2.
	\end{equation*}
	This means the total cell number of the species $u_i$ is constant along the characteristics starting from the volume $ A $. 
	\section{Numerical simulations}\label{Sect.3}
	\subsection{Impact of the segregation on the competitive exclusion}
	We set $ U_i $ to be the total number at time $t= 0 $ 
	\begin{equation}\label{eq:Ui}
	 U_i=\int_{\Omega} u_i(0,x)\dd x,\quad i=1,2.
	\end{equation}
	We give the parameter values used in the simulations and their interpretations in Table \ref{TABLE1}. The parameter fitting for the growth rate $ b_i $ and the intraspecific coefficients $ a_{ii} $ are detailed in Appendix \ref{App:para}.
	\begin{table}[H] \centering
		\begin{tabular}{cccccc}
			\doubleRule
			\textbf{Symbol}  &  \textbf{Interpretation} & \textbf{Value}  &\textbf{Unit} &  \textbf{Method} &  \textbf{Dimensionless value} \\
			\hline
			$ t $  
			& time
			& 1 
			& $ day $
			& -  
			& 1\\
			$ r $  
			& inner radius of the dish 
			& 2.62
			& $ cm $
			& \cite{Pasquier2011}
			& 1\\
			$  U_i $  
			& cell total number at $t=0$
			& $ 10^5 $
			& $- $
			& \cite{Pasquier2011}
			& 0.01 \\
			$ b_1 $  
			& growth rate of cell $ u_1 $ 
			& 0.6420 
			& $day^{-1} $
			& fitted 
			& 0.6420 \\
			$ b_2 $  
			& growth rate of cell $ u_2 $
			& 0.6359 
			& $day^{-1} $
			& fitted  
			& 0.6359\\
			$ a_{11} $  
			& intraspecific competition of $ u_1 $
			& $ 1.07\times 10^{-6} $ 
			& $ cm^2/day$
			& fitted 
			& $ 1.5588 $ \\
			$ a_{22} $  
			& intraspecific competition of $ u_2 $
			& $ 1.06\times 10^{-6}$
			& $ cm^2/day$
			& fitted  
			& $ 1.5415 $ \\	
			$ d_1 $  
			& dispersion coefficient of $ u_1 $
			& 13.73 
			& $ cm^4/day$
			& fitted 
			& 2  \\
			$ d_2 $  
			& dispersion coefficient of $ u_2 $
			& 13.73  
			& $ cm^4 /day$
			& fitted 
			& 2  \\
			$ \chi $  
			& sensing coefficient 
			& $ 6.86\times 10^{-2} $
			& $ cm^2$
			& fitted 
			& $ 0.01 $  \\		
			\doubleRule
		\end{tabular} 
		\caption{\textit{List of the model parameters, their interpretations, values and symbols. Here $ u_1 $ represents MCF-7 (sensitive cell) and $ u_2 $ represents MCF-7/Doxo (resistant cell). From \cite{Pasquier2011}, the surface of the dish is $21.5\, cm^2   $. Thus the inner radius of the dish $ r $ is calculated by $ r^2\pi =21.5\, cm^2 $. }}\label{TABLE1}
	\end{table}
	
	The goal of our simulations is to compare the various cases discussed in Proposition \ref{prop:ode} (ODE case) with our PDE model with segregation. As we will see in the numerical simulations, the model with spatial structure can present totally different results compared to the previous ODE model.  
	To that aim, we firstly consider the case where the drug (doxorubicine) concentration is low in the cell co-culture for MCF-7 and MCF-7/Doxo. The drug treatment causes  an additional mortality to the sensitive population MCF-7 represented by $ u_1 $ while no extra mortality to the resistant population MCF-7/Doxo represented by $ u_2 $  (MCF-7/Doxo is resistant to a small quantity of drug treatment see Table \ref{TABLE3} in Appendix \ref{App:para}).  
	
	Now since we consider the presence of the drug, the equilibrium \eqref{eq:equilibrium} in the ODE case  should be rewritten as
	\begin{equation}\label{eq:newequilibrium}
		\bar{P}_1 = \frac{b_1-\delta_1}{a_{11}},\quad \bar{P}_2 = \frac{b_2-\delta_2}{a_{22}}.
	\end{equation}
	Moreover we assume  the drug concentration is low such that  $ b_1-\delta_1>0 $ and $ \delta_2=0 $, therefore we have
	\[ \bar{P}_1<\bar{P}_2. \]
	The case when $ \bar{P}_1>\bar{P}_2 $ is similar and will be discussed in the end of this section.\medskip\\
	Case (i): $ a_{12}/a_{11}  < \bar{P}_1/\bar{P}_2,\, a_{21}/a_{22}< \bar{P}_2/\bar{P}_1.$ By using \eqref{eq:newequilibrium}, the condition in Case (i) can be interpreted by
	\[  \frac{a_{12}}{a_{22}} < \frac{b_1-\delta_1}{b_2-\delta_2},\quad \frac{a_{21}}{a_{11}} < \frac{b_2-\delta_2}{b_1-\delta_1}. \]
	Since we have $ b_1-\delta_1>0 $ and $ \delta_2=0 $, if the coefficients $ a_{12} $ and $ a_{21} $ are small, then Case (i)  holds.	We give a possible set of parameters satisfying Case (i) :
	\begin{equation}\label{eq3.2}
		\delta_1 =0.4,\,\delta_2=0,\,a_{12}=0.2,\,a_{21} =1.
	\end{equation}
	We assume that for each species $ u_i $, the initial distribution follows the uniform distribution on a disk with $ 20 $ initial cell clusters (represented by the red/green dots in Figure \ref{Figure_case1} (a)). The initial total cell number is $U_i= 0.01 $ in \eqref{eq:Ui} for each species  and we assume each cluster contains the same quantity of cells.
	We present its numerical simulation in Figure \ref{Figure_case1} from day $ 0 $ to day $ 6 $. We also plot the relative cell numbers in Figure \ref{Figure_case1} (f) where we define the relative cell number for species \textit{i} as 
	\[  \frac{U_i(t)}{U_1(t)+U_2(t)},\quad \text{ where }  U_i(t):= \int_{\Omega}u_i(t,x)\dd x,\, i=1,2. \]
	\begin{figure}[H]
		\begin{center}
			\includegraphics[width=0.33\textwidth]{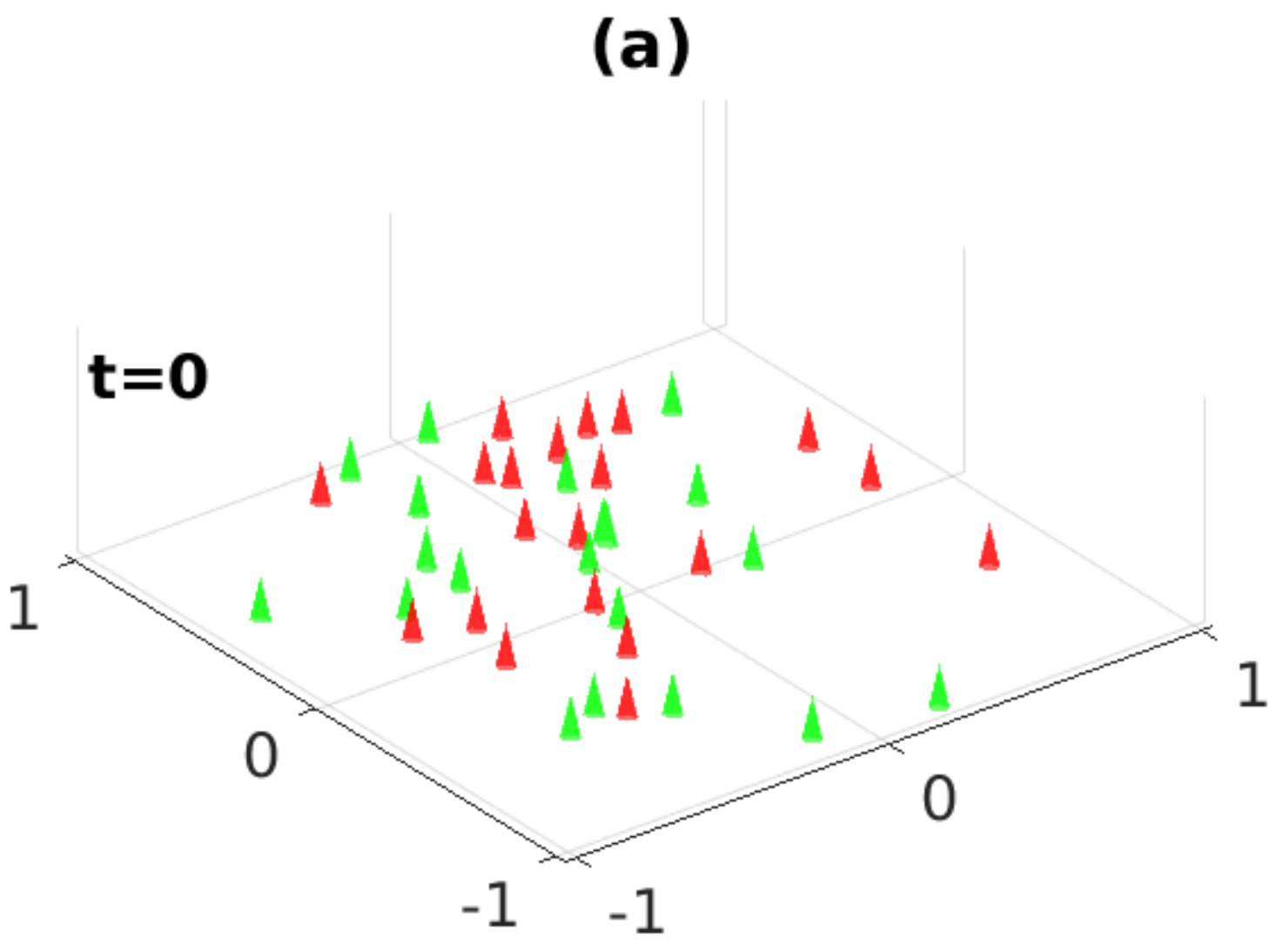}\includegraphics[width=0.33\textwidth]{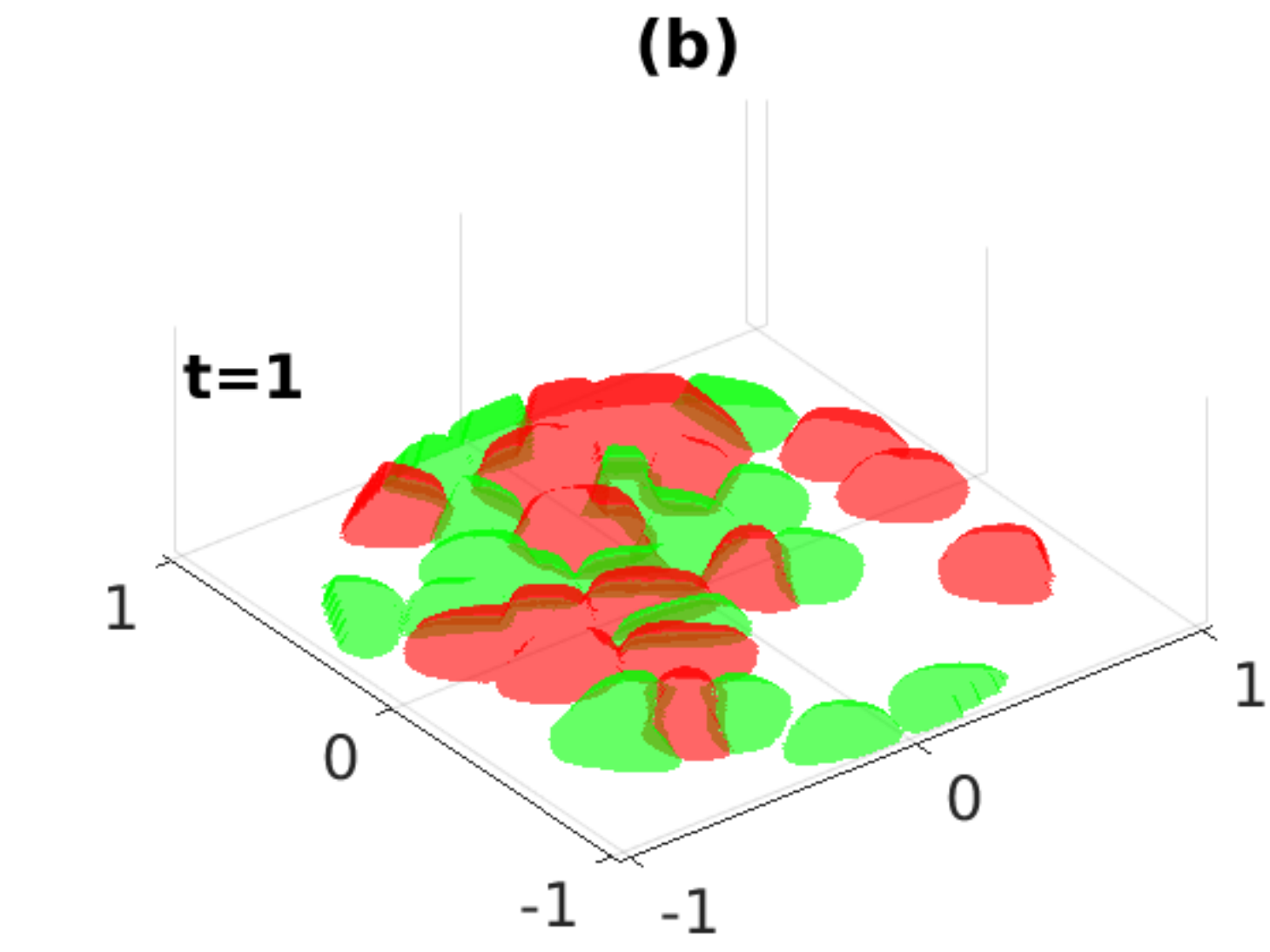}\includegraphics[width=0.33\textwidth]{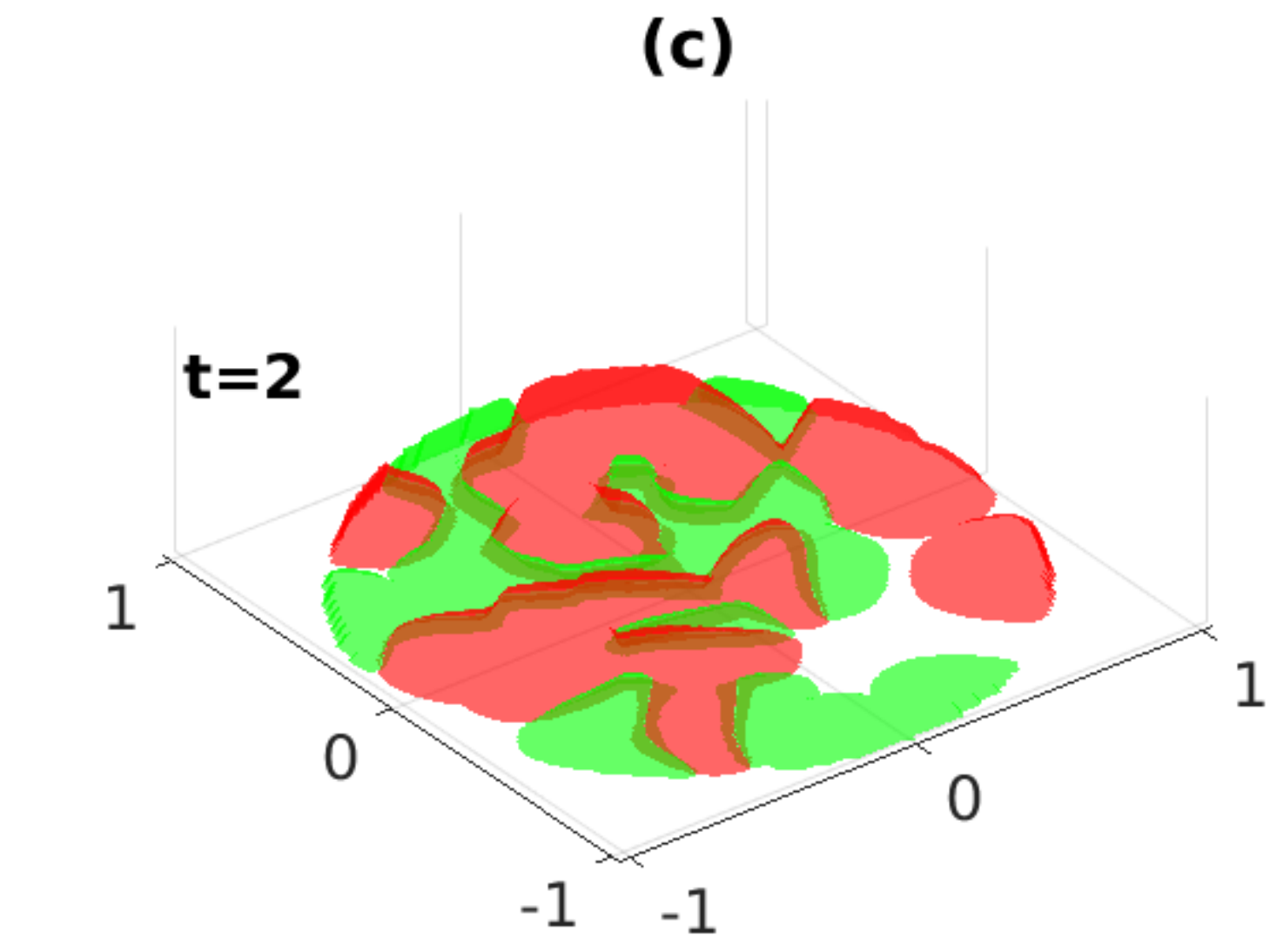}
			\includegraphics[width=0.33\textwidth]{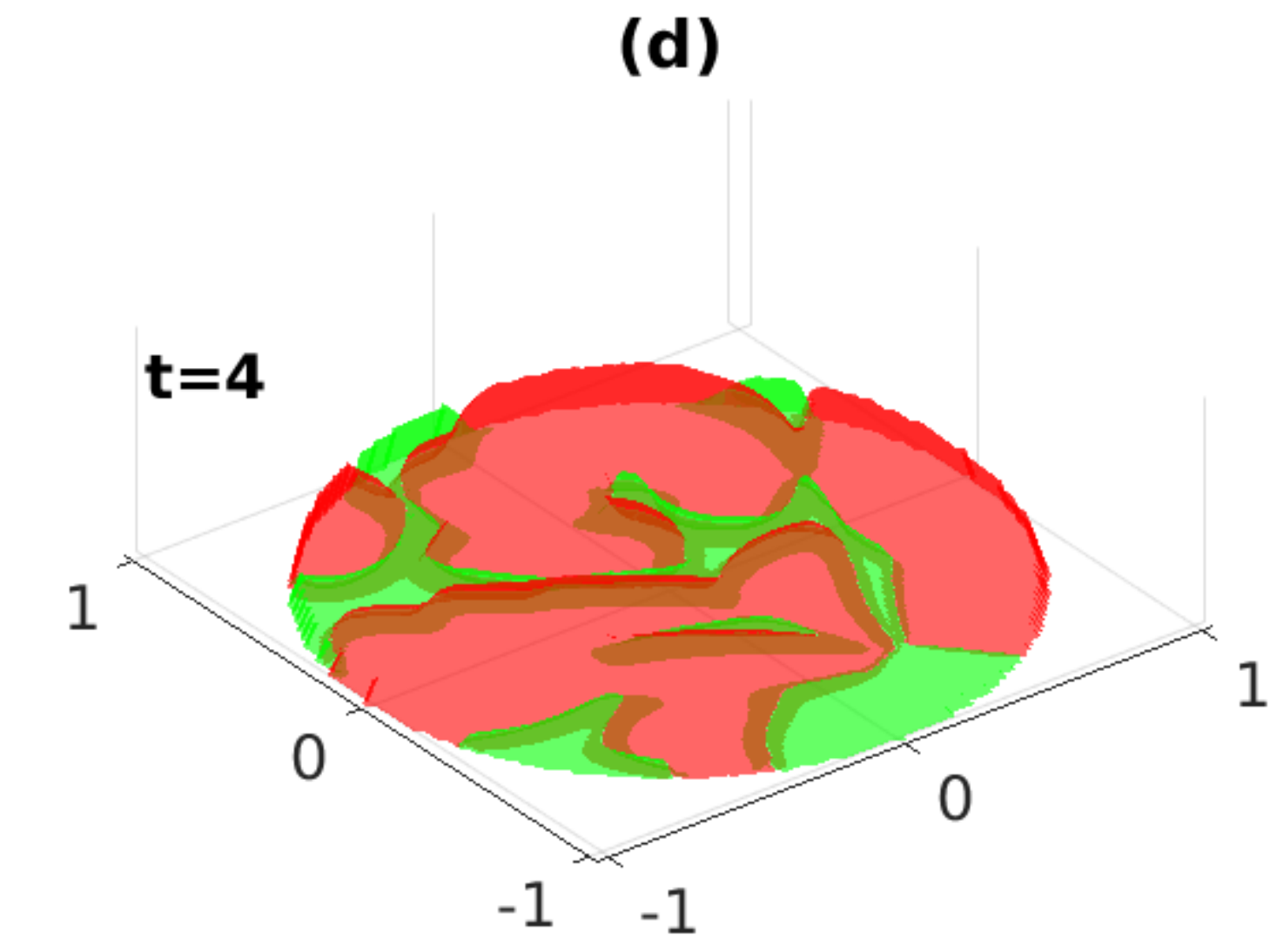}\includegraphics[width=0.33\textwidth]{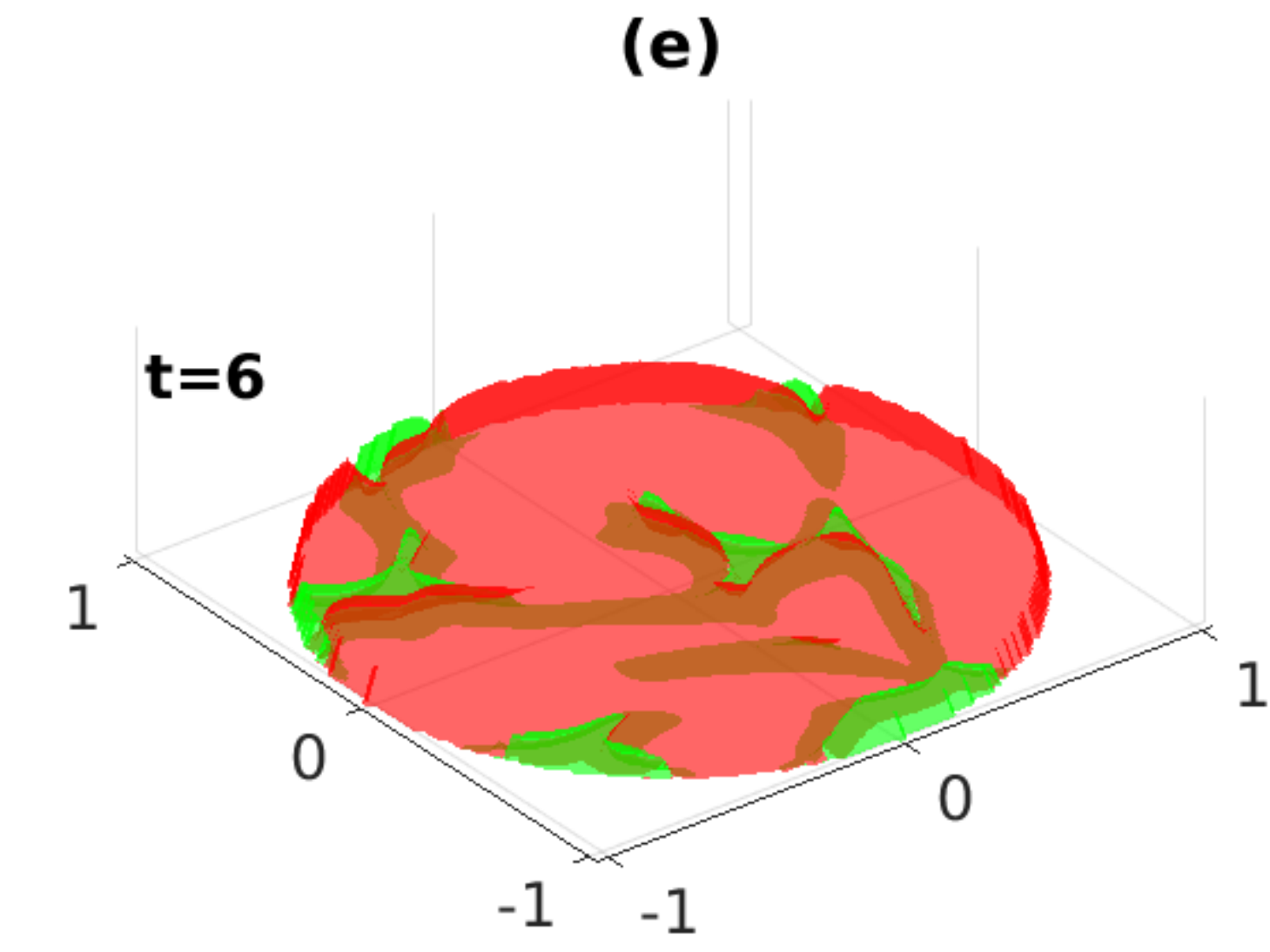}\includegraphics[width=0.33\textwidth]{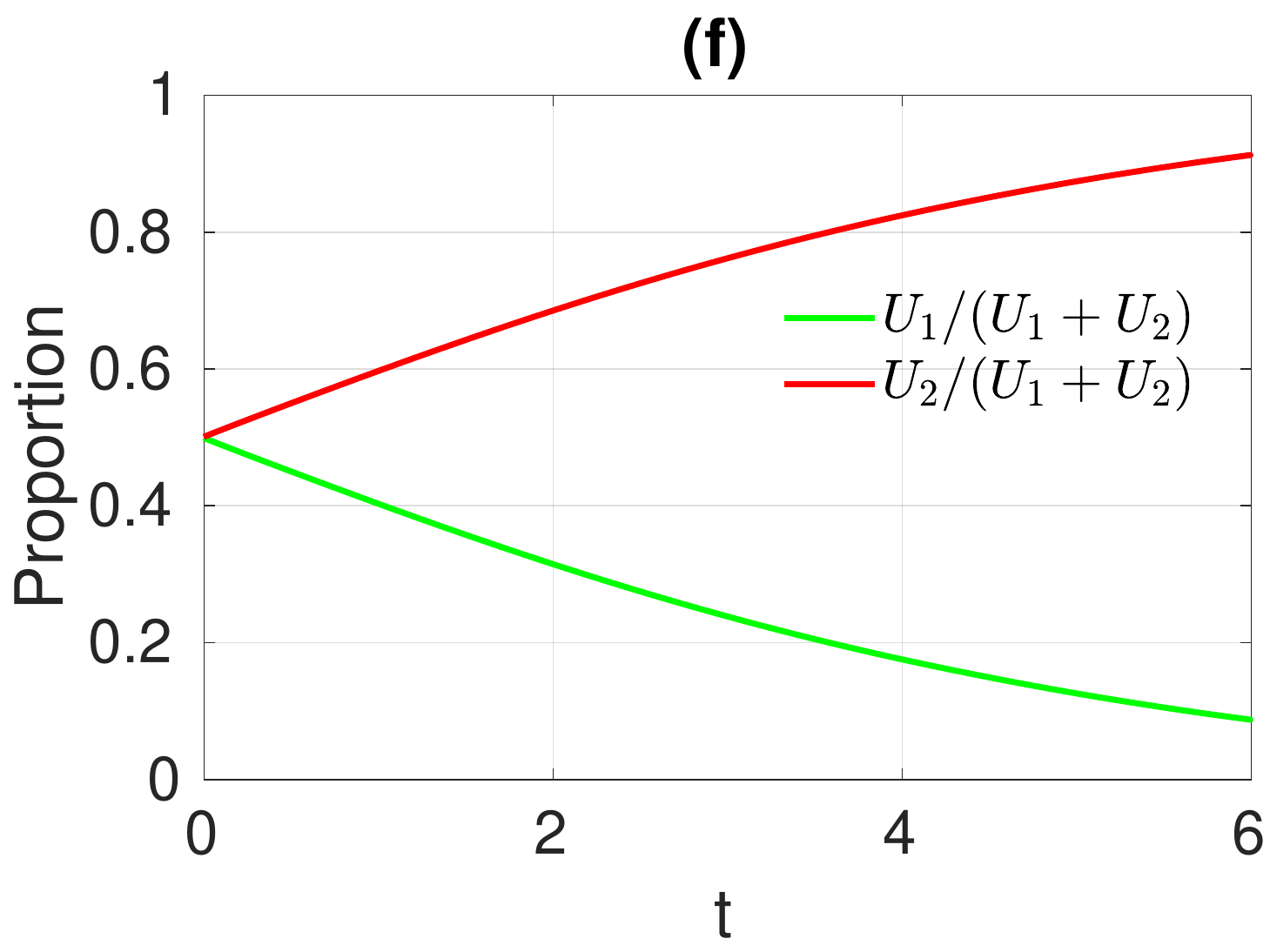}
		\end{center}
		\caption{\textit{Spatial-temporal evolution of the two species $ u_1 $ and $ u_2 $ and its relative proportion. Figures (a)-(e) correspond to the evolution of cell growth form day $ 0 $ to day $ 6 $ and Figure (f) is the relative proportion plot from day $ 0 $ to day $ 6 $. We fix the parameters $\delta_1 =0.4,\,\delta_2=0,\,a_{12}=0.2,\,a_{21} =1$ in \eqref{eq3.2}. The initial distribution follows the uniform distribution on a disk with $ 20 $ initial cell clusters. The initial total cell number is $U_1=U_2= 0.01 $ for each species and cells are equally distributed in each cluster. The other parameters are given in Table \ref{TABLE1}.}}
		\label{Figure_case1}
	\end{figure}
	In Case (i) of the ODE system \eqref{eq:odemodel}, Proposition \ref{prop:ode} shows that the two species can coexist with the equilibrium
	\[ \bar{E}^* :=\left(\dfrac{a_{22} (b_1-\delta_1)-a_{12} (b_2-\delta_2)}{a_{11} a_{22}-a_{12} a_{21}}, \dfrac{a_{21} (b_1-\delta_1)-a_{11} (b_2-\delta_2)}{a_{12} a_{21}-a_{11} a_{22}}\right)\approx (0.11,0.34).  \]
	However, as shown in Figure \ref{Figure_case1}, we can see the population density $ u_1 $ tends to 0 and $ u_2 $ tends to 1.  
	Next, we consider the Cases (ii)-(iv) in Proposition \ref{prop:ode} by choosing the parameters in each case as follows.
	\begin{table}[H]\centering
		\begin{tabular}{cccccc}
			\doubleRule
			\textbf{Parameters} & $ \delta_1 $  & $ \delta_2 $ & $ a_{12} $ & $  a_{21}  $ & Relations \\ \hline
			\textbf{Case (ii)}  & 0.4 & 0 & 1  & 1 & $ a_{12}/a_{11}  > \bar{P}_1/\bar{P}_2,\, a_{21}/a_{22}  < \bar{P}_2/\bar{P}_1.$\\ 
			\textbf{Case (iii)} & 0.4 & 0   & 0.2  & 5 & $ a_{12}/a_{11}  < \bar{P}_1/\bar{P}_2,\, a_{21}/a_{22}  > \bar{P}_2/\bar{P}_1.$\\
			\textbf{Case (iv)}  & 0.4 & 0   & 1  & 5 & $ a_{12}/a_{11}  > \bar{P}_1/\bar{P}_2,\, a_{21}/a_{22}  > \bar{P}_2/\bar{P}_1.$\\ 
			\doubleRule
		\end{tabular}
		\caption{\textit{List of the parameters used in the simulations for Cases (ii)-(iv). Other parameters are given in Table \ref{TABLE1}.}}\label{TABLE2}
	\end{table}
	\begin{figure}[H]
		\begin{center}
			\includegraphics[width=0.33\textwidth]{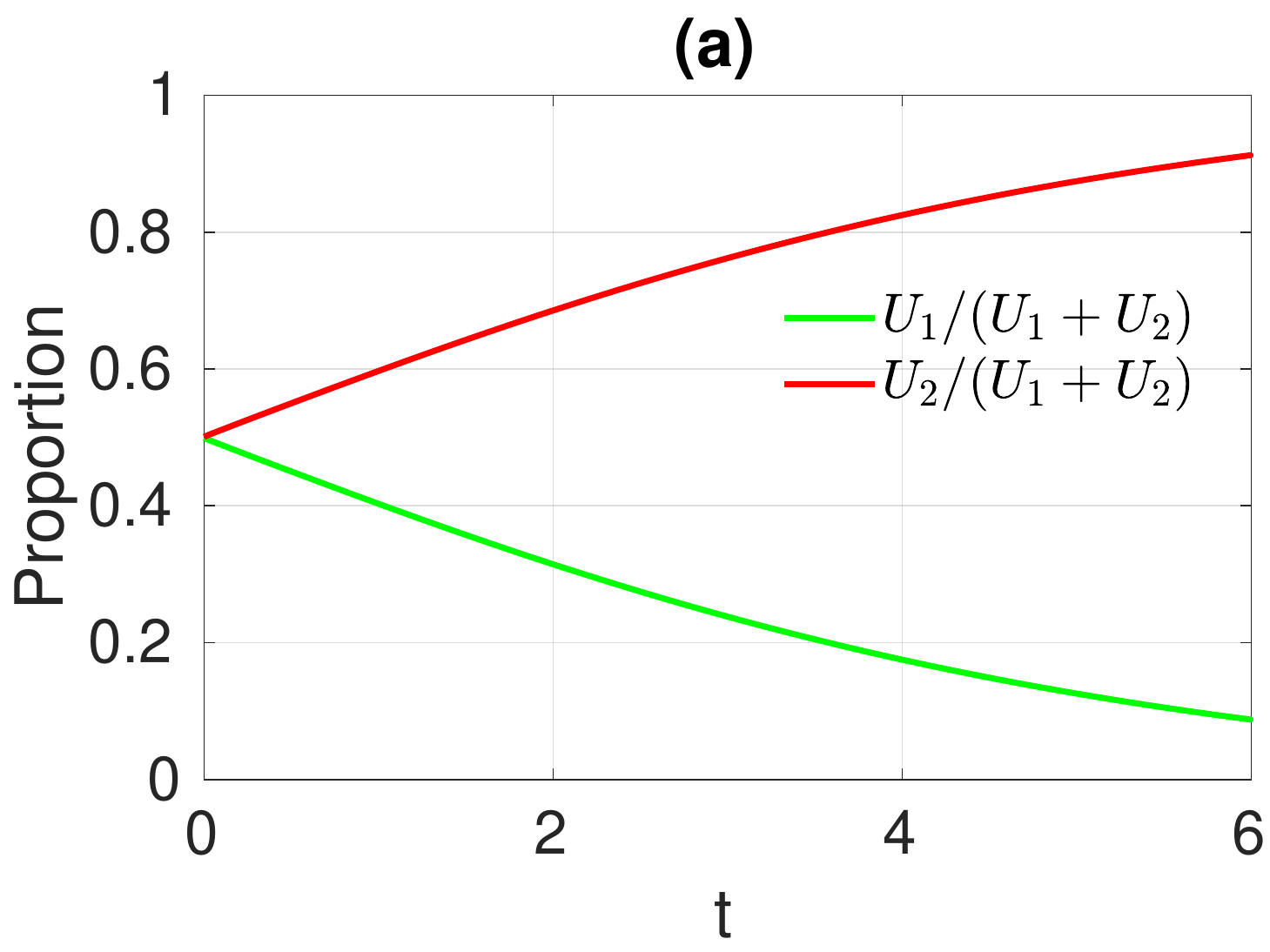}\includegraphics[width=0.33\textwidth]{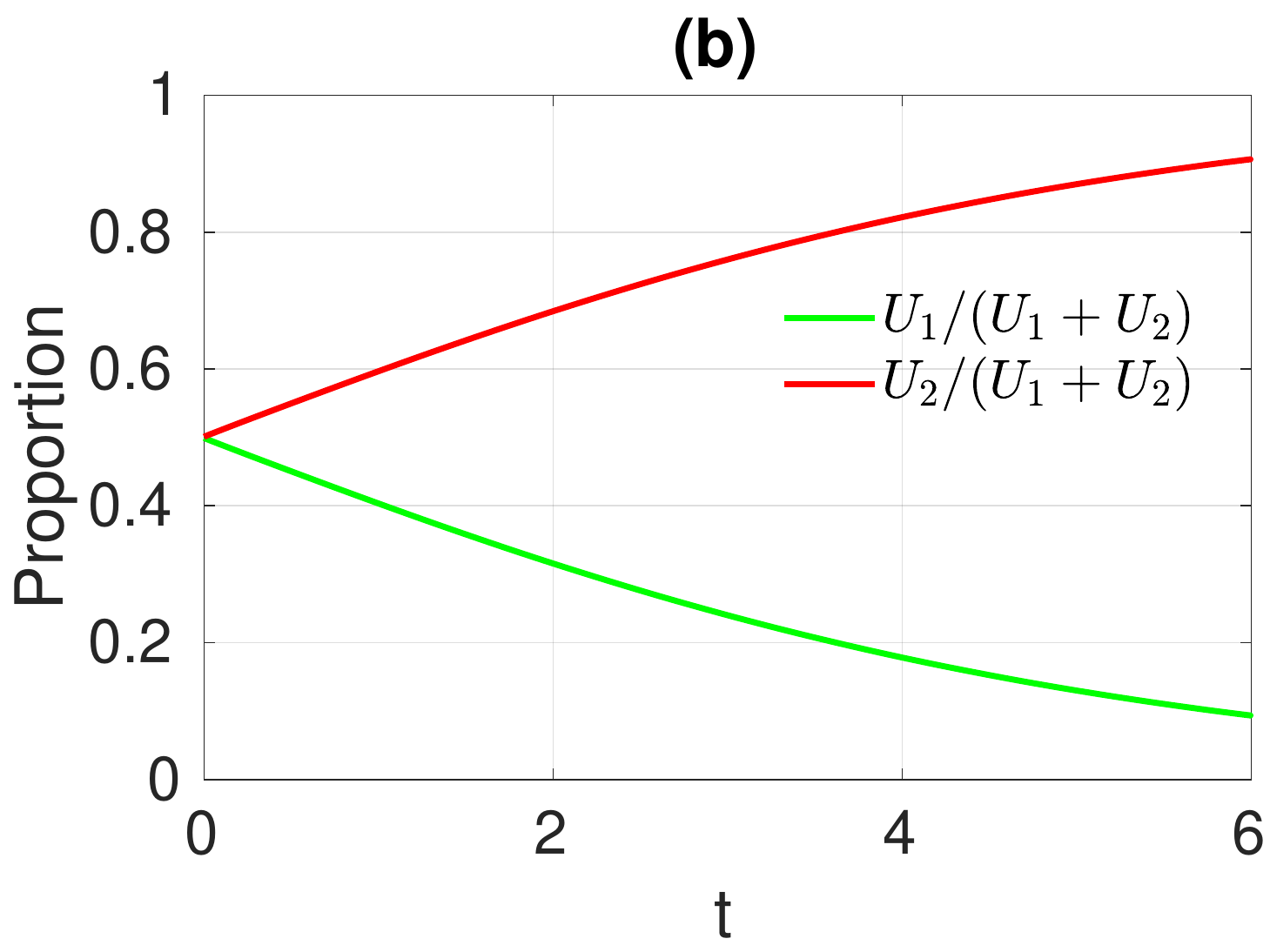}\includegraphics[width=0.33\textwidth]{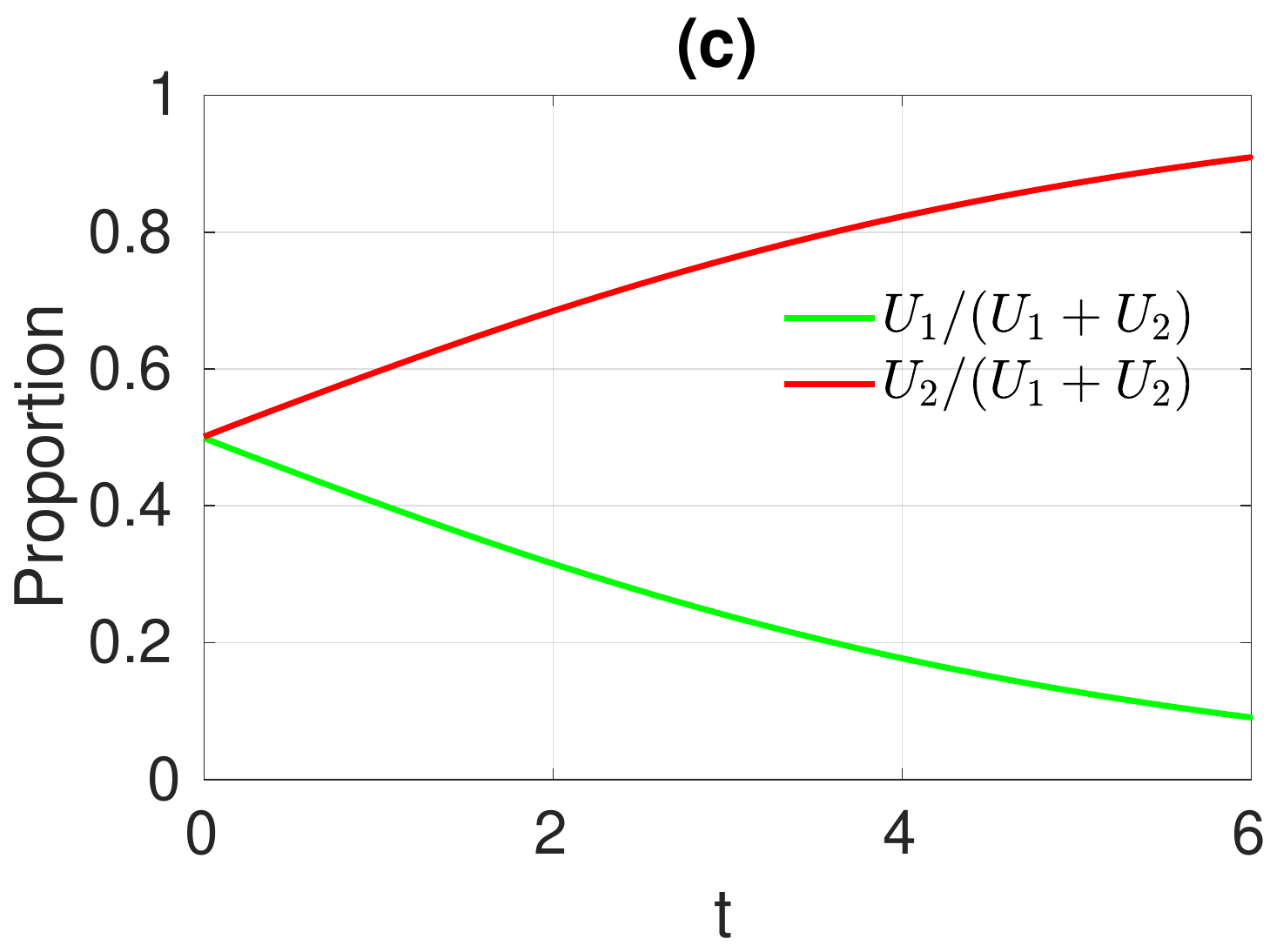}
		\end{center}
		\caption{\textit{Evolution of the relative cell numbers for two species $ u_1 $ and $ u_2 $. Figure (a)-(c) correspond to the parameter values chosen as in Table \ref{TABLE2} for Cases (ii)-(iv) and other parameters are given in Table \ref{TABLE1}.}}
		\label{Figure_case234}
	\end{figure}
	With the simulations in Figure \ref{Figure_case1}, Figure \ref{Figure_case234} and the results in Proposition \ref{prop:ode}, by setting
	\[ \bar{E}_1 =(\bar{P}_1,0)\quad \bar{E}_2 =(0,\bar{P}_2)\quad \bar{E}^*=\left(\dfrac{a_{22} (b_1-\delta_1)-a_{12} (b_2-\delta_2)}{a_{11} a_{22}-a_{12} a_{21}}, \dfrac{a_{21} (b_1-\delta_1)-a_{11} (b_2-\delta_2)}{a_{12} a_{21}-a_{11} a_{22}}\right),    \]
	we can compare the stability between ODE model \eqref{eq:odemodel} and PDE model \eqref{eq2.1} under four different cases. 
	\begin{table}[H]\centering
		\begin{tabular}{cccccc}
			\doubleRule
			$ \bar{P}_1<\bar{P}_2 $ & Case (i)  & Case (ii) & Case (iii) & Case (iv) \\ \hline
			\textbf{Global attractor in ODE}  & Coexistence $ \bar{E}^* $ & $ \bar{E}_2 $ & $ \bar{E}_1 $  & Region dependent \\ 
			\textbf{Stable steady state in PDE} & $ \bar{E}_2 $ & $ \bar{E}_2 $   & $ \bar{E}_2 $  & $ \bar{E}_2 $  \\ 
			\doubleRule
		\end{tabular}
		\caption{\textit{A summary for the stability to four cases (i)-(iv) under ODE model and PDE model with segregation.}}\label{TABLE4}
	\end{table}
	The numerical simulation strongly indicates that the stable steady states only depend on the relation between $ \bar{P}_1 $ and $ \bar{P}_2 $. If $ \bar{P}_1<\bar{P}_2 $ (resp. $ \bar{P}_1>\bar{P}_2 $), the population $ u_2 $ (resp. $ u_1 $) will dominate and the other species will die out. We also did the four cases when $ \bar{P}_1>\bar{P}_2 $, the results showed that $ \bar{E}_1 $ is the only stable steady state, which verifies our conjecture. Since the results are similar we omit the numerical simulations. \medskip

	One can notice that unlike the ODE system \eqref{eq:odemodel}, the segregation property for the PDE model implies that it is impossible for the two species to coexist at a same position $ x\in \Omega $. Moreover, through the numerical simulations we observed that the PDE model \eqref{eq2.1} always undergoes a competitive exclusion principle, unless the equilibrium $ \bar{P}_1=\bar{P}_2 $ in \eqref{eq:newequilibrium}. 
	\subsection{Impact of the initial distribution on the population ratio}\label{Sect.3.2}
	In the previous section, we considered the competitive exclusion principle for the two species.  By studying the relative proportions of $u_1$ and $u_2$, we presented the relation of the interspecific competition in our numerical simulation. Moreover, we can discover in Figure \ref{Figure_case1} (f) and in Figure \ref{Figure_case234} (a)-(c) that the increase of the proportion of the dominant population  $u_2$ (red curve) is varying with time.  It is evident to see from day 0 to day 2 the increase of the dominant population $u_2$ is faster than the increase from day 4 to day 6. If we further study the spatial-temporal evolution of the cell co-culture presented in Figure \ref{Figure_case1} (a)-(e), we can observe that from day 0 to day 2 the competition between the two groups is mainly expressed in the competition for space resources. However, from  day 4 to day 6, when the surface of the dish is almost fully occupied by cells, the reaction term $u_i h_i(u_1,u_2) $ in the equation begins to play a major role influencing the change in the number of cells.  In order to explore the major factors in cell competition, we consider the impact on the initial distribution. We will mainly focus on two factors, namely the  initial cell total number and the law of initial distribution, which might influence the proportions for $ u_1 $ and $ u_2 $. To that aim, we set the following parameters 
	\begin{equation}\label{eq:para_Table_1plus}
		\delta_1 =0.15,\,\delta_2=0,\,a_{12}=0,\,a_{21} =0.
	\end{equation}
	and the other parameters are given in Table \ref{TABLE1}.

	\subsubsection{Dependency on the initial total cell number}	
	In cell culture, the initial number of cell cluster is an important factor. Bailey et al. \cite{Bailey2018} study the sphere-forming efficiency of MCF-7 human breast cancer cell by comparing the cell culture with different initial numbers of the cell cluster.  Here we consider the impact of the initial cluster number on the final proportion of species. To that aim, we assume the initial distribution follows the uniform distribution on a disk. 
	We consider two sets of initial condition, that is
	\begin{equation}\label{eq:para_cluster}
		\begin{aligned}
			&U_1=U_2=0.005,\quad N_{u_1}=N_{u_2}=10,\\
			&U_1=U_2=0.1, \quad N_{u_1}=N_{u_2}=200, 
		\end{aligned}
	\end{equation}
	where $U_1$ and $U_2$ are defined in \eqref{eq:Ui} and  $N_{u_1}$ (respectively $N_{u_2}$) is the initial number of cell clusters of species $u_1$ (respectively species $u_2$). 
	
	The above initial conditions correspond to different types of seeding in the experiment, namely cells are sparsely seeded or densely seeded. We assume the total cell number is proportional to the initial number of cell cluster, meaning the dilution procedure adopted in the experiment is the same, thus the number of cells in each cell cluster is a constant.   
	
	In Figure \ref{FIG6.1}, we first give a numerical simulation for the cell growth with sparse seeding.
	\begin{figure}[H]
		\begin{center}
			\includegraphics[width=0.33\textwidth]{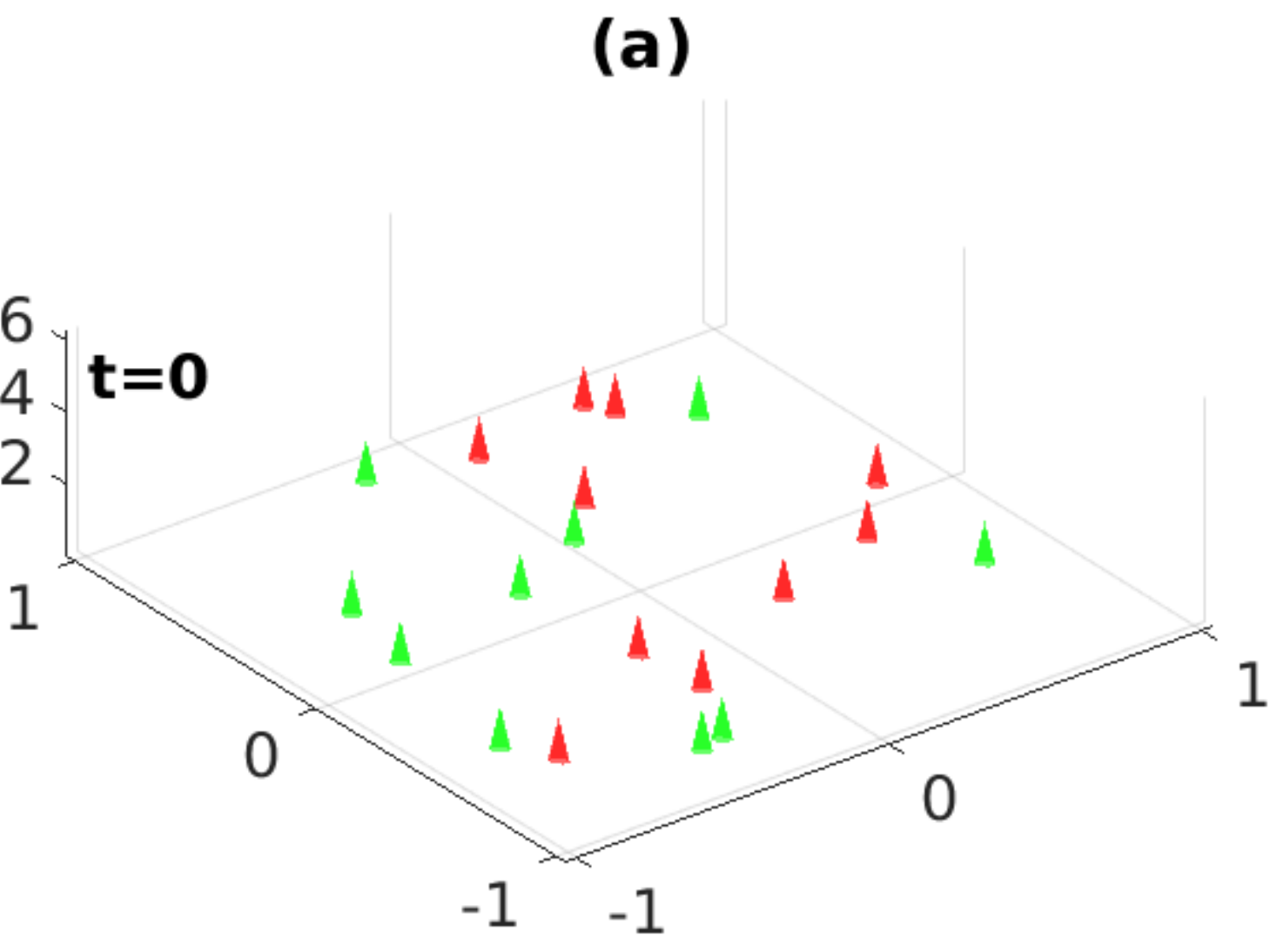}\includegraphics[width=0.33\textwidth]{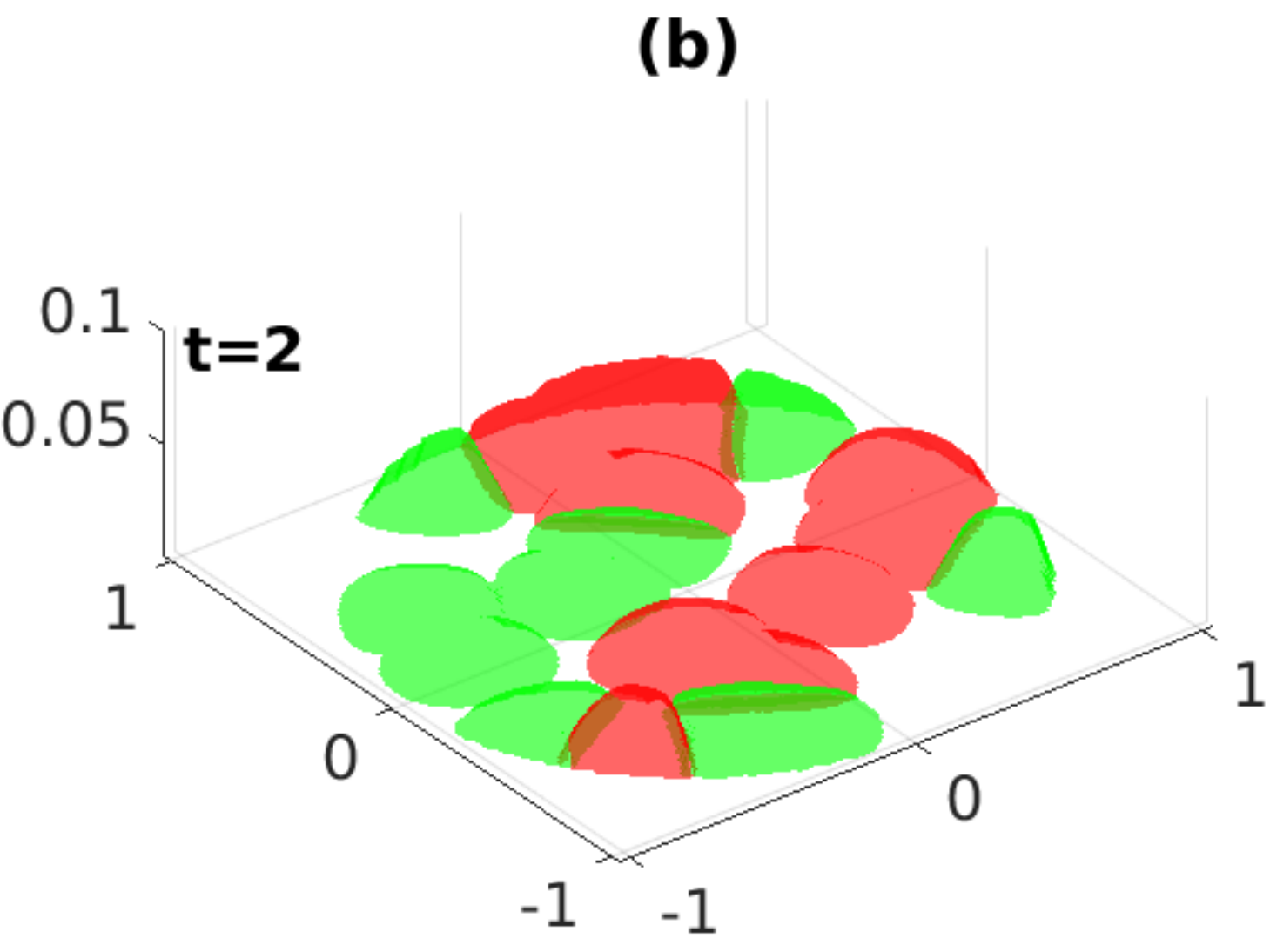}\includegraphics[width=0.33\textwidth]{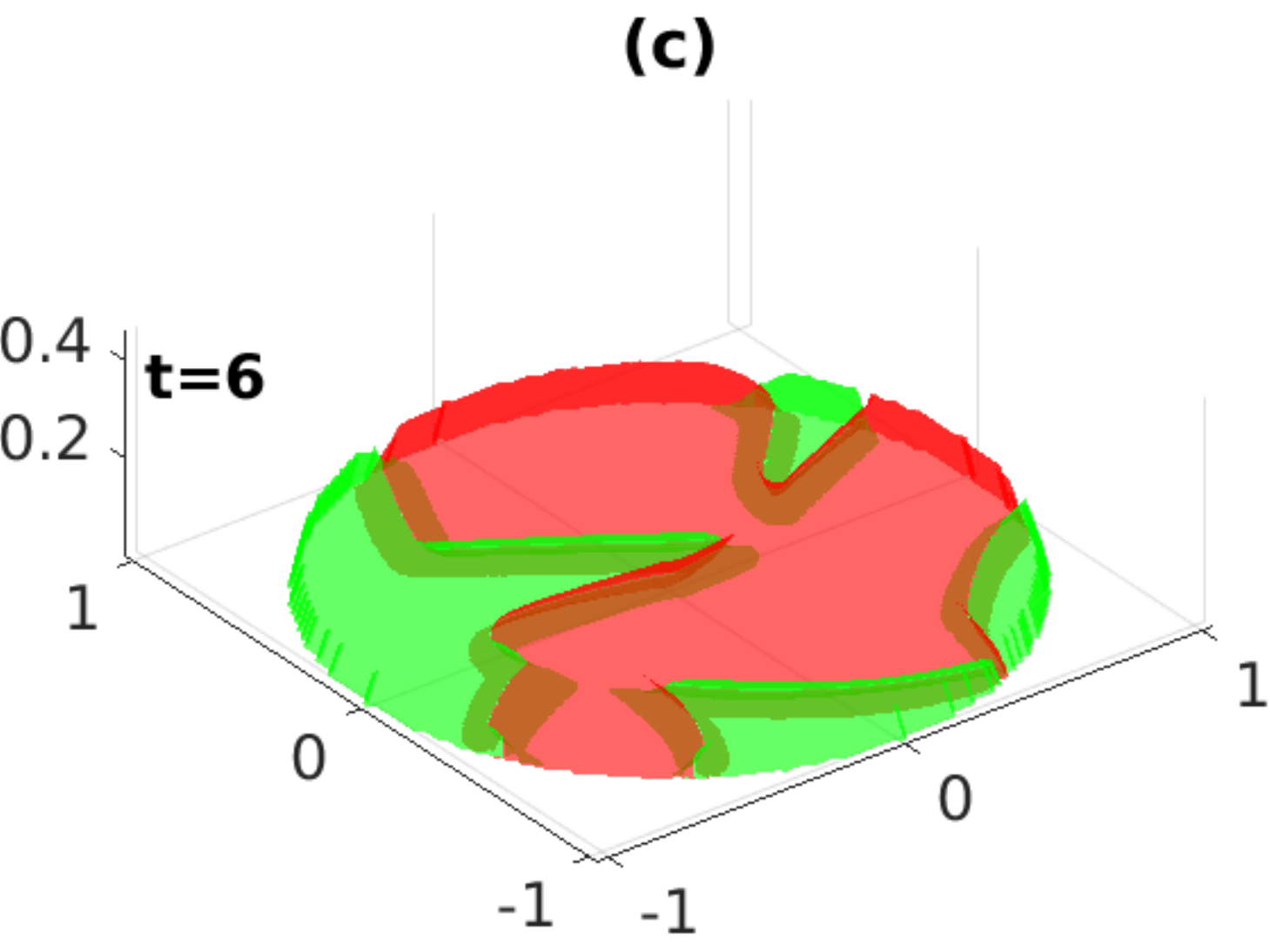}
		\end{center}
		\caption{\textit{Cell co-culture for species $ u_1 $ and $ u_2 $ over 6 days. We plot the case where cells are sparsely seeded, i.e., $U_1=U_2=0.005,\, N_{u_1}=N_{u_2}=10$ for day $ 0,2 $ and day $ 6 $. We set parameters as $\delta_1 =0.15,\,\delta_2=0,\,a_{12}=0,\,a_{21} =0$ in \eqref{eq:para_Table_1plus}. Other parameters are given in Table \ref{TABLE1}. }}
		\label{FIG6.1}
	\end{figure}
	In Figure \ref{FIG6.2}, we present the simulation under the dense seeding condition, tracking from day 0 to day 6.
	\begin{figure}[H]
		\begin{center}
			\includegraphics[width=0.33\textwidth]{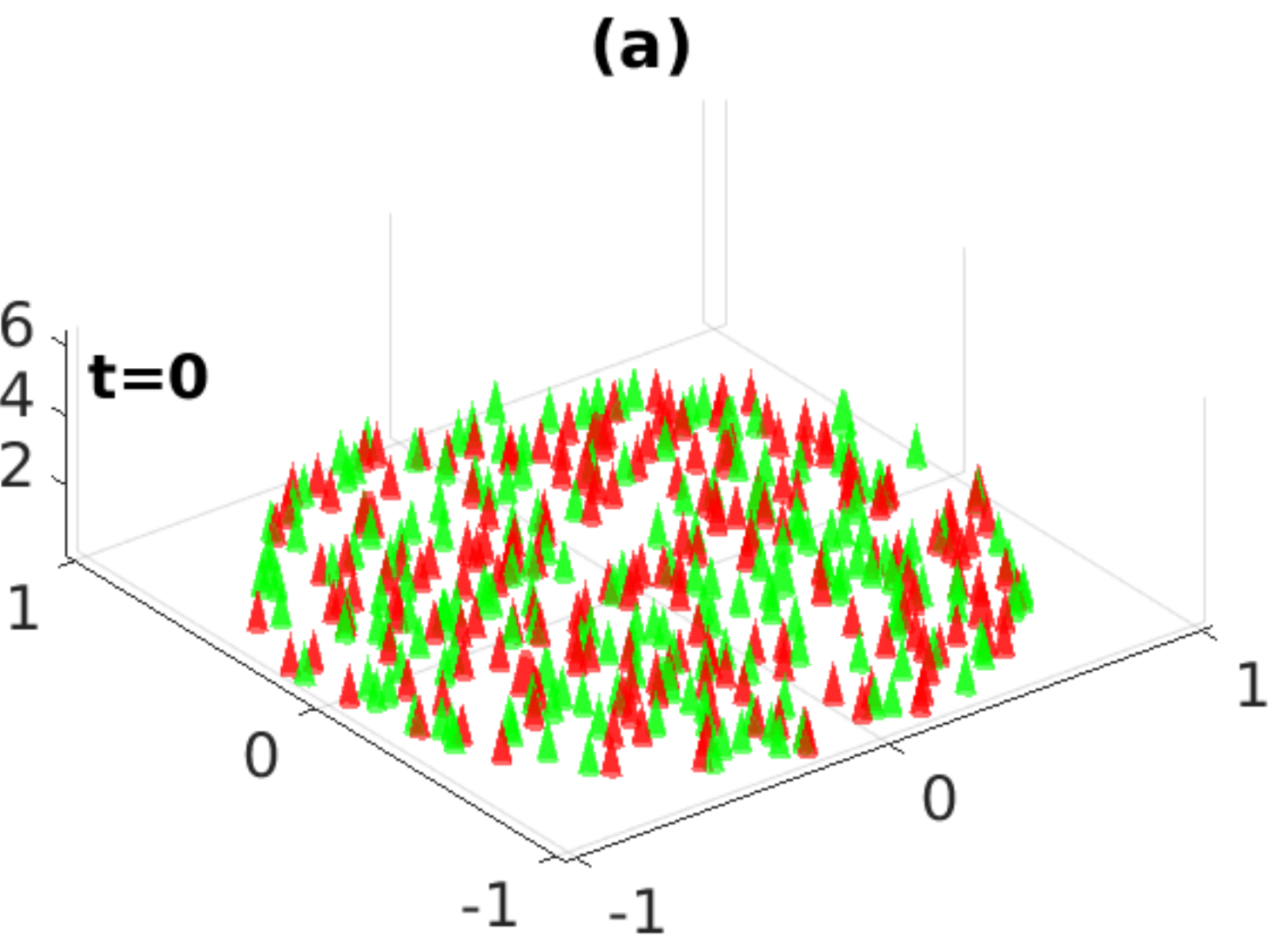}\includegraphics[width=0.33\textwidth]{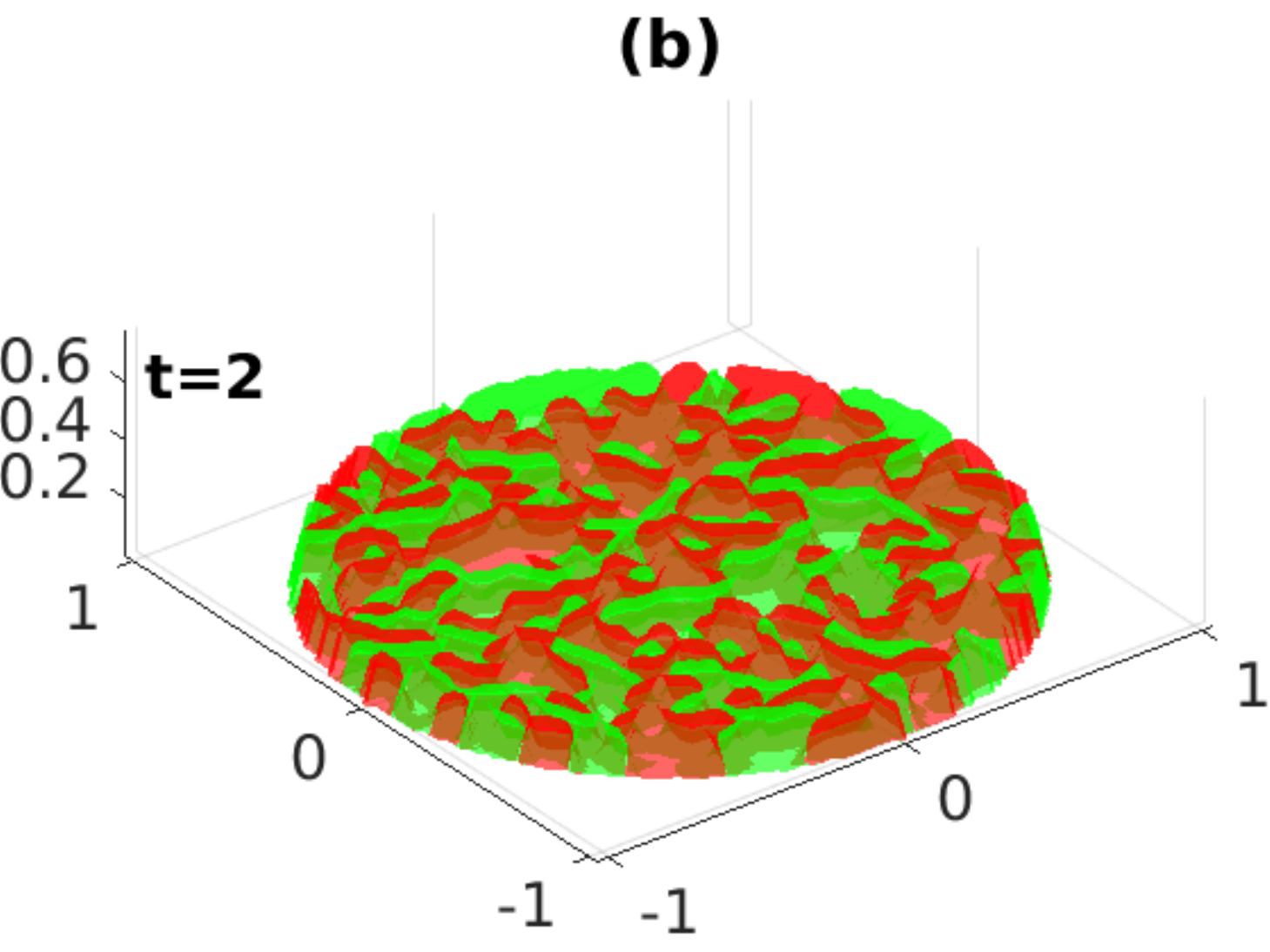}\includegraphics[width=0.33\textwidth]{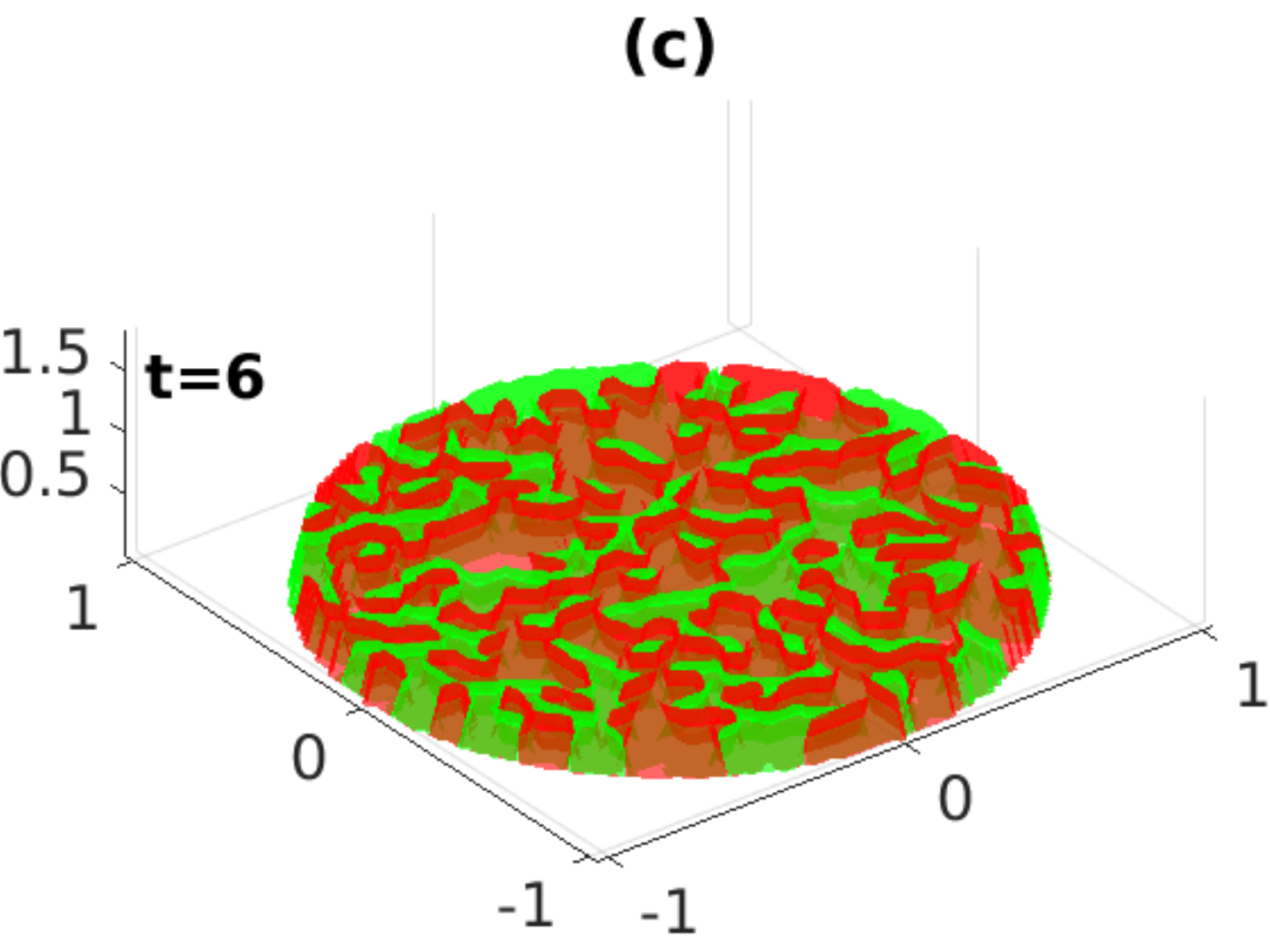}
		\end{center}
		\caption{\textit{Cell co-culture for species $ u_1 $ and $ u_2 $ over 6 days. We plot the case where where cells are densely seeded, i.e., $U_1=U_2=0.1,\, N_{u_1}=N_{u_2}=200$ for day $ 0,2 $ and day $ 6 $. We set parameters as $\delta_1 =0.15,\,\delta_2=0,\,a_{12}=0,\,a_{21} =0$ in \eqref{eq:para_Table_1plus}. Other parameters are given in Table \ref{TABLE1}. }}
		\label{FIG6.2}
	\end{figure}
	In Figure \ref{FIG7} we plot the evolution of the total number and its proportion for species $ u_1 $ and $ u_2 $ over 6 days of the simulation.
	\begin{figure}[H]
		\begin{center}
			\includegraphics[width=0.33\textwidth]{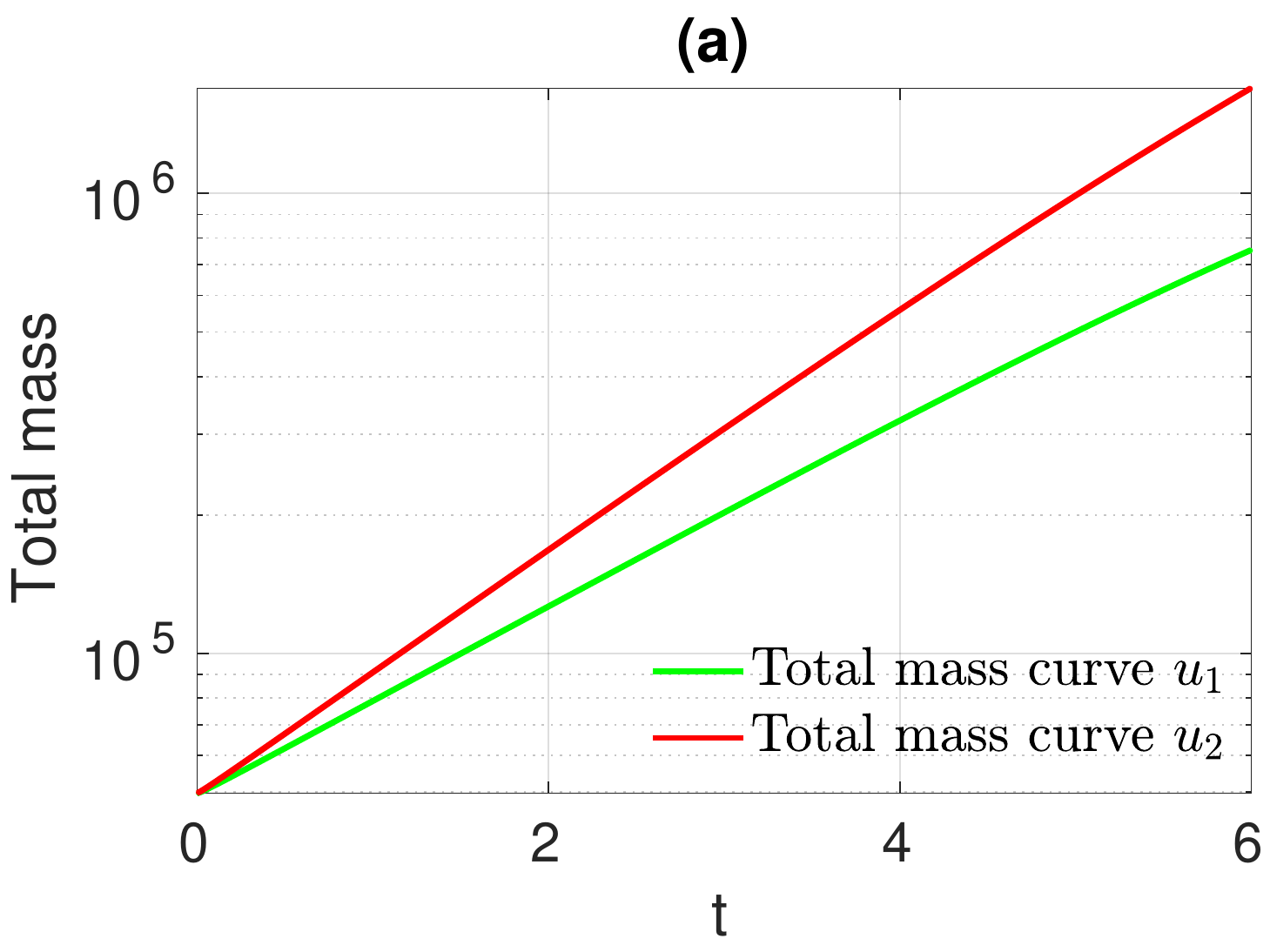}\includegraphics[width=0.33\textwidth]{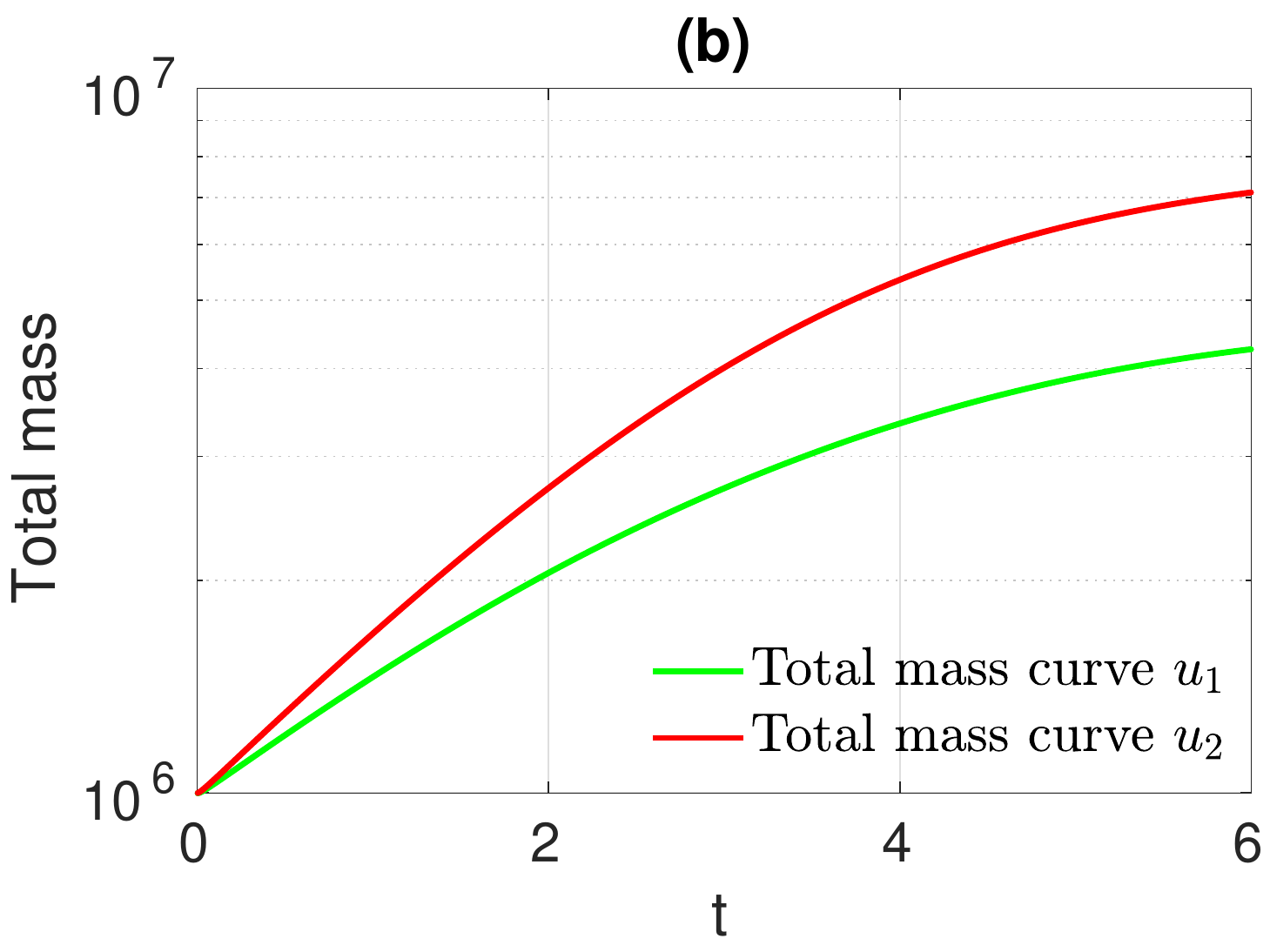}\includegraphics[width=0.33\textwidth]{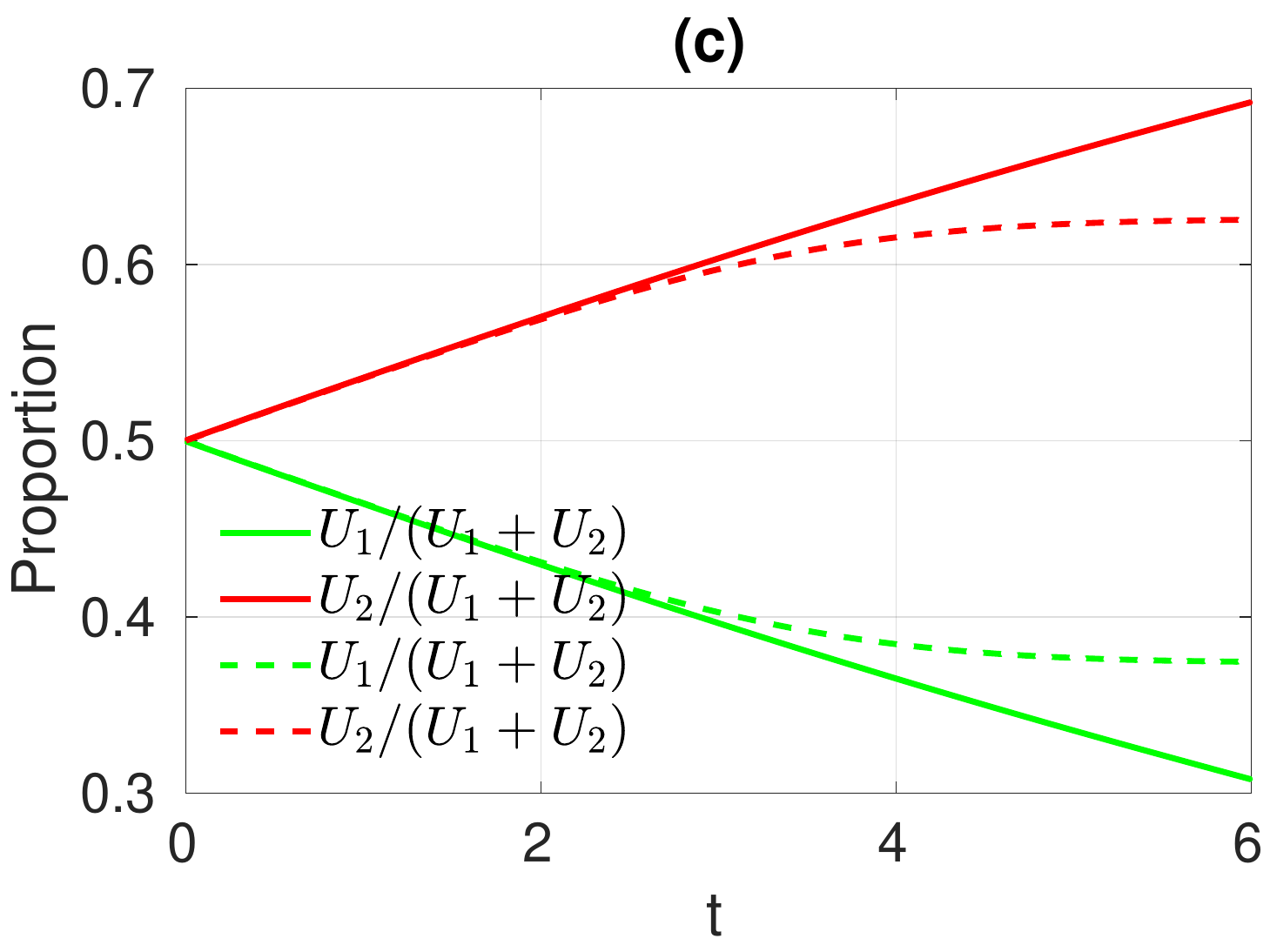}
		\end{center}
		\caption{\textit{Evolution of the  total number (in log scale) and its proportion for species $ u_1 $ and $ u_2 $ over 6 days. Figure (a) is the total number plot corresponding to the simulations in Figure \ref{FIG6.1} while Figure (b) corresponds to the simulations in Figure \ref{FIG6.2}. In Figure (c), the solid lines represents the proportion when the number of initial cell cluster equals $ N_{u_1}=N_{u_2}=10 $ and the dash lines represents the proportion when $ N_{u_1}=N_{u_2}=200$.  Parameters are given in Table \ref{TABLE1} and $\delta_1 =0.15,\,\delta_2=0,\,a_{12}=0,\,a_{21} =0$ in \eqref{eq:para_Table_1plus}.}}
		\label{FIG7}
	\end{figure}
	From Figure \ref{FIG7} (a)-(b), since $u_2$ is resistant to the drug, the number of population $u_2$ is much greater than $u_1$. However, we can also observe a difference in their cell growth curves. In Figure (a) we can see that both cells are in the period of exponential growth from day 0 to day 6 (a base-10 log scale is used for the y-axis). Conversely, in Figure (b) the growth curves for both cells are converging to a constant from day 4 to day 6, implying that the cell co-culture is reaching a saturation stage. More importantly, in Figure (c), we observe a significant difference in the development of population ratios. In fact, since the spatial competition is still the dominant factor in the first two days, we can hardly see any difference between the dashed lines and solid lines. The proportion of the dominant population grows almost linearly. However, the proportion of the densely seeded group changed much slower after day 4, while the sparsely seeded group still grows linearly. This shows that although the growth rate of $u_1$ is at a competitive disadvantage, due to the sufficient number of cluster in the initial stage, and due to the segregation principle, $u_1$ does not die out in a short time in the competition. Although the competitive exclusion applies in this case, the time for the extinction of  $u_1$ will be very long.
	
	\subsubsection{Dependency on the law of the initial distribution}
	In the experiment, the size of the cell dish can be a factor to determine the law of the initial distribution for the cell. In general, under the same total cell number, a small size cell dish will lead to a biased initial distribution and cells are more likely to aggregate at the border. While a big size cell dish will make the cell distribution more homogeneous, thus the initial distribution follows a uniform distribution. Therefore, in this section, we study whether the population ratio can be affected by the law of initial distribution.
	We will choose the beta distribution for the choice of the radius $ r $ and the angle $ \theta $ follows the uniform distribution on $ [0,2\pi] $, that is
	\[ \lbrace r_n\rbrace_{n=1,\ldots,N} \sim  \mathrm{Beta}(\alpha,\beta),\quad \lbrace \theta_n\rbrace_{n=1,\ldots,N} \sim \mathrm{U}(0,2\pi). \]
	The coordinate transformation of the initial distribution to a unit disk is as follows
	\begin{equation}\label{2.2}
		\begin{cases}
			x_n = \sqrt{r_n} \cos(\theta_n)\\
			y_n = \sqrt{r_n} \sin(\theta_n)
		\end{cases}n=1,2,\ldots,N.
	\end{equation}
	In Figure \ref{FIG1}, we plot the density function of the beta distribution for different $ \alpha,\beta $
	\[ f_{\alpha,\beta} (x) = 1/ B(\alpha,\beta)\, x^{\alpha-1}(1-x)^{\beta-1}, \]
	where $ B(\alpha,\beta) $ is a normalization constant to ensure that the total integral is $ 1 $.
	\begin{figure}[H]
		\begin{center}
			\includegraphics[width=0.5\textwidth]{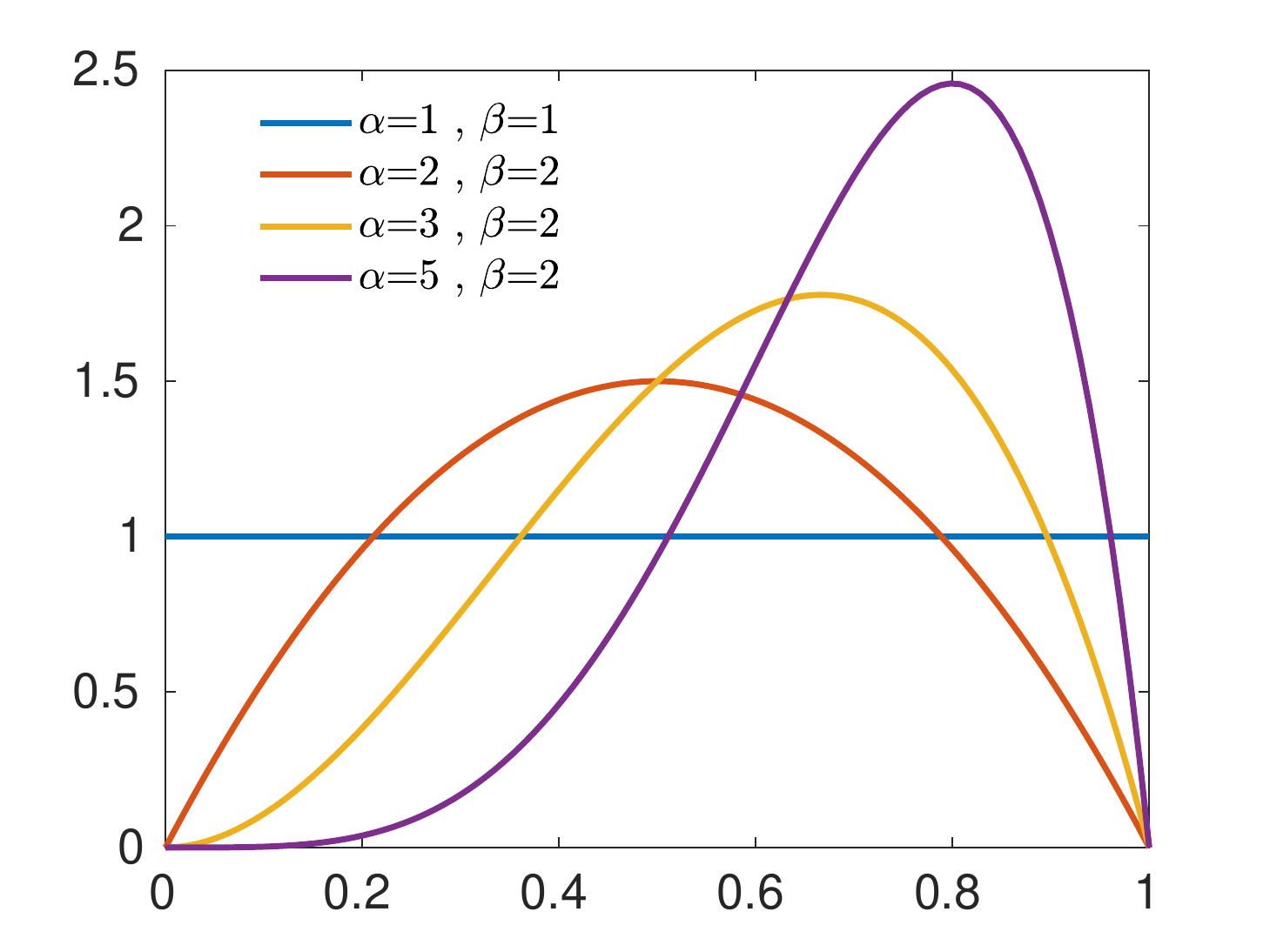}
		\end{center}
		\caption{\textit{Density function of the initial distribution $ f_{\alpha,\beta} (x) = 1/ B(\alpha,\beta)\, x^{\alpha-1}(1-x)^{\beta-1} $ for different $ \alpha $ and $ \beta $, where $ B(\alpha,\beta) $ is a normalization constant to ensure that the total integral is $ 1 $.}}
		\label{FIG1}
	\end{figure}
	Our simulation will mainly compare the following two cases 
	\[ (\alpha_1,\beta_1) = (1,1),\quad (\alpha_2,\beta_2) = (3,2).\]
	We plot the initial distributions of the two different cases in Figure \ref{FIG2} where we choose $ 40 $ cell clusters (i.e., $ N_{u_1}=40 $ and $ N_{u_2}= 40 $  in \eqref{2.2}) for species $ u_1 $ and species $ u_2 $.	
	\begin{figure}[H]
		\begin{center}
			\includegraphics[width=0.5\textwidth]{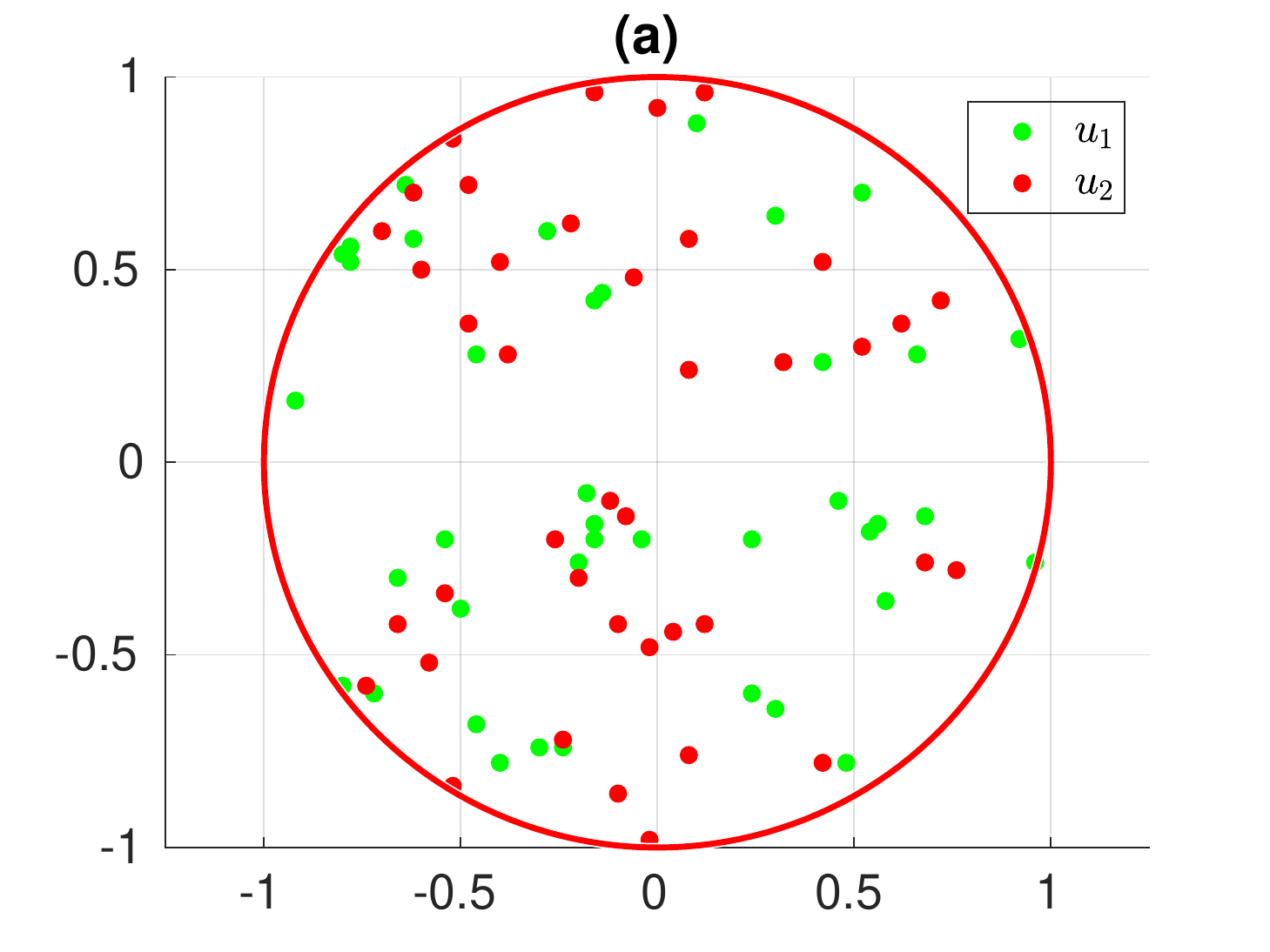}\includegraphics[width=0.5\textwidth]{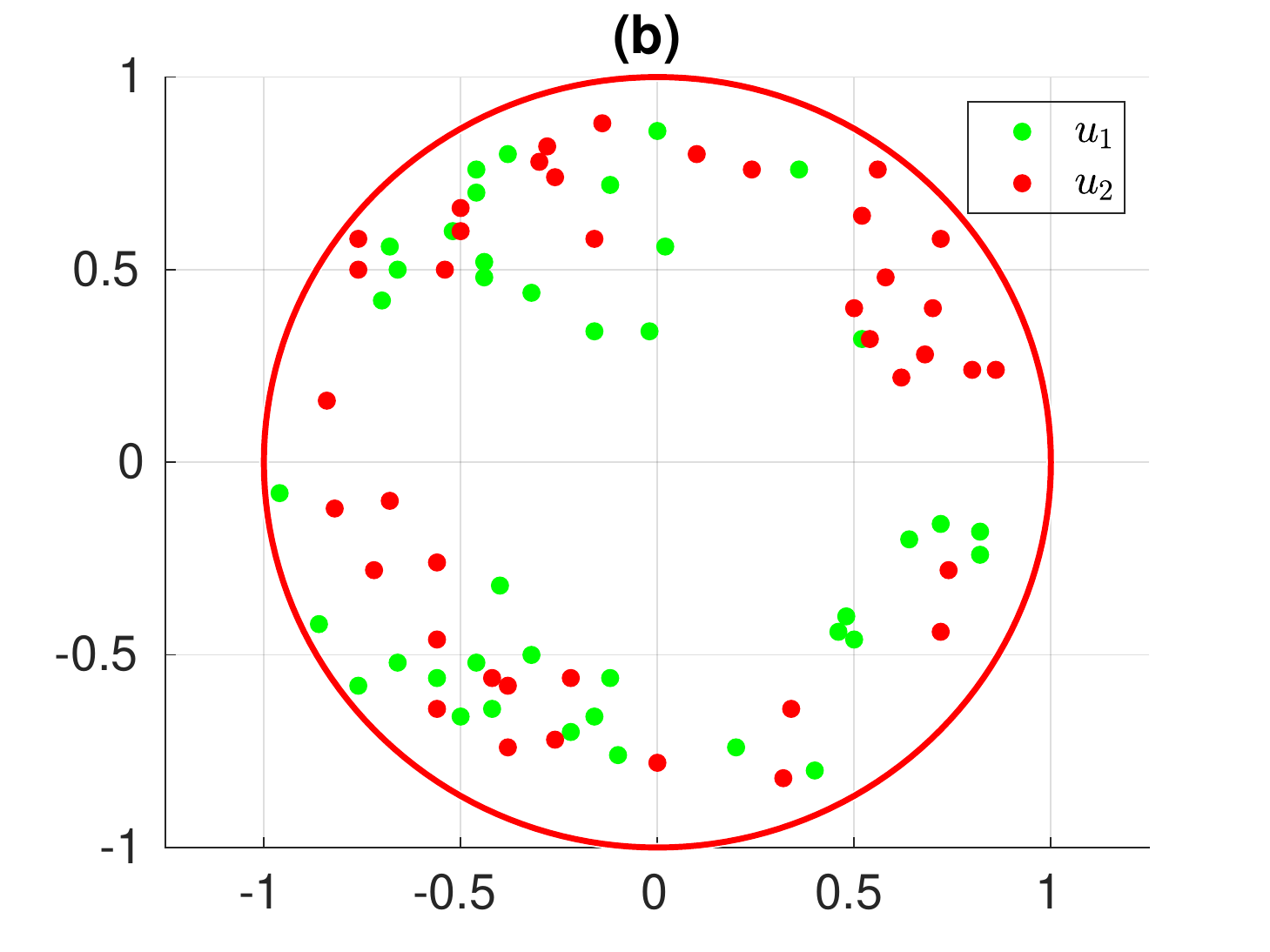}
		\end{center}
		\caption{\textit{Spatial distribution of the initial values when  $(\alpha,\beta) =(1,1)  $ (Figure (a)) and $ (\alpha,\beta) =(3,2) $ (Figure (b)). Here red dots and green dots in Figure \ref{FIG2} represent cell clusters. We take the cell cluster number $ N_{u_1}=40 $ and $ N_{u_2}=40 $ for both cases. }}
		\label{FIG2}
	\end{figure}	
	Suppose the initial cell clusters $ N_{u_1}=N_{u_2}=40 $ and cell total number $ U_1=U_2=0.02 $, which is equally distributed in each cell cluster. Typical numerical solutions are shown in Figure \ref{FIG3} when $  (\alpha_1,\beta_1) = (1,1) $ and in Figure \ref{FIG4} when $ (\alpha_2,\beta_2) = (3,2) $. 
	\begin{figure}[H]
		\begin{center}
			\includegraphics[width=0.33\textwidth]{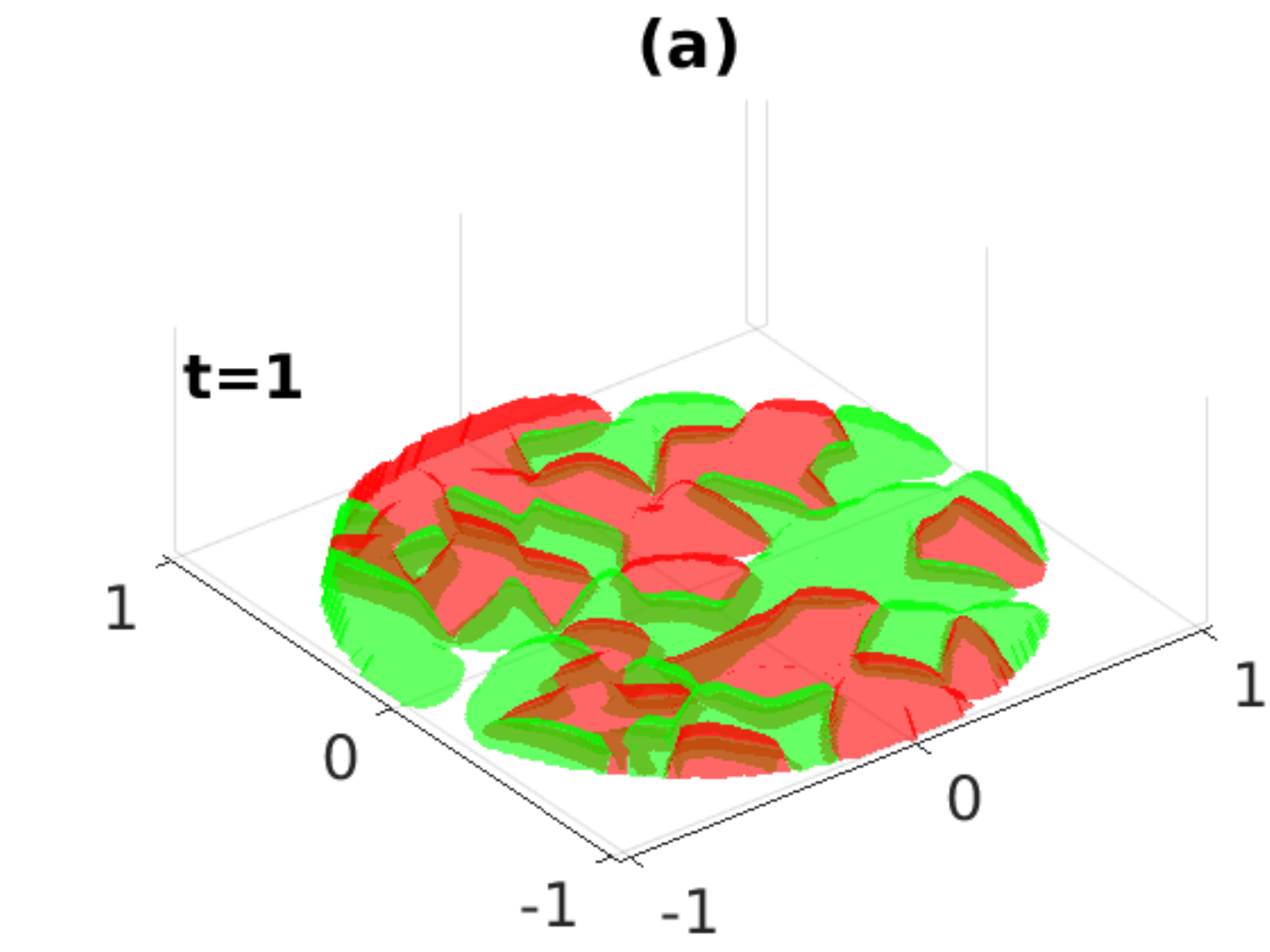}\includegraphics[width=0.33\textwidth]{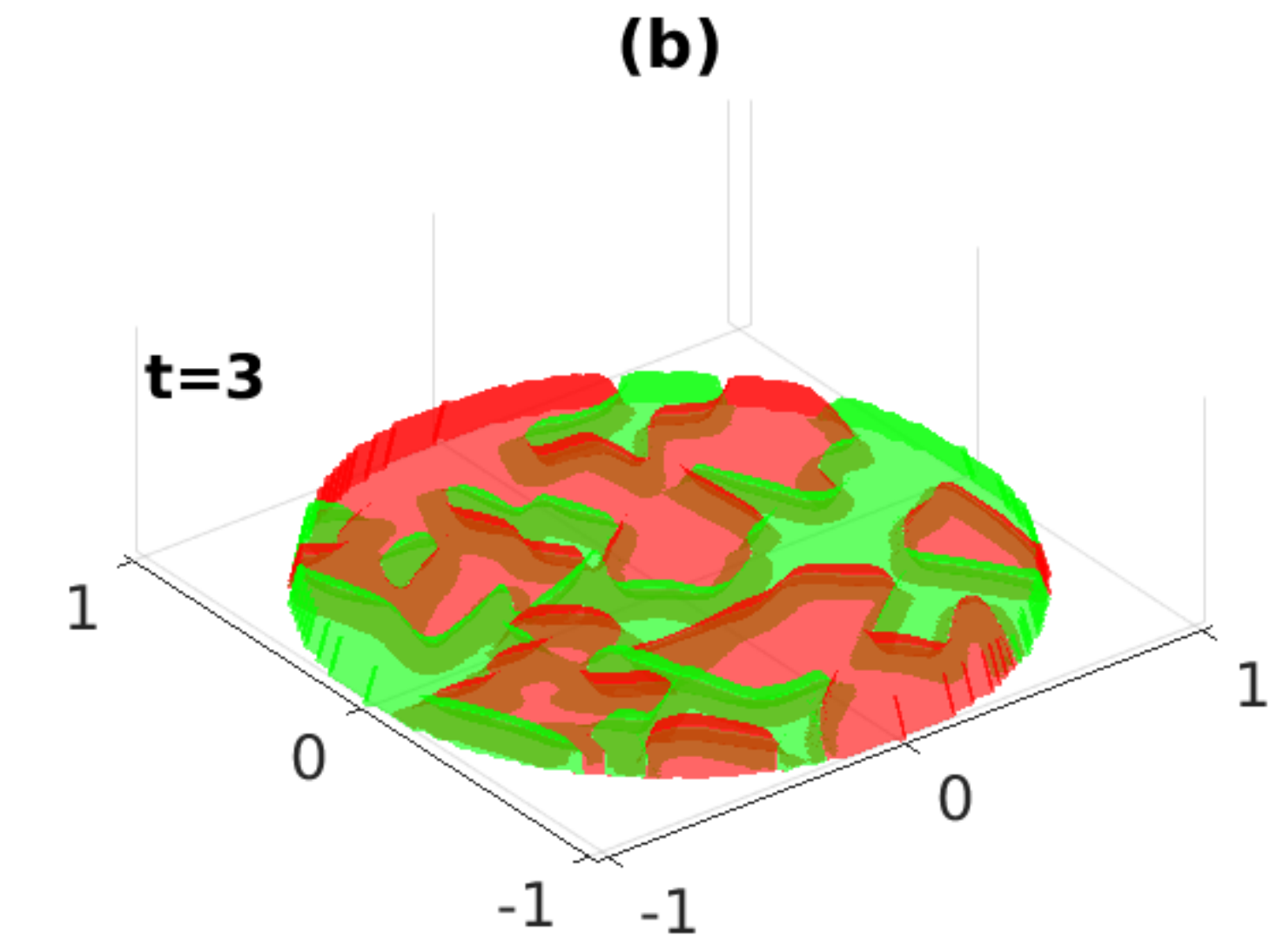}\includegraphics[width=0.33\textwidth]{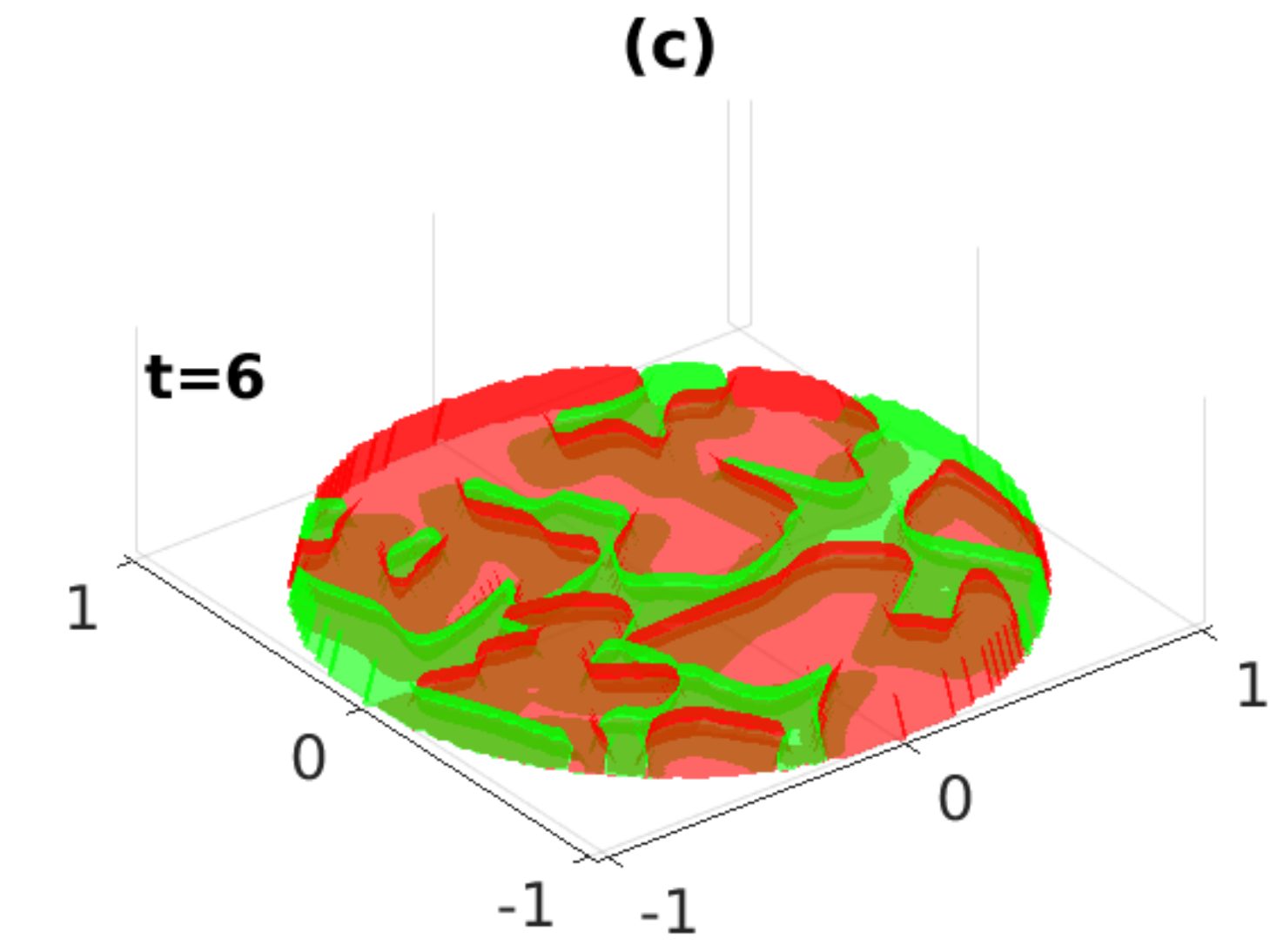}
		\end{center}
		\caption{\textit{Cell co-culture for species $ u_1 $ and $ u_2 $ over 6 days. We plot the case where the initial distribution follows beta function with parameters $(\alpha,\beta) =(1,1)$, namely the uniform distribution, for day $ 1,3 $ and day $ 6 $. Parameters are given in Table \ref{TABLE1} and $\delta_1 =0.15,\,\delta_2=0,\,a_{12}=0,\,a_{21} =0$ in \eqref{eq:para_Table_1plus}.}}
		\label{FIG3}
	\end{figure}
	\begin{figure}[H]
		\begin{center}
			\includegraphics[width=0.33\textwidth]{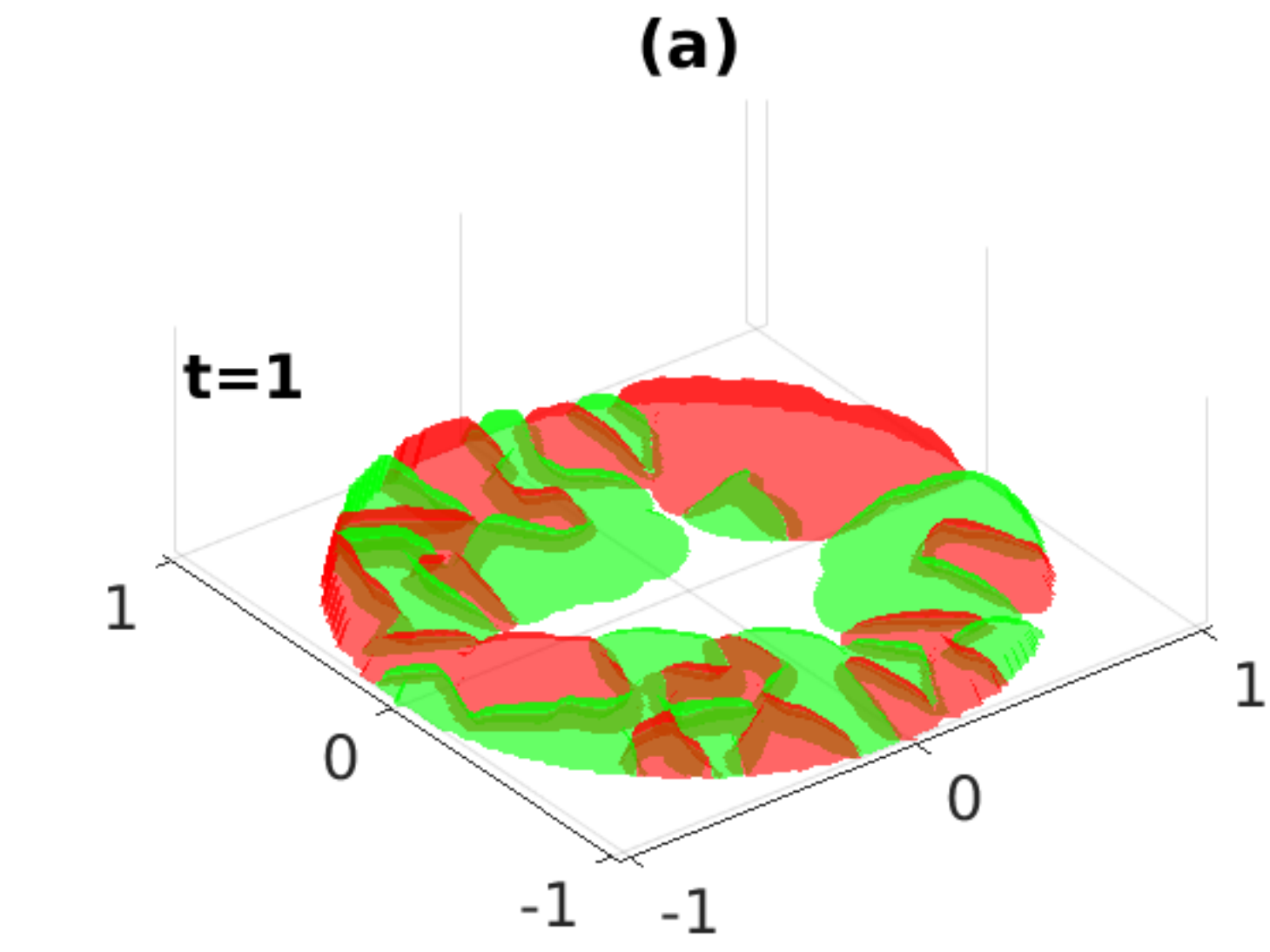}\includegraphics[width=0.33\textwidth]{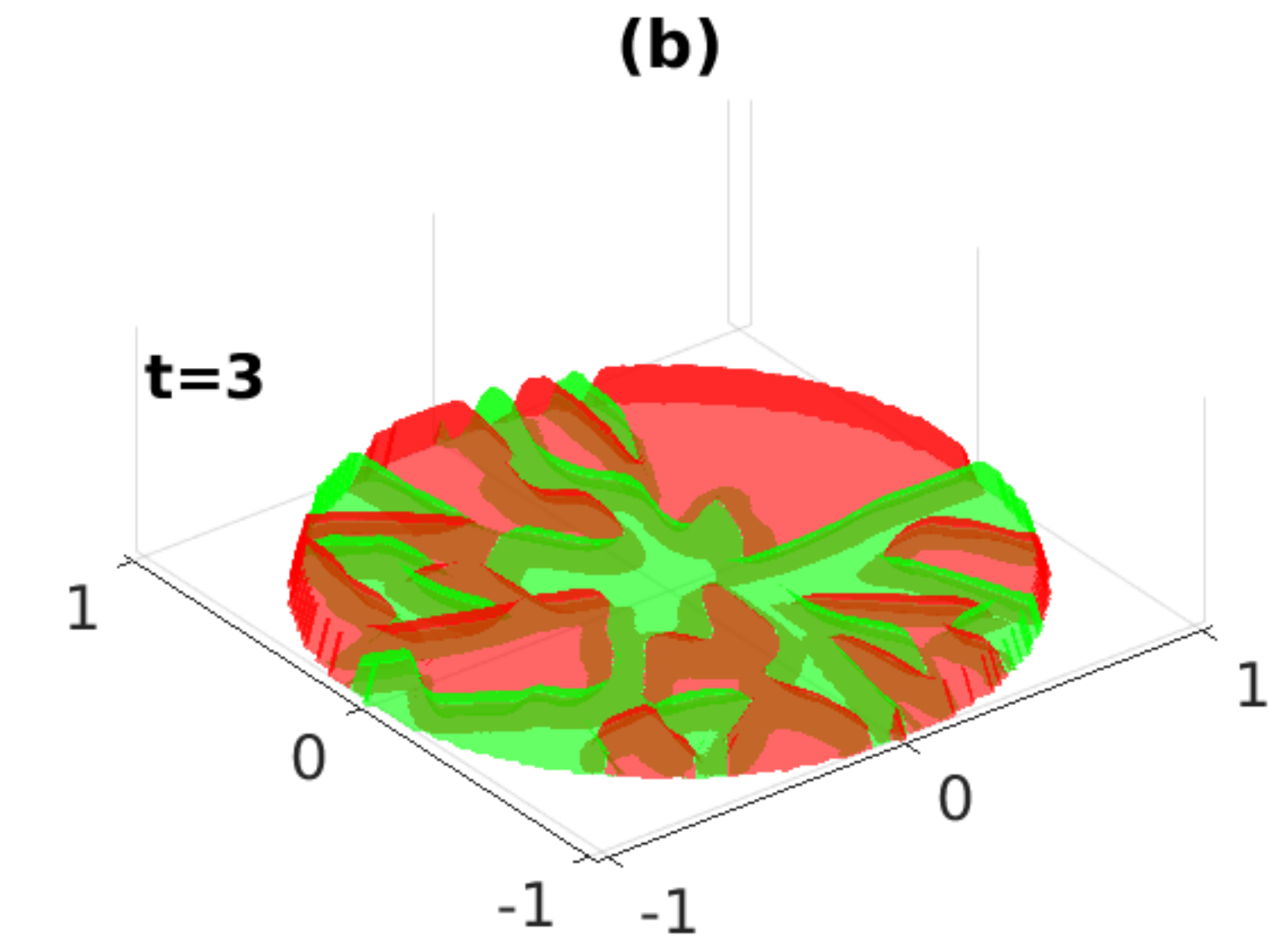}\includegraphics[width=0.33\textwidth]{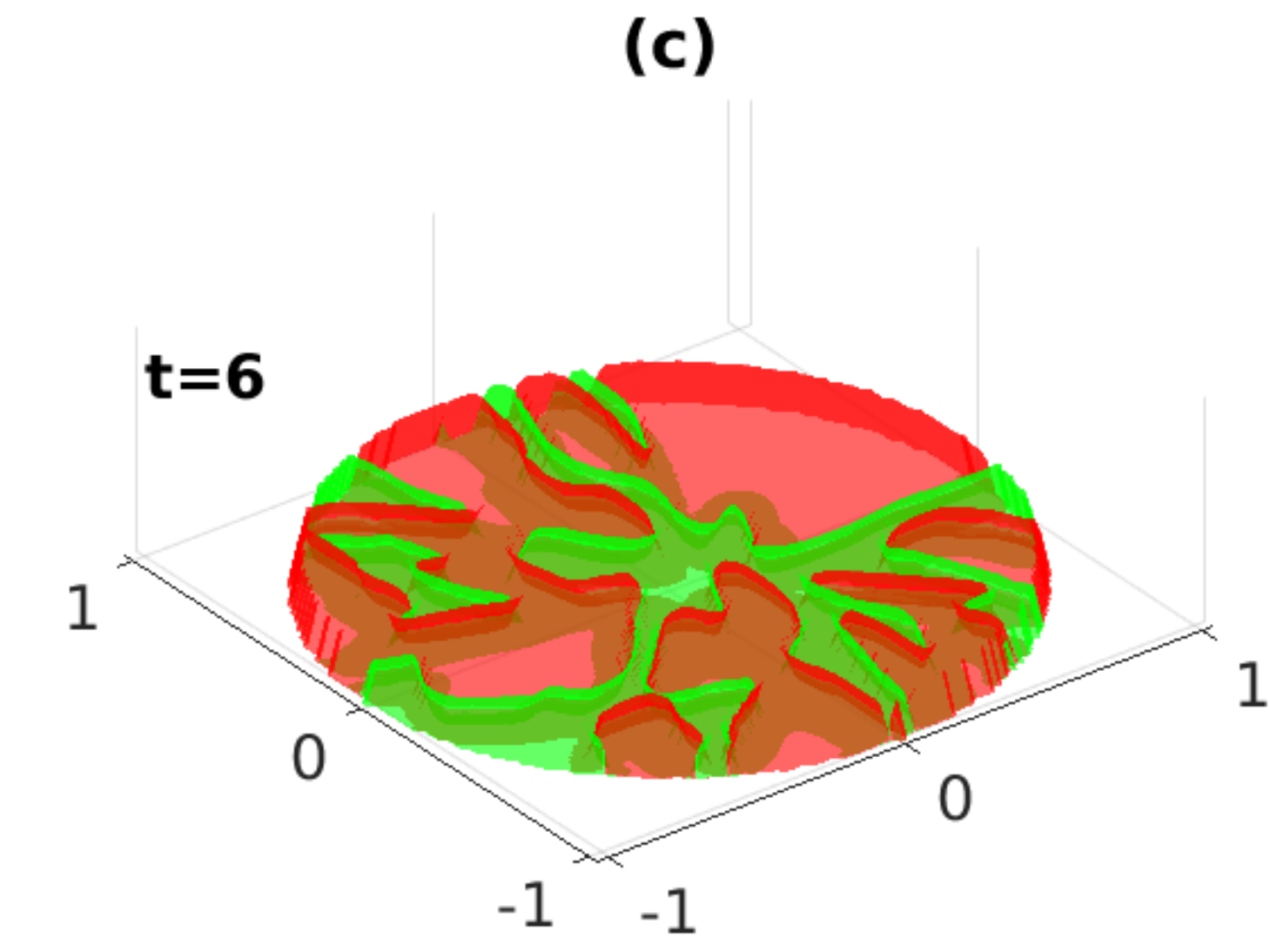}
		\end{center}
		\caption{\textit{Cell co-culture for species $ u_1 $ and $ u_2 $ over 6 days. We plot the case where the initial distribution follows beta function with parameters $(\alpha,\beta) =(3,2)$, namely a biased distribution, for day $ 1,3 $ and day $ 6 $. Parameters are given in Table \ref{TABLE1} and $\delta_1 =0.15,\,\delta_2=0,\,a_{12}=0,\,a_{21} =0$ in \eqref{eq:para_Table_1plus}. }}
		\label{FIG4}
	\end{figure}
	Now we plot the evolution of the  total number for species $ u_1 $ and $ u_2 $ over 6 days.
	\begin{figure}[H]
		\begin{center}
			\includegraphics[width=0.33\textwidth]{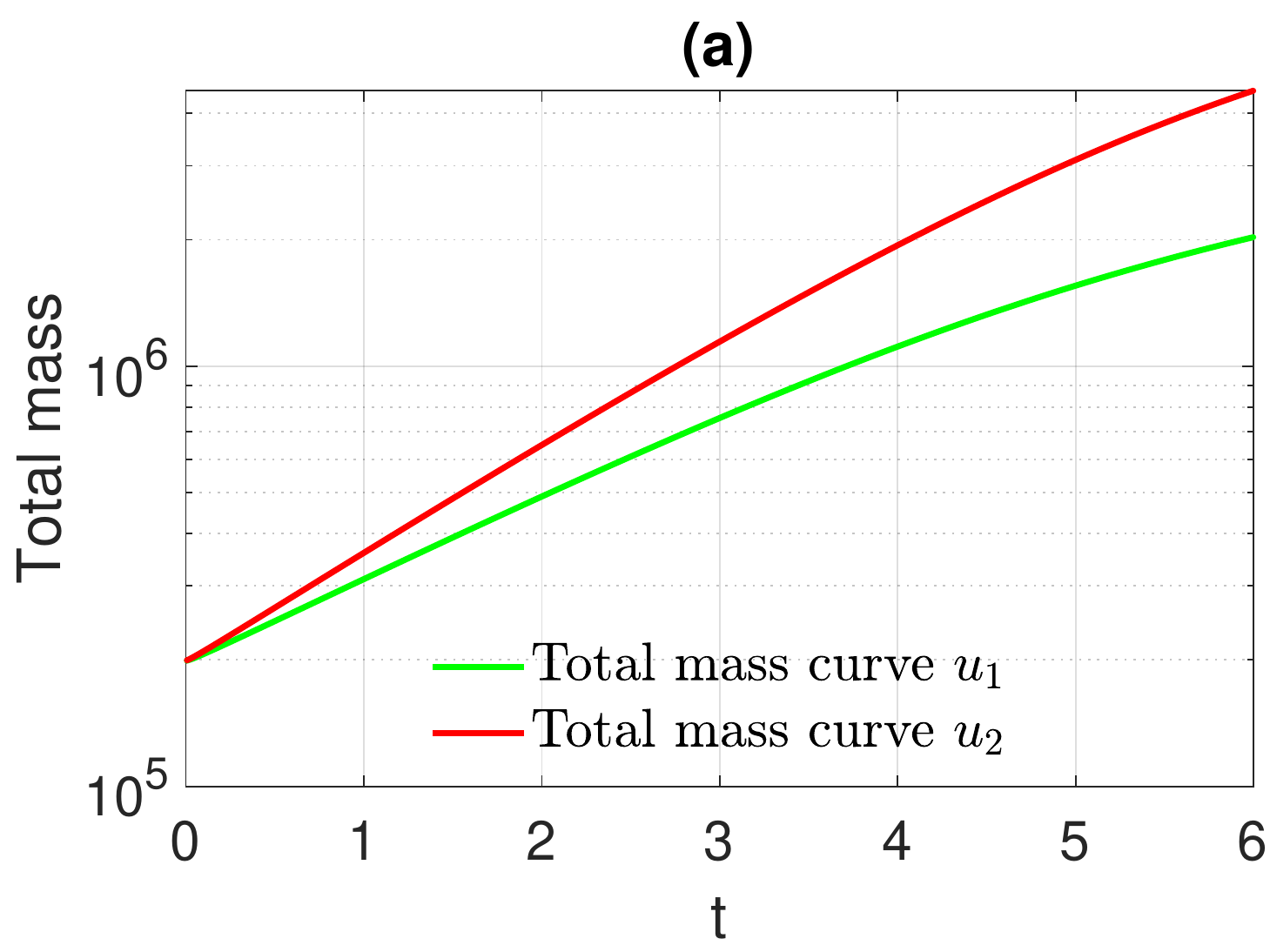}\includegraphics[width=0.33\textwidth]{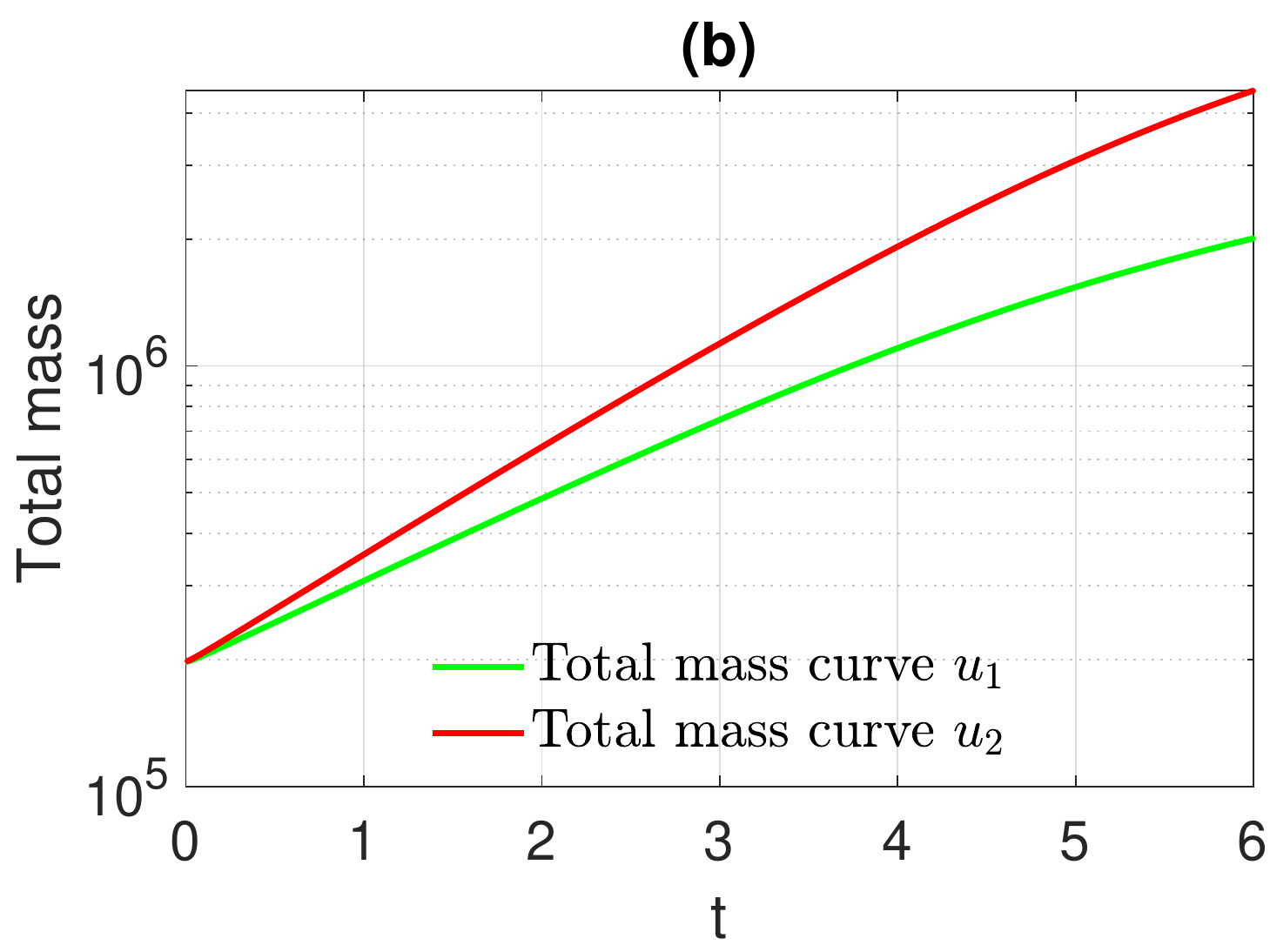}\includegraphics[width=0.33\textwidth]{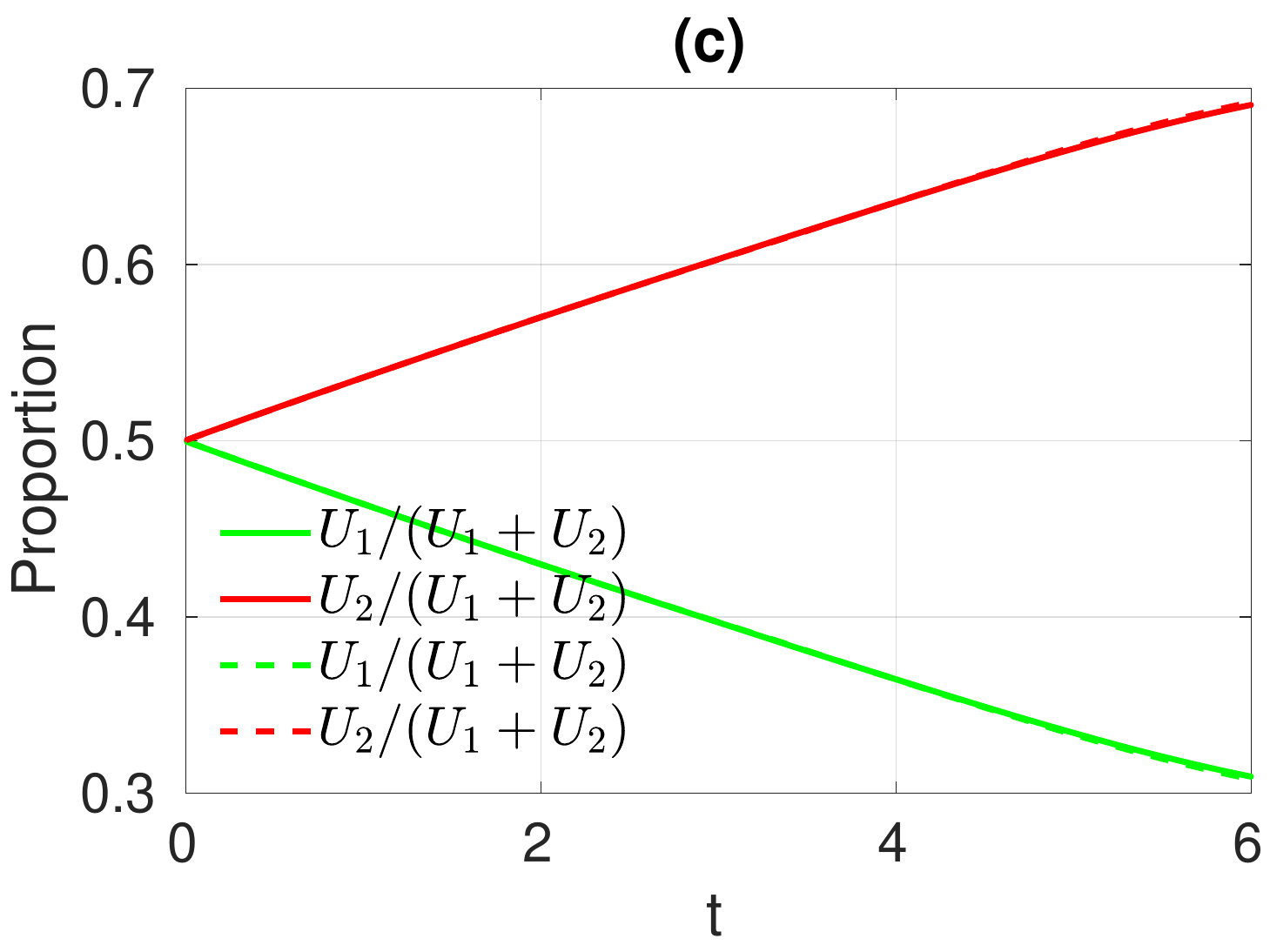}
		\end{center}
		\caption{\textit{Evolution of the  total number (in log scale) and its proportion for species $ u_1 $ and $ u_2 $ over 6 days. Figure (a) is the total number plot corresponding to the simulation with an uniform initial distribution in Figure \ref{FIG3} while Figure (b) corresponds to the simulation with a biased initial distribution in Figure \ref{FIG4}. In Figure (c), the solid lines represents the proportion when the initial distribution follows uniform distribution, i.e., $(\alpha,\beta) =(1,1)  $ and the dash lines represents the proportion to the case $ (\alpha,\beta) =(3,2) $. From Figure (c), we can see  that they are overlapped.  Parameters are given in Table \ref{TABLE1} and $\delta_1 =0.15,\,\delta_2=0,\,a_{12}=0,\,a_{21} =0$ in \eqref{eq:para_Table_1plus}.}}
		\label{FIG5}
	\end{figure}
	From  Figure \ref{FIG5} we can see that the law of initial distribution has almost no influence on the final proportion of species. We also tried different scenarios when the cell total numbers are 20, 50 and 100 or with different extra mortality rate $\delta_1 =0,0.2 $ and $0.5$,  the results are similar. Thus we can deduce that the final relative proportion is stable under the variation of the law of the initial distribution.   
	
	Combining the above numerical experiments in Section 3.2.1 and Section 3.2.2, we can see that under the competitive principle, the difference in the spatial resources can change the competition induced by the cell dynamics.  To be more precise, under the case of sufficient spatial resources, the competitive mechanism can be more sufficiently expressed than in the case of less spatial resources. In Section 3.2.2, although we changed the law of the initial distribution of the cell seeding, as for the overall spatial resources, it is the same for both species. Therefore, the result of the competitive principle is almost the same in terms of the total number and the population ratio of the two populations.

	\subsection{Impact of the dispersion coefficient on the population ratio}
	
	In Section 3.2, when the parameters of the model are the same,  the competition induced by the cell dynamics can be reflected by the difference in the spatial resource. Now we assume the spatial resource is the same and we investigate the role of the dispersion coefficient in the evolution of the species.
	
	To that aim, we let the initial distribution of the two species follow the same uniform distribution and they are sparsely seeded on the dish. Furthermore, we let the cell dynamics for the two population to be almost the same, the only variable we control here is the dispersion coefficient for the population. We take the same uniform initial distribution at day 0, with the same number of initial cluster and the same amount of cell total number, i.e.,
	\begin{equation}\label{eq3.7}
		U_1=U_2= 0.005,\quad N_{u_1}  = N_{u_2}=10,\quad a_{12}=a_{21}=0. 
	\end{equation}
	We compare the following two scenarios in Table \ref{Table:Spatial_competition} where the only difference is the dispersion parameters.
	\begin{table}[H]
		\centering
		\begin{tabular}{ccccc}
			\doubleRule
			Parameters     & $d_1$ & $  d_2$ & $\delta_1$ & $\delta_2$\\
			\hline
			\noindent\textbf{scenario 1:}     & $ 2 $ & $ 2 $ & 0 & 0\\
			\noindent\textbf{scenario 2:}     & $ 2 $ & $ 0.2 $& 0 & 0\\
			\doubleRule
		\end{tabular}
		\caption{\textit{Two sets of dispersion coefficients for $u_1$ and $u_2$.}}
		\label{Table:Spatial_competition}
	\end{table}
	
		In scenario 1, the dispersion coefficients of the two species are the same, while in scenario 2 we suppose the species $u_1$ has an advantage in the spatial competition over its competitor $u_2$.
	\begin{figure}[H]
		\begin{center}
			\includegraphics[width=0.33\textwidth]{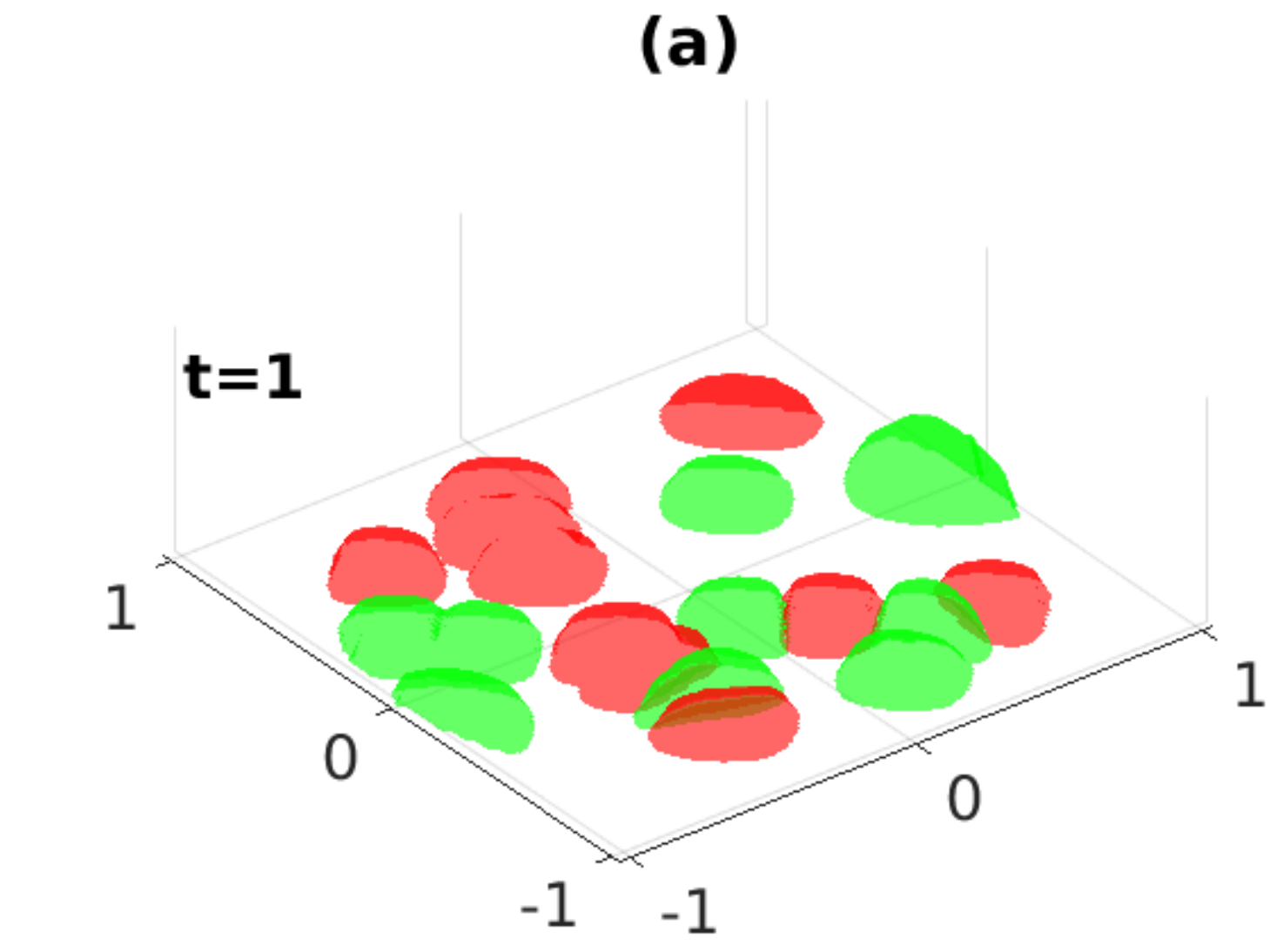}\includegraphics[width=0.33\textwidth]{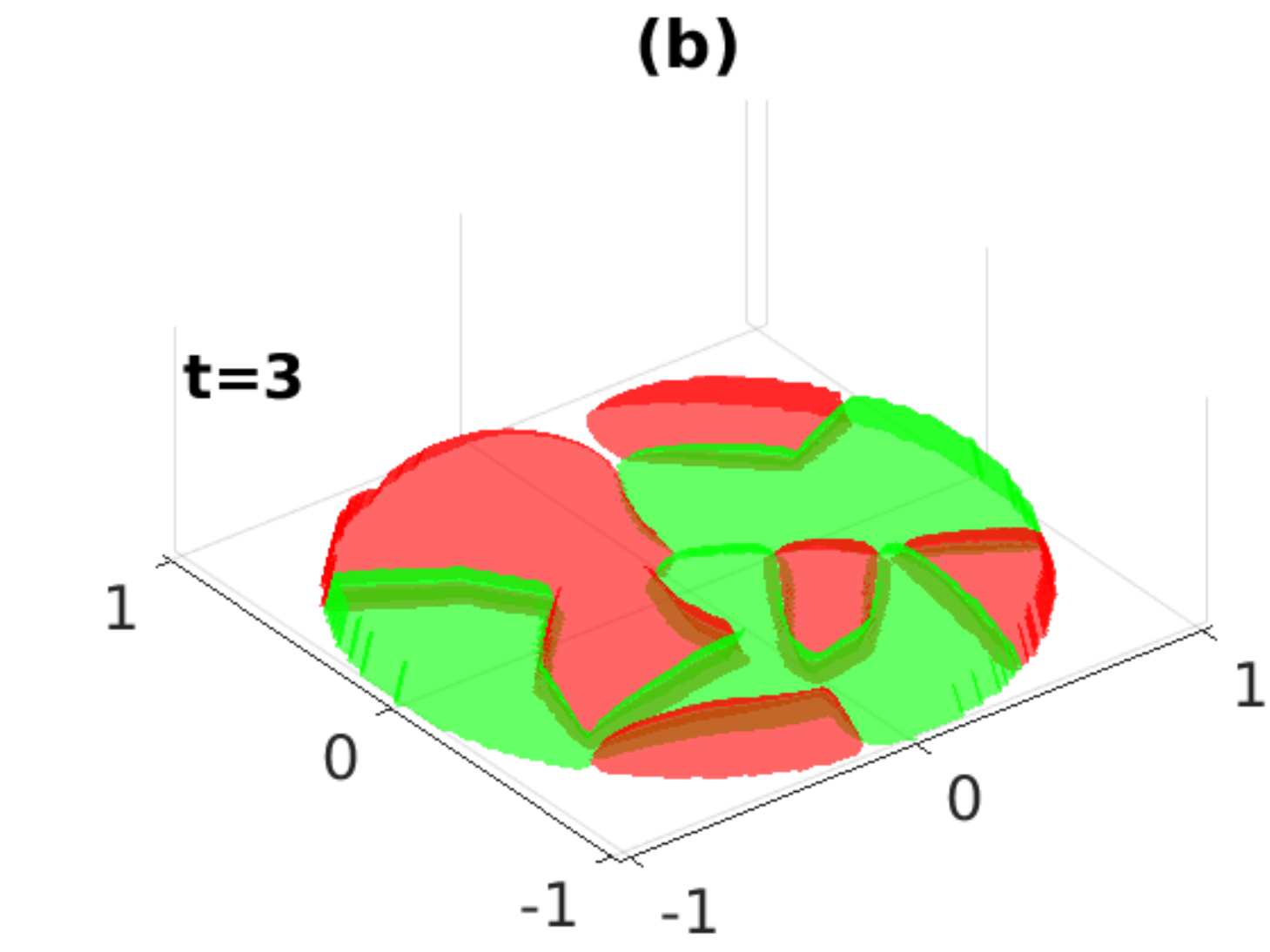}\includegraphics[width=0.33\textwidth]{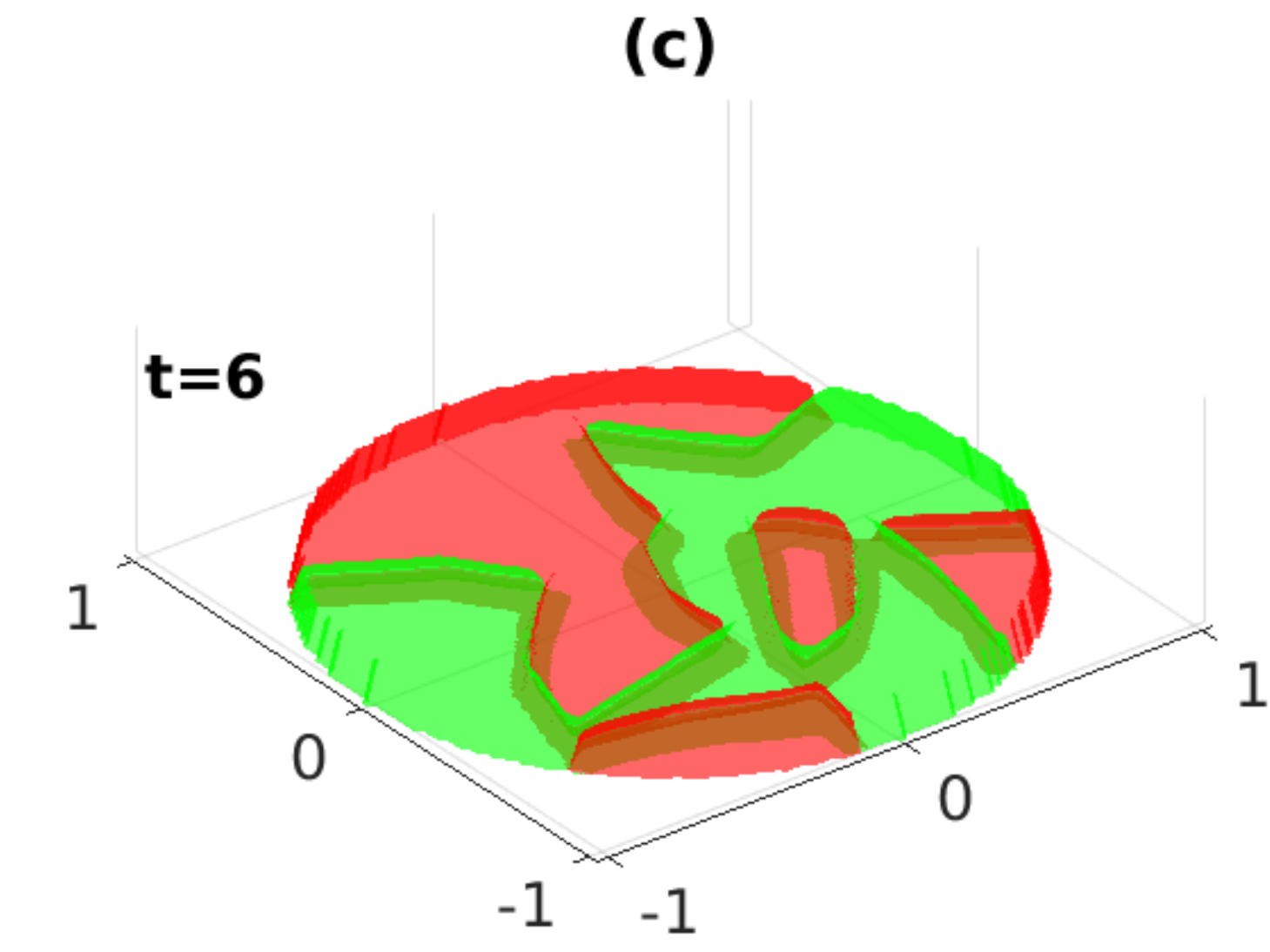}
			\includegraphics[width=0.33\textwidth]{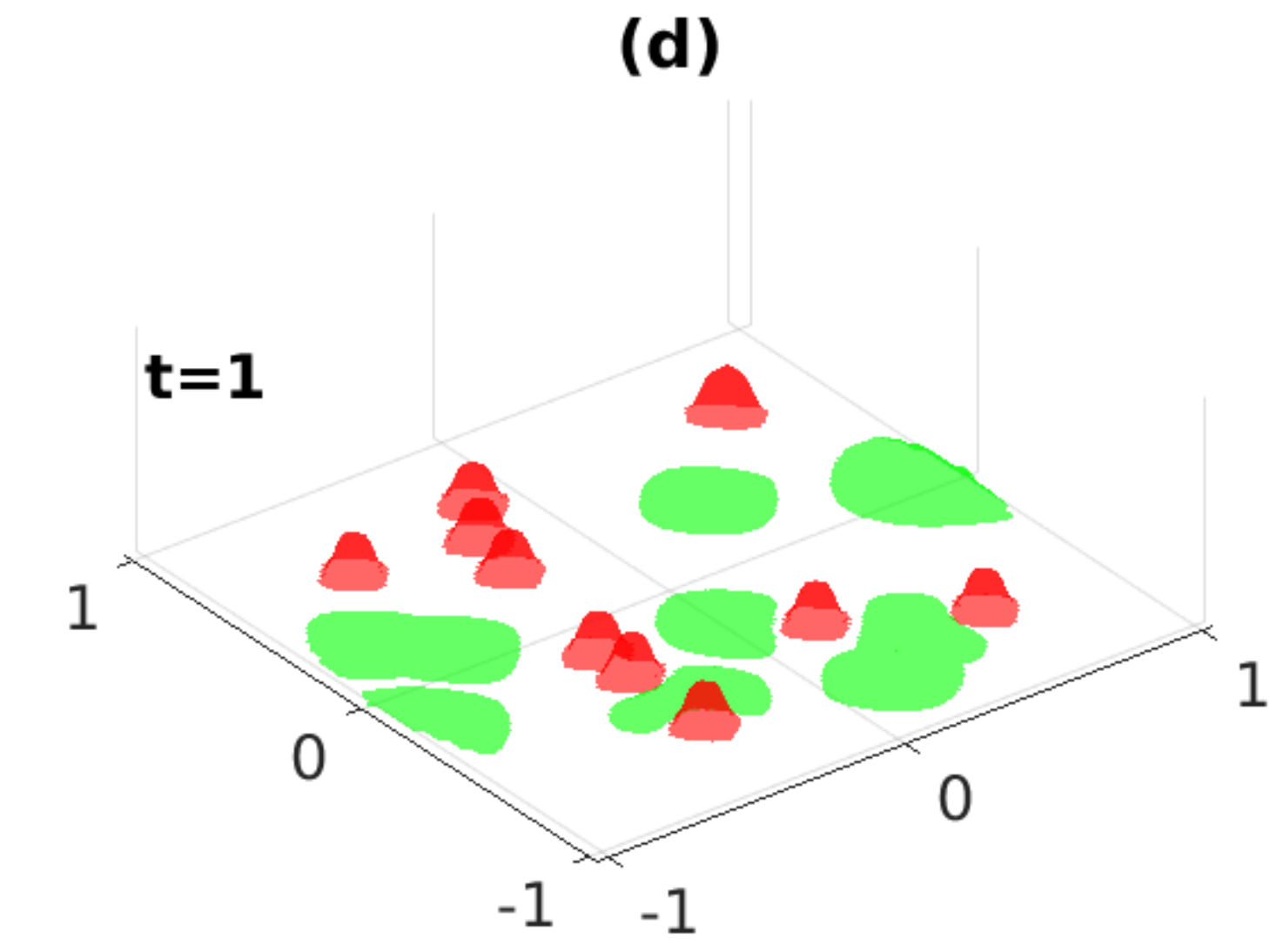}\includegraphics[width=0.33\textwidth]{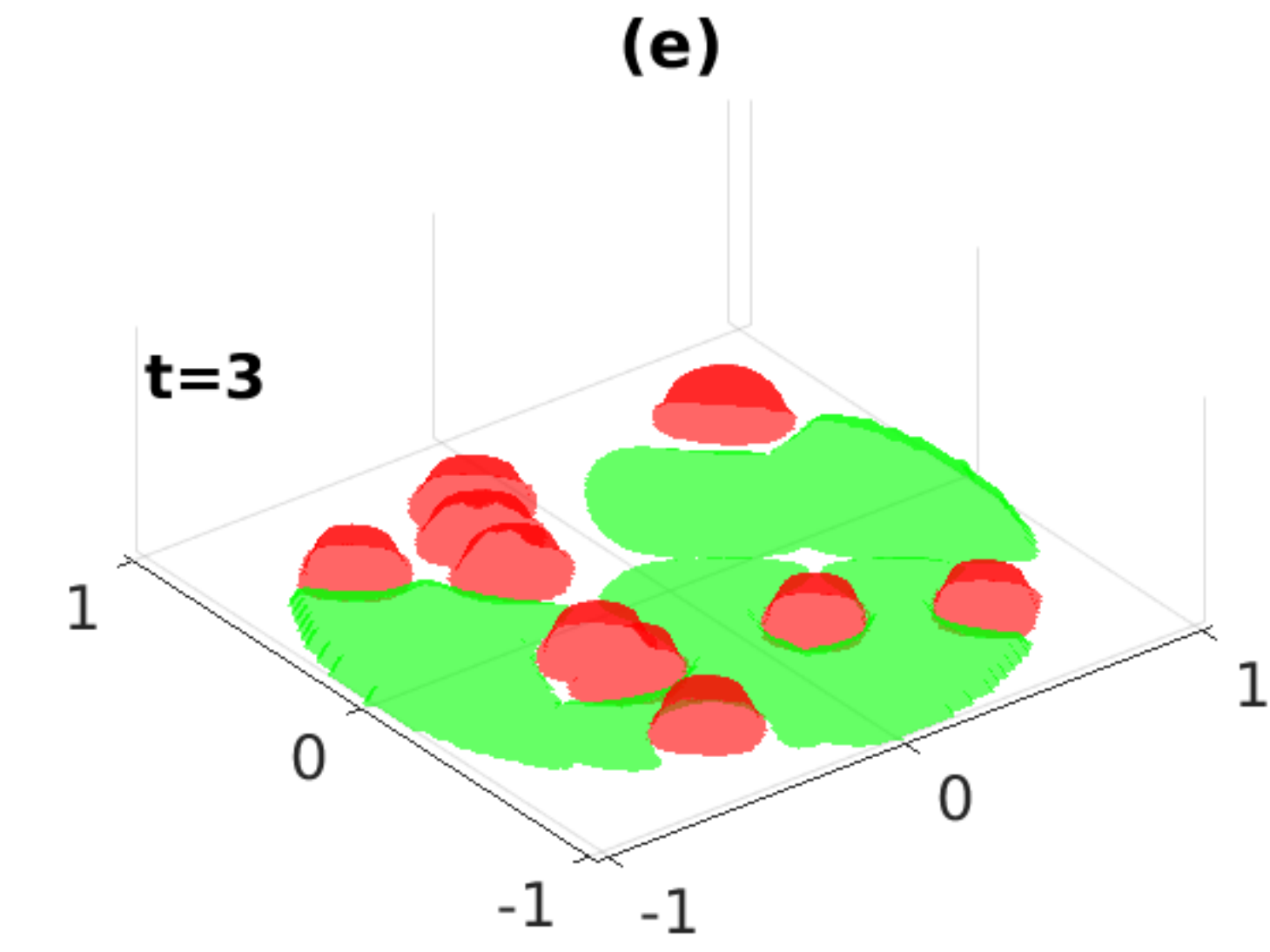}\includegraphics[width=0.33\textwidth]{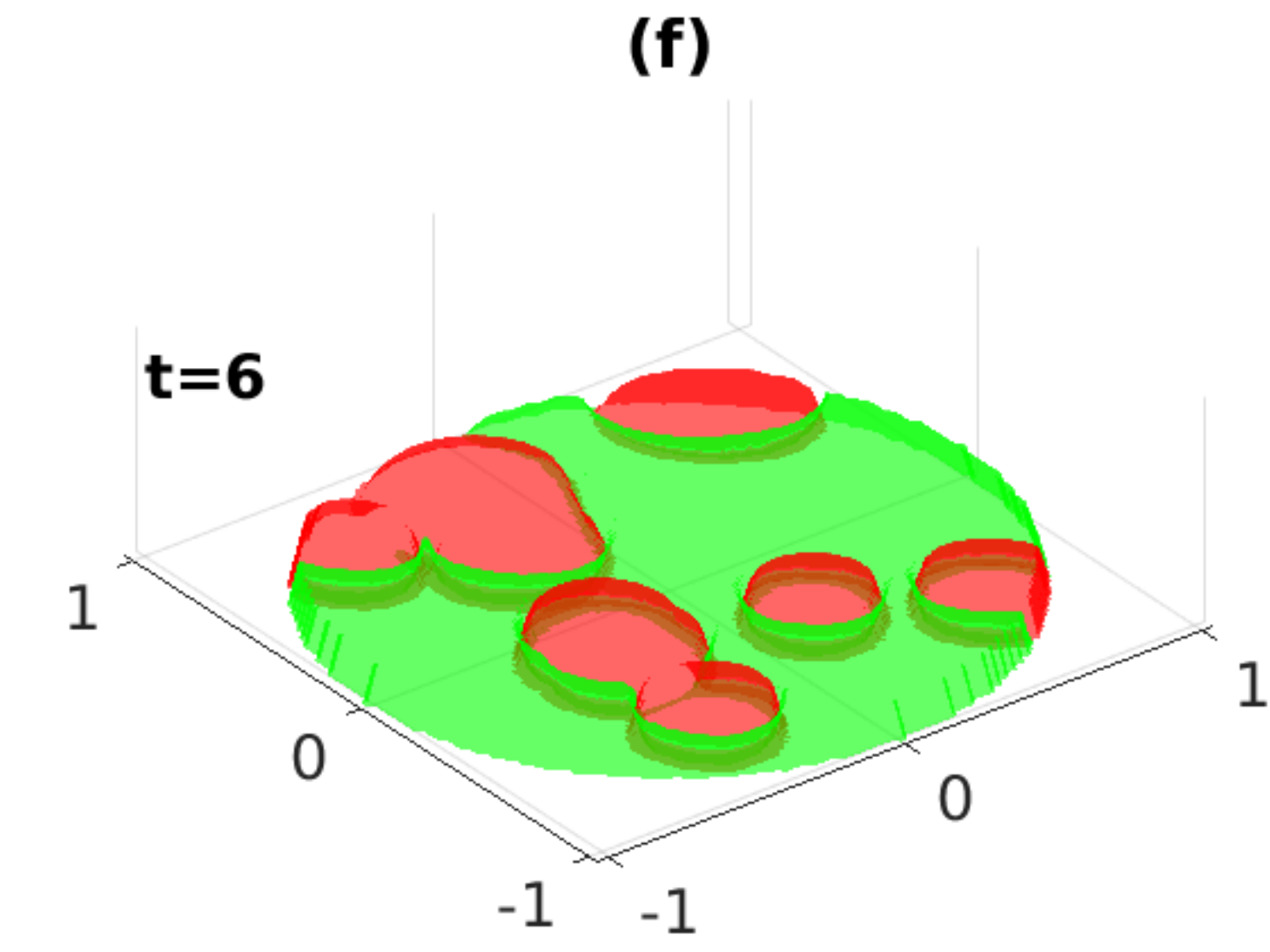}
		\end{center}
		\caption{\textit{Cell co-culture for species $ u_1 $ and $ u_2 $ over 6 days. Figure (a)-(c) corresponds to scenario 1  (i.e. with the parameters $d_1=2,d_2=2,\delta_1 =\delta_2=0$) while Figure (d)-(f) corresponds to scenario 2 (i.e. with $d_1=2,d_2=0.2,\delta_1 =\delta_2=0$). In both scenarios, the number of initial cluster and the cell total number are the same and follow \eqref{eq3.7} and the same uniform distribution. We plot the simulations for day $1,3 $ and day $ 6 $. Other parameters are given in Table \ref{TABLE1}. }}
		\label{FIG8}
	\end{figure}
	Now we plot the evolution of the total number and the population ratios for species $ u_1 $ and $ u_2 $ over 6 days.
	\begin{figure}[H]
		\begin{center}
			\includegraphics[width=0.33\textwidth]{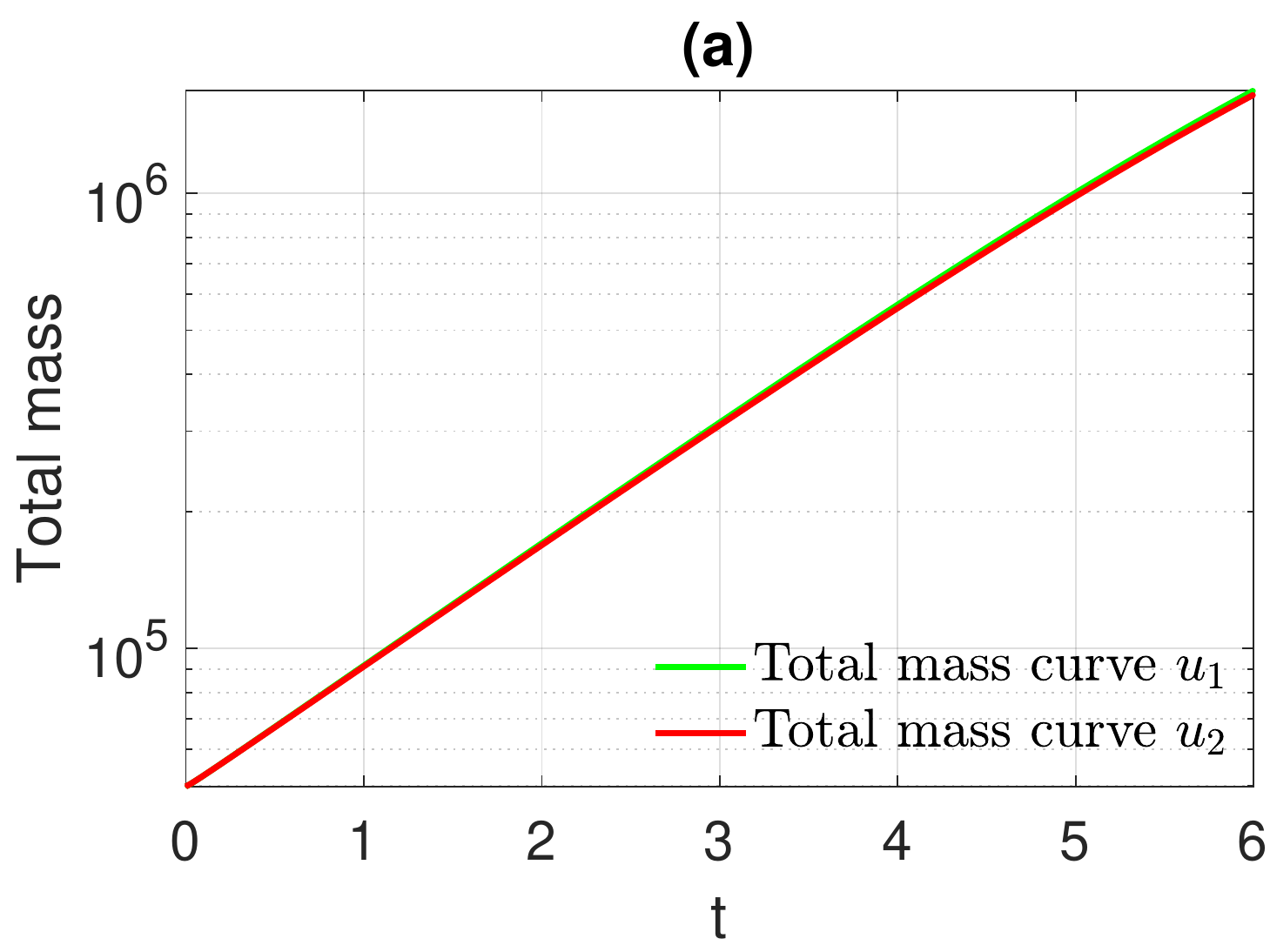}\includegraphics[width=0.33\textwidth]{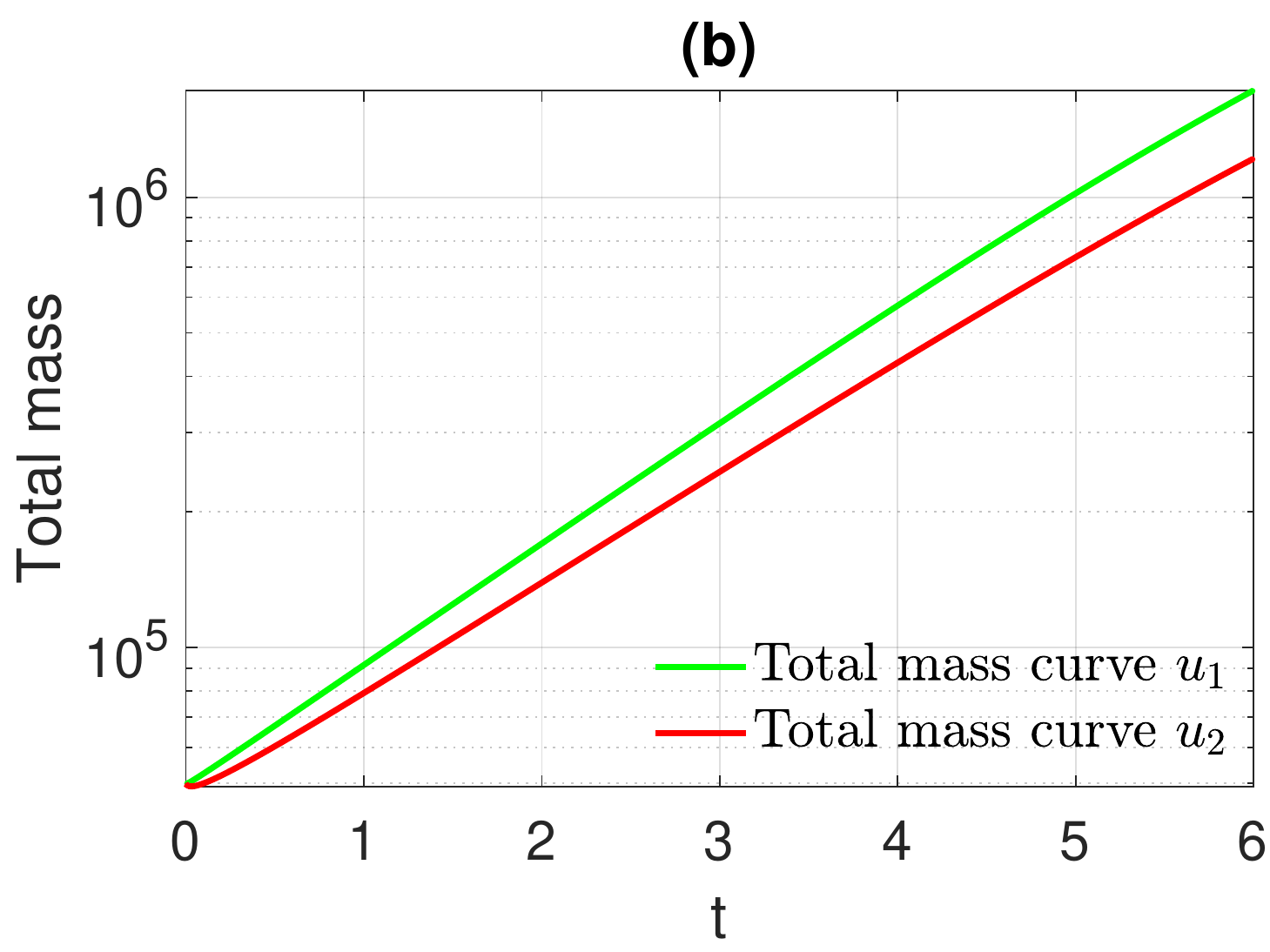}\includegraphics[width=0.33\textwidth]{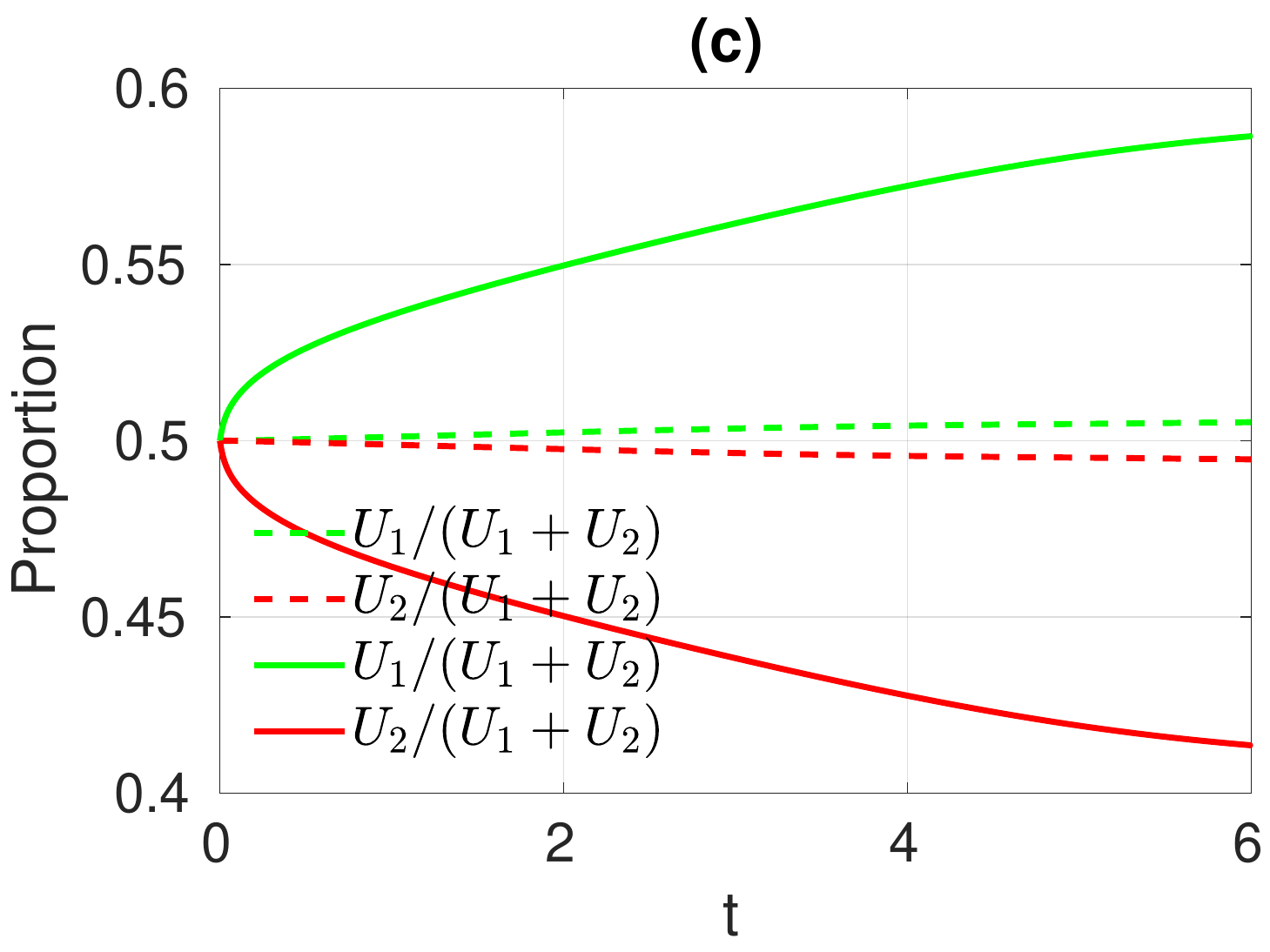}
		\end{center}
		\caption{\textit{Evolution of the  total number (in log scale) and its proportion for species $ u_1 $ and $ u_2 $ over 6 days. In Figure (a) we plot the total number of cells corresponding to the scenario 1. In Figure (b) we plot the total number of cells corresponding to the scenario 2.  In Figure (c) we plot the population ratios and the dashed lines corresponds to scenario 1 while the solid lines corresponds to scenario 2 in \ref{Table:Spatial_competition}. Other parameters are given in Table \ref{TABLE1} and \eqref{eq3.7}.}}
		\label{FIG9}
	\end{figure}
	
	The main result from Figure \ref{FIG9} is that the dispersion coefficient can have a great impact on the population ratio after 6 days.

	Next, we consider the following scenario where $u_1$ has the advantage in dispersion coefficient but is at a disadvantage induced by drug treatment. Therefore 
	\begin{table}[H]
		\centering
		\begin{tabular}{ccccc}
			\doubleRule
			Parameters     & $d_1$ & $  d_2$ & $\delta_1$ & $\delta_2$ \\
			\hline
			\noindent\textbf{scenario 3:}     & $ 2 $ & $ 0.2 $ & $\textbf{0.1}$ & $0$   \\
			\doubleRule
		\end{tabular}
		\caption{\textit{This scenario corresponds to the case where the species $u_1$ spreads faster than the species $u_2$. Moreover, due to a drug treatment, the mortality of the species $u_1$ is strictly positive while the mortality of the species $u_2$ is zero (i.e. the drug treatment does not affect the second species). In the context of cancer cell, the species $u_1$ would correspond to the sensitive cells to the drug while $u_2$ would correspond to the cell resistant to the drug treatment.  }}
		\label{Table:Spatial_competition_2}
	\end{table}
	\begin{figure}[H]
		\begin{center}
			\includegraphics[width=0.33\textwidth]{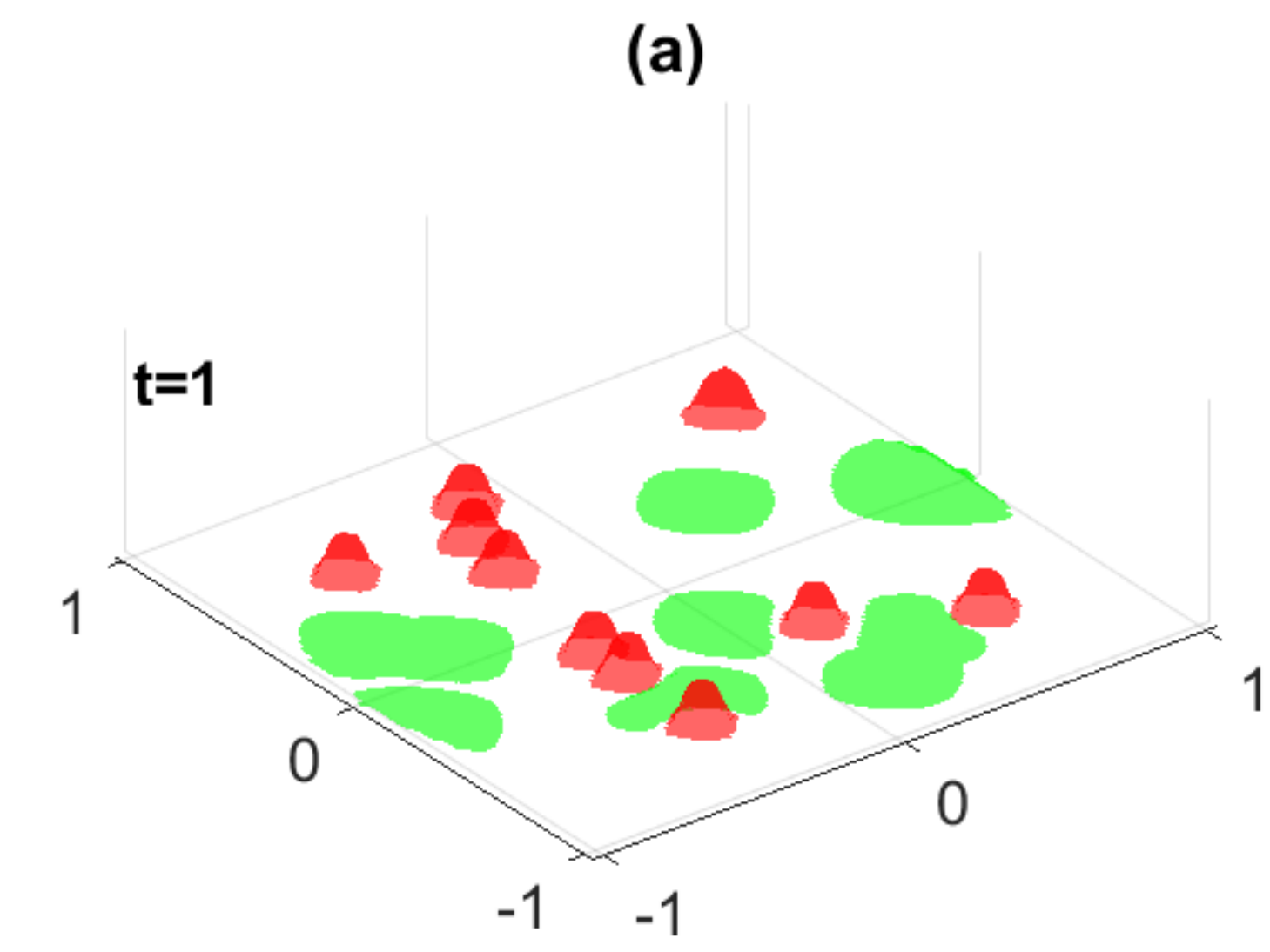}\includegraphics[width=0.33\textwidth]{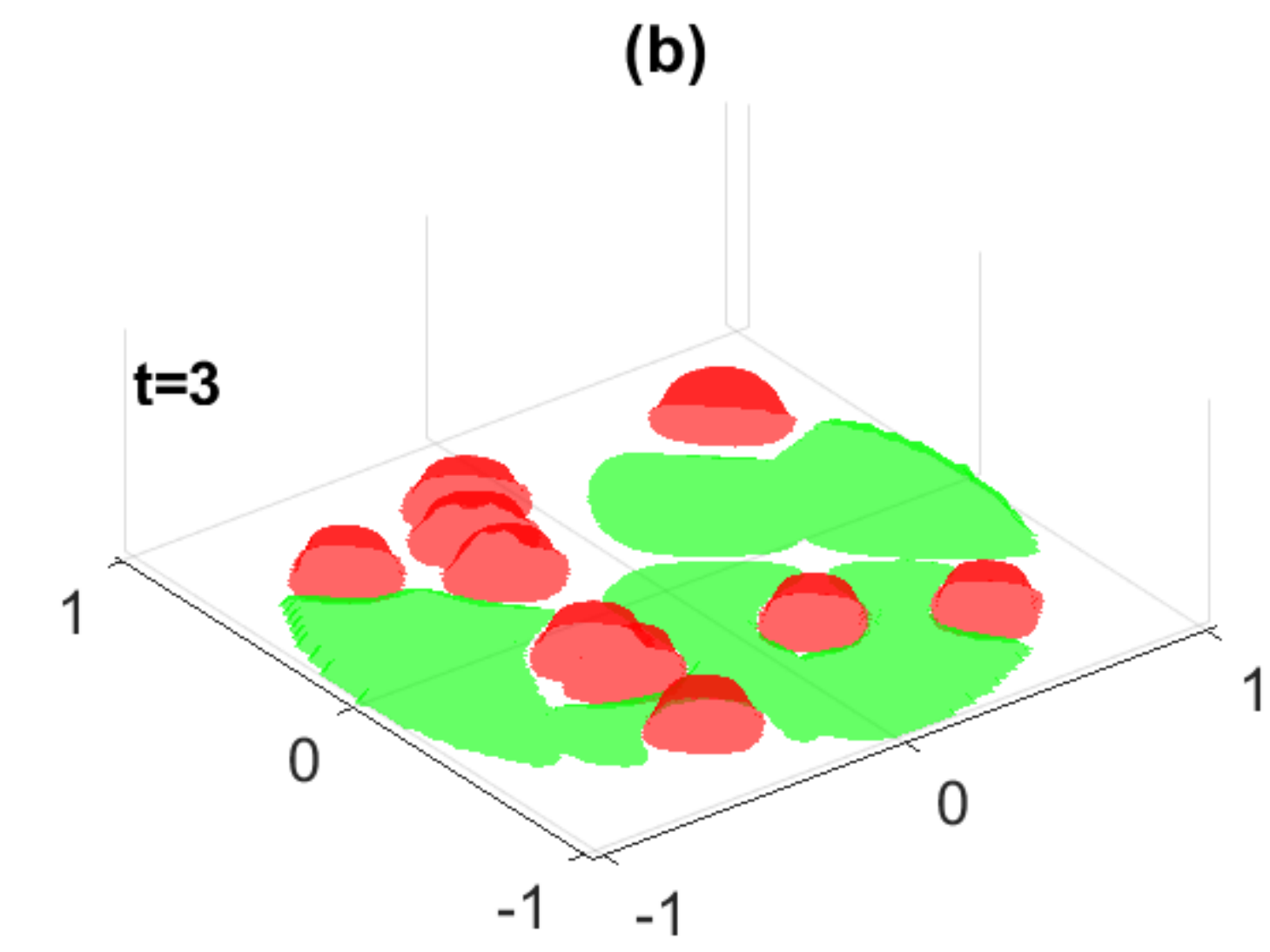}\includegraphics[width=0.33\textwidth]{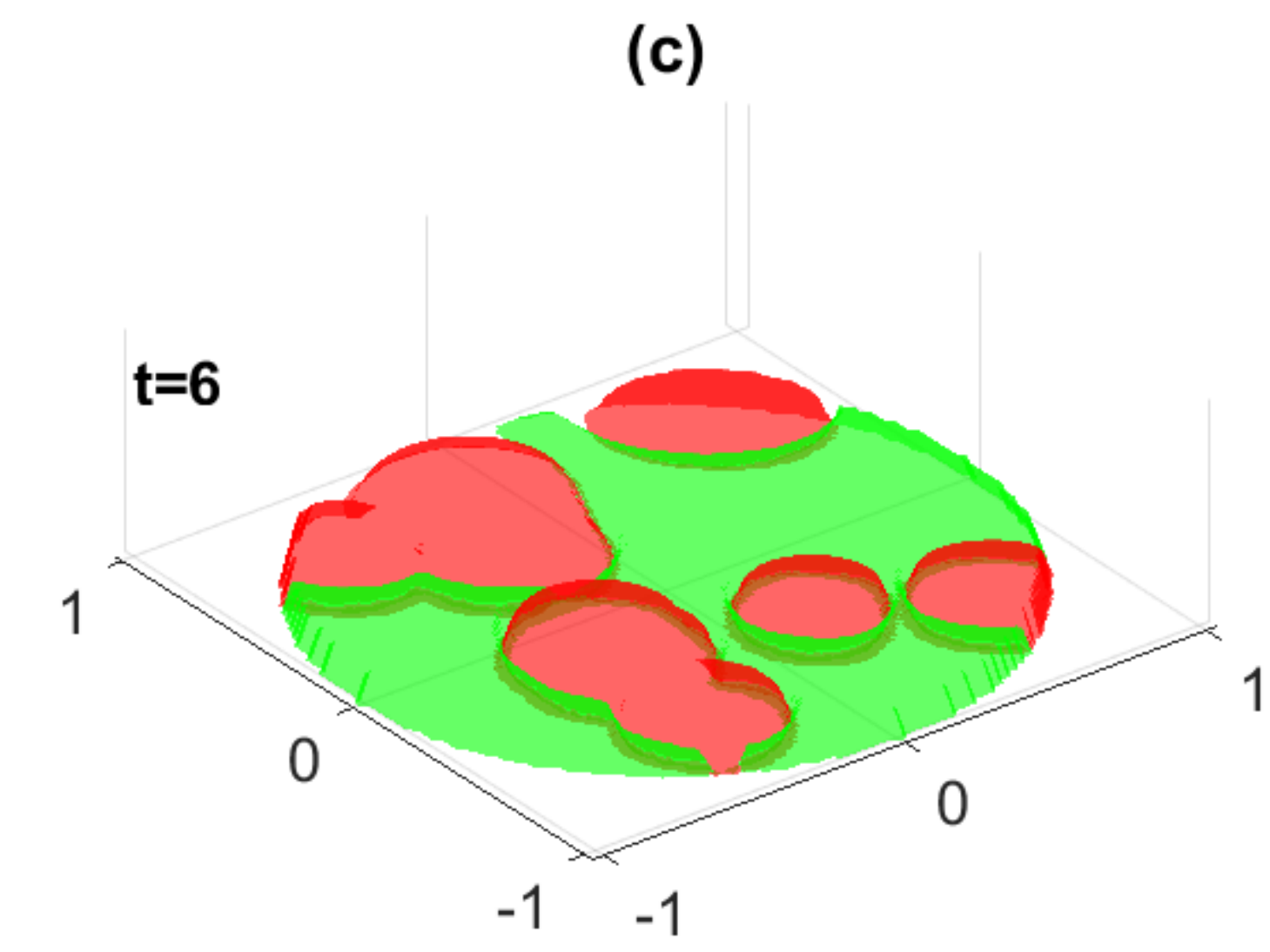}
		\end{center}
		\caption{\textit{Cell co-culture for species $ u_1 $ and $ u_2 $ over 6 days. Figure (a)-(c) corresponds to the scenario 3 with $d_1=2,d_2=0.2,\delta_1 =0.1,\delta_2=0$ in Table \ref{Table:Spatial_competition_2}. The number of initial cluster and cell total number follow \eqref{eq3.7}. Other parameters are given in Table \ref{TABLE1}. }}
		\label{FIG10}
	\end{figure}
	\begin{figure}[H]
		\begin{center}
			\includegraphics[width=0.45\textwidth]{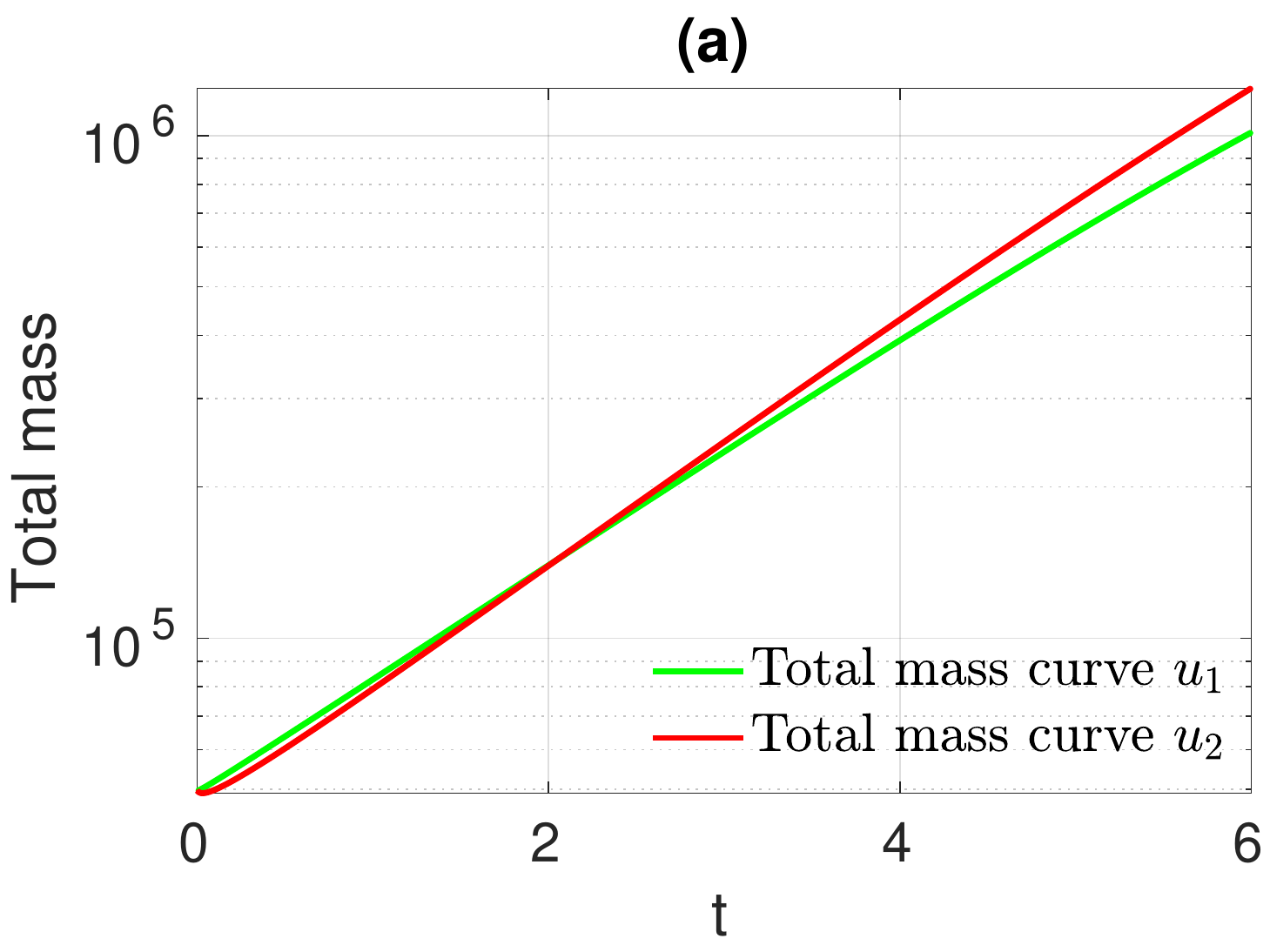}\includegraphics[width=0.45\textwidth]{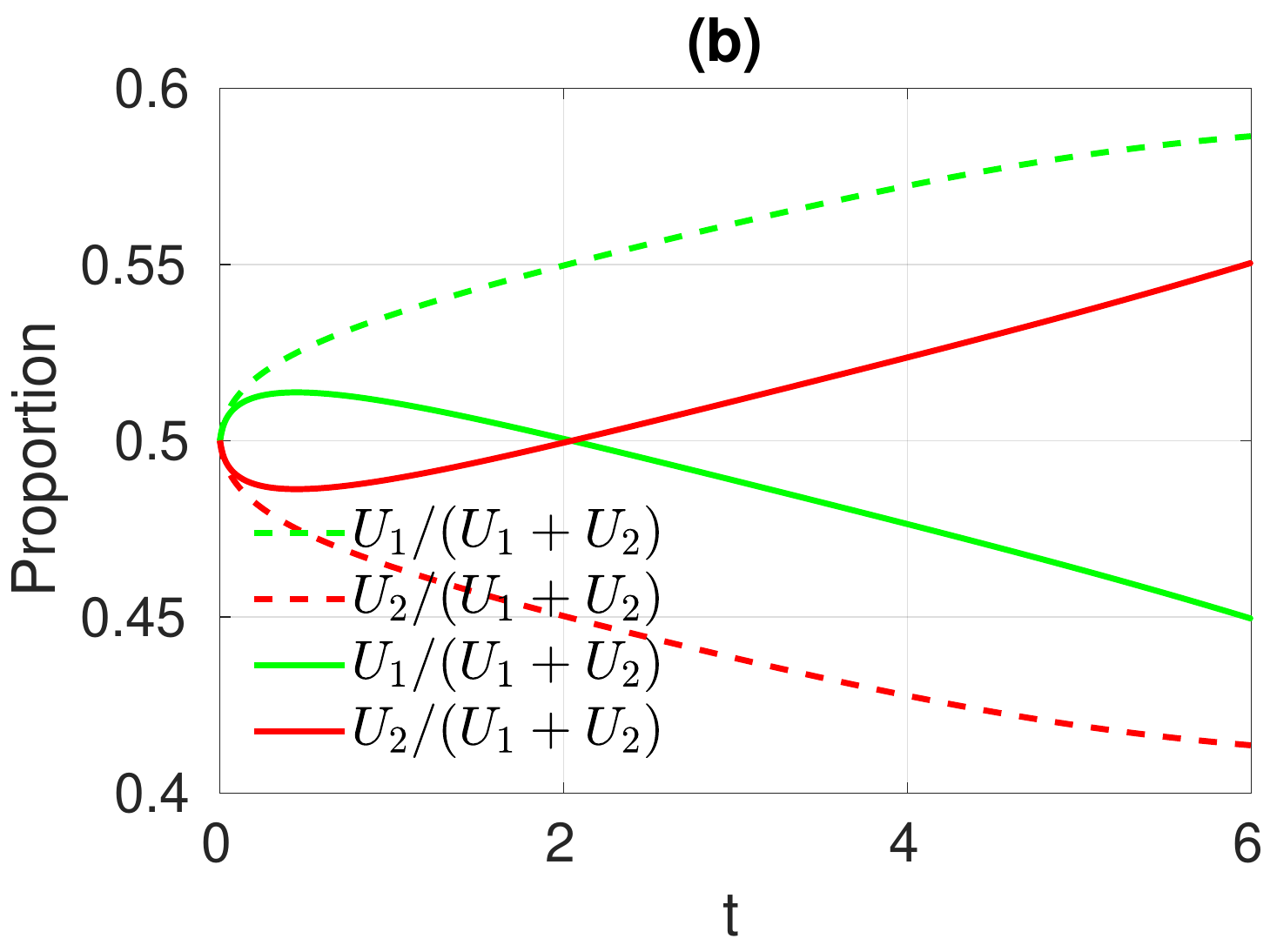}
		\end{center}
		\caption{\textit{Evolution of the  total number (in log scale) and its proportion for species $ u_1 $ and $ u_2 $ over 6 days. Figure (a) is the total number plot corresponding to the scenario 3 in Figure \ref{FIG10}. In Figure (b), the dashed lines corresponds to the population ratios of scenario 2 with $d_1=2,d_2=0.2,\delta_1 =0,\delta_2=0$ in Table \ref{Table:Spatial_competition} and while the solid lines corresponds to scenario 3 with $d_1=2,d_2=0.2,\delta_1 =0.1,\delta_2=0$ in Table \ref{Table:Spatial_competition_2}.  Other parameters are given in Table \ref{TABLE1} and \eqref{eq3.7}.}}
		\label{FIG11}
	\end{figure}
	By including now a drug treatment, we can see from Figure \ref{FIG10} and Figure \ref{FIG11} that between day $0$ and day $2$, the population $u_1$ dominates over $u_2$ thanks to a larger dispersion rate. After day $2$, since the drug is killing the cell from species $u_1$ while the drug has no effect on the species $u_2$, the species $u_2$ finally takes over the species $u_1$. It leads to a gradual increase in its proportion of the population ratio. 
	
	In the numerical simulations for the scenarios  $1$ and $2$ in Table \ref{Table:Spatial_competition}, we let the cell dynamics of the two species be almost equal. Thus the competition due to the cell dynamics is almost negligible. We have shown the dispersion coefficient of populations can have a great impact on the population ratio after 6 days.
	
	In the simulation for scenario 3 in Table \ref{Table:Spatial_competition_2}, we can observe that despite the competitive exclusion, a larger dispersion coefficient can lead to a short-term advantage in the population. In the long term, the competitive exclusion principle still dominates.

	\section{Conclusion and Discussion}
	From the experimental data in the work of Pasquier et al. \cite{Pasquier2011}, we modeled the mono-layer cell co-culture by a hyperbolic Keller-Segel equation \eqref{eq2.1}. We proved the local existence and uniqueness of solutions by using the notion of the solution integrated along the characteristics in Theorem \ref{thm:existence} and proved the conservation law in Theorem \ref{thm:conservation}. For the asymptotic behavior, we analyzed the problem numerically in Section 3.
	
	In Section 3.1 we discussed the competitive exclusion principle, indicating that the asymptotic behavior of the population depends only on the relationship between the steady states $\bar{P}_1$ and $\bar{P}_2$ (see \eqref{eq:newequilibrium} for definition) which is different from the ODE case. We found that except for the case $P_1=P_2$, the model with spatial segregation always exhibits an exclusion principle. 

Even though the long term dynamics of cell density is decided by the relative values of the equilibrium, the short term behavior need a more delicate description. We  studied two factors which may influence the population ratios. The first factor is the initial cell distribution, as measured by the initial total cell number and the law of initial distribution. We found that the impact of the initial distribution on the population ratio lies in the initial total cell number but not in the law of initial distribution.
	
	The second factor influencing the population ratio is the cell movement in space, as measured by the dispersion coefficient $d_i$. We can conclude that in the case of 
sparsely seeded initial distribution by the dispersion rate and cell dynamics. In the transient stage (i.e. before the dish is saturated), the dispersion rate $d_i$ are the dominant factor. Once the surface of the dish is saturated by cells, cell dynamics $u_i\,h(u_1,u_2)$ becomes the key factor. Note that the coefficients $a_{12},a_{21}$ do not play any role in the competition because of the segregation principle.

	We can briefly summarize the following main factors that can influence the population ratio in cell culture for model \eqref{eq2.1}:
	\begin{enumerate}
		\item[(a).] The difference of cell dynamics in the two species (internal factor): if the equilibrium $\bar{P}_1>\bar{P}_2$ (see \eqref{eq:newequilibrium} for definition), then $u_1$ will dominate, $u_2$ will die out (and vice-versa when $\bar{P}_1<\bar{P}_2$) (see Figures \ref{Figure_case1}-\ref{Figure_case234} and Table \ref{TABLE4});
		\item[(b).] If cells are densely seeded at the beginning, despite of competitive advantage, the dominant species can not take out its competitor in a short time (see Figures \ref{FIG6.1}-\ref{FIG6.2}). We also concluded that the law of initial distribution has almost no influence on the population ratio (see Figures \ref{FIG3}-\ref{FIG4});
		\item[(c).] If cells are sparsely seeded at the beginning, we need to distinguish the period of time needed for the cell to occupy the surface of the dish and  the time needed for each species to reach a saturation stage (see Figure \ref{FIG9} and Figure \ref{FIG11}).  
	\end{enumerate}
	
	\section{Appendix}
	\subsection{Invariance of domain $ \Omega $}\label{App:Invariance}
	In this section, we prove the invariance of domain $ \Omega $ for the characteristic equation.
	\begin{assumption}\label{ass:4.1}
		Let $ \Omega\subset \R^2 $ be an open bounded subset with $ \partial \Omega $  of class $ C^2 $.
	\end{assumption}
Since $ \Omega $ is a bounded domain of class $ C^2 $, there exists $U$ a neighborhood
of the boundary $\partial \Omega$ such that the distance function $x \to {\rm dist}(x, \partial \Omega):=\inf_{y \in \partial \Omega} \Vert x-y\Vert$ restricted to $U$ has the regularity $ C^2 $ (see Foote \cite[Theorem 1]{Foote1984}).  Furthermore, by Foote \cite[Theorem 1]{Foote1984} and Ambrosio \cite[Theorem 1 p.11]{Ambrosio2000}, we have the following properties for $ \Omega $.
	\begin{lemma}\label{lem4.2}
		Let Assumption \ref{ass:4.1} be satisfied. Then
		\begin{enumerate}
			\item[\rm(i).]  There exists a small neighborhood $ U $ of $ \partial \Omega $ with $ U \subset \overline{\Omega} $ such that, for every $ x \in U $ there is a unique projection $ P(x)  \in \partial \Omega$ satisfying
			${\rm dist}(x,P(x))={\rm dist}(x,\partial \Omega) $.   
			\item[\rm(ii).]  The distance function $x \mapsto \delta(x):={\rm dist}(x,\partial \Omega) $ is $ C^2 $ on $ U \backslash \partial \Omega $.
			\item[\rm(iii).] For any $ x\in U $, $ \nabla \delta(x) = -\nu(P(x)) $ 	where $ \nu(x) $ is the outward normal vector.
		\end{enumerate}
	\end{lemma}

	We consider the following non-autonomous differential equation on $ \Omega $
	\begin{equation}\label{s1.1}
		\begin{cases}  
			x'(t)=f(t,x(t))&t>0\\
			x(0)=x_0\in \Omega.
		\end{cases}
	\end{equation}
	\begin{assumption}\label{ass:4.2}
		The vector field $ f:[0,\infty)\times \overline{\Omega} \to \R^2 $  is continuous and satisfies 
		\begin{equation}\label{s1.2}
			\nu(x)\cdot f(t,x)\leq 0,\quad \forall t>0,\,\forall x\in \partial \Omega.
		\end{equation}
		Moreover, for any $ T>0 $, there exists a constant $ K = K (T) $ such that vector field $ f $ satisfies
		\begin{equation}\label{s1.3}
			|f(t,x)-f(t,y)| \leq K |x-y |,\quad \forall x,y\in \overline{\Omega},\,t\in[0,T].
		\end{equation}
	\end{assumption}
	By \eqref{s1.3}, we have the existence and uniqueness of the solutions of \eqref{s1.1} and the solutions may eventually reach the boundary $\partial \Omega$ in finite time. We will prove that \eqref{s1.2} implies that the solutions of \eqref{s1.1} actually  stay in $ \Omega $  and can not attain  boundary $ \partial \Omega $ in finite time under Assumption \ref{ass:4.1}.
	\begin{theorem}
		Let Assumption \ref{ass:4.1} and \ref{ass:4.2} be satisfied. For any $ T>0 $, let $ x(t) $  be the solution of \eqref{s1.1} on $ [0,T] $. Then $ x(t) \in \Omega$ for any $ t\in [0,T] $. 
	\end{theorem}
	
	\begin{proof}
		We prove this theorem by contradiction.
		Let $ t^*\in (0,T] $ be the first time when $ x(t) $ reaches boundary $ \partial \Omega $, i.e.,
		\[ t^*=\inf \lbrace 0<t\leq T : \delta(x(t)) =0  \rbrace.\]
		We can find a $ \theta>0 $ such that, $ x(t)\in U \cap \overline{\Omega}$ for any $ t\in [t^*-\theta,t^*] $. Since $ t\to x(t) $ is $ C^1 $, the mapping $ t\mapsto\delta(x(t))$ is $ C^1 $ on $ [t^*-\theta, t^*] $. By  Lemma \ref{lem4.2} (iii), we  have
		\begin{equation}\label{s1.4}
			\frac{d}{dt} \delta(x(t)) =  x'(t)\cdot \nabla \delta(x(t)) =-f(t,x(t)) \cdot \nu (y(t)), 
		\end{equation}
		where $ \nu  $ is the outward normal vector and  $ y(t):= P_{\partial \Omega}(x(t)) $ is the unique projection of $ x(t) $ onto $ \partial \Omega $. By assumption \eqref{s1.2}, we have
		\[-f(t,x(t)) \cdot \nu (y(t))= \big(f(t,y(t))-f(t,x(t))\big) \cdot \nu (y(t)) - f(t,y(t))\cdot \nu (y(t))\geq \big(f(t,y(t))-f(t,x(t))\big) \cdot \nu (y(t)).\]
		 Hence \eqref{s1.4} becomes
		\begin{align*}
			\frac{d}{dt} \delta(x(t)) =&-f(t,x(t)) \cdot \nu (y(t))\\
			\geq& \big(f(t,y(t))-f(t,x(t))\big) \cdot \nu (y(t))\\
			\geq& -\left|f(t,y(t))-f(t,x(t))\right| |\nu (y(t))|\\
			\geq& -K|y(t)-x(t)|=-K \delta(x(t)),\quad t\in [t^*-\theta,t^*],
		\end{align*}
		which yields
		\[ \delta(x(t))\geq \delta(x(t^*-\theta)) e^{-K(t-t^*+\theta)},\quad \forall t\in [t^*-\theta, t^*], \]
		and $\delta(x(t^*-\theta))>0$ implies $ \delta(x(t^*))>0 $  which contradicts our assumption $\delta(x(t^*))=0$.
	\end{proof}
	\subsection{Proof of Theorem \ref{thm:existence}}\label{app:existence}
	\noindent\textbf{Solution integrated along the characteristics.} Let us temporarily suppose $ u\in C^1\left([0,T]\times\Omega\right) $, we can rewrite the first equation in \eqref{eq1.3} as
	\begin{align*}
		\partial _{t}u(t,x) -d\,\nabla u(t,x)\cdot\nabla P(t,x)&= 
		u(t,x)h(u(t,x))+d\,u(t,x)\Delta P(t,x)\\
		&=u(t,x)\left(h(u(t,x))+\frac{d}{\chi}(P(t,x)-u(t,x))\right).
	\end{align*}
	Moreover, if we differentiate the solution along the characteristic with respect to $ t $ then
	\begin{equation}\label{eq2.4}
		\begin{aligned}
			\frac{\dd }{\dd t} u(t,\Pi(t,0;x)) &=\partial_t u(t,\Pi(t,0;x))+\nabla u(t,\Pi(t,0;x)) \cdot \partial_t \Pi(t,0;x)\\
			&= \partial_t u(t,\Pi(t,0;x))-d\, \nabla u(t,\Pi(t,0;x)) \cdot \nabla P(t,\Pi(t,0;x))\\
			&= u(t,\Pi(t,0;x))\left(h(u(t,\Pi(t,0;x)))+\frac{d}{\chi}\left(P(t,\Pi(t,0;x))-u(t,\Pi(t,0;x))\right)\right).
		\end{aligned}
	\end{equation}
	The solution along the characteristics can be written as
	\vspace{-0.5em}
	\begin{equation*}
		u(t,\Pi(t,0;x))=u_0(x)\exp\left(\int_0^t h(u(l,\Pi(l,0;x)))+\frac{d}{\chi}\big(P(l,\Pi(l,0;x))-u(l,\Pi(l,0;x))\big)\dd l\right).
	\end{equation*}
	For the simplicity of notation, we let $d=\chi=1$ in our following discussion and define $w(t,x):=u(t,\Pi(t,0;x))$. We construct the following Banach fixed point problem for the pair $ (w,P) $. For each $ (w,P) $, we let
	\begin{equation}\label{eq:w^1}
		w^{1}(t,x)=u_0(x)\exp \Big(\int_{0}^{t}F(w^{}(l,x))+P^{}(l,\Pi^{}(l,0;x))\dd l\Big).
	\end{equation}
	where we set $ F(u)=h(u)-u $ for any $ u\geq 0 $ and we define
	\begin{equation}\label{eq:fixpoint}
		\mathcal T \begin{pmatrix}
			w^{}(t,x)\\
			P^{}(t,x)
		\end{pmatrix}:=\begin{pmatrix}
			w^{1}(t,x) \\
			(I-\Delta)^{-1} w^{1}(t,\Pi^{}(0,t;x))
		\end{pmatrix}=\begin{pmatrix}
			w^{1}(t,x)\\
			P^{1}(t,x)
		\end{pmatrix},
	\end{equation}
	where $ (I-\Delta)^{-1} $ is the resolvent of the Laplacian operator with Neumann boundary condition. 
	
	We define
	\begin{gather}\label{eq:X_tauY_tau}
		\begin{aligned}
			X^\tau&:=C^0\big([0,\tau],C^0(\overline{\Omega})\big),\quad Y^\tau:=C^0\big([0,\tau],C^{1}(\overline{\Omega})\big),\\
			\tilde X^\tau&:=\left\lbrace w \in C^0\big([0,\tau],C^0(\overline{\Omega})\big)\;\Big|\; w\geq 0 ,\sup_{t\in [0,\tau]}\|w(t,\cdot)\|_{W^{1,\infty}(\Omega)}\leq C_1  \right\rbrace,\\
			\tilde Y^\tau&:=\left\lbrace P\in C^0\big([0,\tau],C^{1}(\overline{\Omega})\big)\; \Big|\; \sup_{t\in [0, \tau]}\big\|P(t,\cdot)\big\|_{W^{2,\infty}(\Omega)}\leq C_2  \right\rbrace,
		\end{aligned}	
	\end{gather}
	where $ C_i,i=1,2 $ are two constants to be  fixed later. We also set
	\begin{equation*}
		\begin{aligned}
			Z^\tau&:=X^\tau\times Y^\tau, & \tilde Z^\tau&:=\tilde X^\tau\times \tilde Y^\tau.
		\end{aligned}
	\end{equation*}
	Notice $ \tilde Z^\tau $ is a complete metric space for the distance induced by the norm $\left(\|\cdot \|_{X^\tau},\|\cdot \|_{Y^\tau}\right)$. For simplicity, we denote $ \|\cdot\|_{C^{\alpha,k}} := \|\cdot\|_{C^{\alpha,k}(\Omega)} $ and $ \|\cdot\|_{W^{k,\infty}} := \|\cdot\|_{W^{k,\infty}(\Omega)}$ for $ \alpha\in (0,1],k\in \mathbb N_+ $.
	
	\begin{theorem}[Existence and uniqueness of solutions]
		For any initial value $u_0\in W^{1,\infty}(\Omega)$ and $ u_0\geq 0 $, for any $ C_1$, $ C_2  $ large enough in \eqref{eq:X_tauY_tau}, there exists  $\tau=\tau(C_1,C_2)>0$ such that the mapping $\mathcal T$ has a unique fixed point in $\tilde Z^\tau$. 
	\end{theorem}
	\begin{proof}
		For any positive initial value $ u_0\in  W^{1,\infty}(\Omega) $ and $ r>0 $, we fix $ C_1 $ to be a constant such that $ 4\|u_0 \|_{W^{1,\infty}} \leq C_1 $ and $ C_2 $ is a constant defined in \eqref{eq:def_C2} later in the proof.	
		
		We also denote 
		\[\begin{pmatrix}
		w^0\\P^0
		\end{pmatrix}= \begin{pmatrix}
		u_0\\
		(I-\Delta)_{\mathcal{N}}^{-1} u_0
		\end{pmatrix} \]
		and let $\overline{B_{\tilde Z^\tau}}\left(\begin{pmatrix}
		w^0\\
		P^0
		\end{pmatrix},r\right)$ be the closed ball centered at $ \begin{pmatrix}
		w^0\\
		P^0
		\end{pmatrix} $ with radius $ r $ in
		$ \tilde Z^\tau=\tilde X^\tau\times \tilde Y^\tau $ with usual product norm 
		\[ \left\Vert\begin{pmatrix}
		w\\P
		\end{pmatrix}\right\Vert_{\tilde Z^\tau} :=\|w \|_{X^\tau}+\|P \|_{Y^\tau}\]
and we set
		\[  
		\kappa := \left\|  \begin{pmatrix}
		w^0\\
		P^0
		\end{pmatrix}\right\|_{\tilde{Z}^\tau}+r. 
		\]
Suppose $\begin{pmatrix}
		w^{}\\P^{}
		\end{pmatrix}\in \overline{B_{Z^\tau}}\left(\begin{pmatrix} w^0\\P^0 \end{pmatrix},r\right)$,  we need to prove that there exits a $\tau$ small enough  such that the following properties hold
		\begin{enumerate}
			\item[(a).] For any $t\in[0,\tau]$, $\left(w^{1}(t,\cdot),P^{1}(t,\cdot)\right) $ in \eqref{eq:w^1} and \eqref{eq:fixpoint} belong to $ W^{1,\infty}(\Omega)\times W^{2,\infty}(\Omega) $ and their norms satisfy 
			\begin{align}
				\sup_{t\in[0,\tau]} \| w^{1}(t,\cdot)\|_{W^{1,\infty}} \leq C_1,\label{condition1aw}\\
				\sup_{t\in[0,\tau]} \| P^{1}(t,\cdot)\|_{W^{2,\infty}} \leq C_2.\label{condition1aP}
			\end{align}
			\item[(b).]  Moreover, we have
			\begin{align}
				\|w^{1}-w^0\|_{X^\tau}&\leq \frac{r}{2}，\label{condition1bw}\\
				\|P^{1}-P^0\|_{Y^\tau}&\leq \frac{r}{2}.\label{condition1bP}
			\end{align}
		\end{enumerate}
		Moreover, we plan to show that the mapping is a contraction: there exists a $ \theta \in (0,1) $ such that for any $ \begin{pmatrix}
		\tilde{w}^{}\\
		\tilde{P}^{}
		\end{pmatrix},\begin{pmatrix}
		w^{}\\
		P^{}
		\end{pmatrix} \in \overline{B_{\tilde Z^\tau}}\bigg(\begin{pmatrix}
		w^0\\
		P^0
		\end{pmatrix},r\bigg) $ we have
		\begin{equation}\label{condition2}
			\left\| \mathcal T \begin{pmatrix}
				\tilde{w}^{}\\
				\tilde{P}^{}
			\end{pmatrix} -\mathcal T \begin{pmatrix}
				w^{}\\
				P^{}
			\end{pmatrix} \right\|_{\tilde Z^\tau}\leq \theta 	\left\| \begin{pmatrix}
				\tilde{w}^{}\\
				\tilde{P}^{}
			\end{pmatrix} - \begin{pmatrix}
				w^{}\\
				P^{}
			\end{pmatrix} \right\|_{\tilde Z^\tau}	.
		\end{equation}
		\noindent\textbf{Step 1.} We show that there exists a $\tau $ small enough such that for any $ (w^{},P^{})\in \tilde{X}^\tau\times \tilde{Y}^\tau $ then 
		\[ \sup_{t\in [0,\tau]} \| w^{1}(t,\cdot)\|_{W^{1,\infty}} \leq C_1, \]
		where $ w^1 $ is defined in \eqref{eq:w^1}.\medskip\\
		Indeed, since $\nabla P^{}(t,\cdot)$ is Lipschitz continuous, then $x\to \Pi^{}(t,0,x)$ is also Lipschitz continuous. Since $\Pi^{}(t,0;\cdot)$ maps $\Omega $ into $\Omega$, we have
		\begin{align*}
			\|P^{}(t,\Pi^{}(t,0;\cdot)) \|_{W^{1,\infty}}&\leq\| P^{}(t,\Pi^{}(t,0;\cdot
			))\|_{L^\infty}+\| \nabla P^{}(t,\cdot)\|_{L^\infty}\| \Pi^{}(t,0;\cdot
			)\|_{W^{1,\infty}} \\
			&\leq \|P^{}(t,\cdot)\|_{W^{1,\infty}}\max\{\|\Pi^{}(t,0;\cdot)\|_{W^{1,\infty}},1\}.
		\end{align*}
		For any $t\in [0,\tau]$, we let $ \tilde{F} := \sup_{u\in [0,\kappa]} \left\{|F(u)|+|F'(u)| \right\} $. By the definition of $w^1$ in \eqref{eq:w^1}, we have
		\begin{align}
			\|w^{1}&(t,\cdot) \|_{W^{1,\infty}} \notag\\
			&\leq \| u_0\|_{W^{1,\infty}} \left\|\exp\left\{\int_0^t F(w(l,\cdot))+P(l,\Pi(l,0,\cdot)) \dd l \right\}\right\|_{W^{1,\infty}}\notag\\
			&\leq \| u_0\|_{W^{1,\infty}} \left\|\exp\left\{\int_0^t F(w(l,\cdot))+P(l,\Pi(l,0,\cdot)) \dd l \right\}\right\|_{L^{\infty}}\notag\\
			&\quad \times \left(1+\int_0^t \|F(w^{}(l,\cdot)) \|_{W^{1,\infty}} +\|P^{}(l,\Pi^{}(l,0,\cdot)) \|_{W^{1,\infty}} \dd l \right)\notag\\
			&\leq \| u_0\|_{W^{1,\infty}} \exp \left\{\int_0^t \|F(w^{}(l,\cdot)) \|_{L^{\infty}} +\|P^{}(l,\Pi^{}(l,0,\cdot)) \|_{L^{\infty}} \dd l \right\}\notag\\
			&\quad \times \Big(1+\tau \tilde{F}\max\{\sup_{l\in [0,\tau]}\|w^{}(l,\cdot) \|_{W^{1,\infty}},1\}+\tau \|P^{}(l,\cdot)\|_{W^{1,\infty}}\max\{\|\Pi^{}(l,0,\cdot)\|_{W^{1,\infty}},1\} \Big)\notag\\
			&\leq \| u_0\|_{W^{1,\infty}} e^{\tau \left(\tilde{F}+\kappa\right) }\Big(1+\tau \tilde{F}\max\{C_1,1\}+\tau \kappa \max\{\|\Pi^{}(l,0,\cdot)\|_{W^{1,\infty}},1\} \Big)\label{eq:w1b}
		\end{align}
		Next we estimate $ \max\left\{\sup_{l\in [0,\tau]}\|\Pi^{}(l,0,\cdot)\|_{W^{1,\infty}},1\right\} $. We have for any $t,s\in [0,\tau]$
		\begin{equation*}
			\Pi^{}(t,s;x) =x-\int_s^t \nabla P^{}(l,\Pi^{}(l,s;x))\dd l.
		\end{equation*}
		Since $ \Omega$ is the unit open disk, $ \| x\|_{W^{1,\infty}(\Omega)}=1+\sqrt{2}\leq 3 $.  We can obtain the following estimate
		\begin{equation*}
			\begin{aligned}
				\| \Pi^{}(t,s;\cdot) \|_{W^{1,\infty}} &\leq 3+ \int_s^t \| \nabla P^{}(l,\Pi^{}(l,s;\cdot))\|_{W^{1,\infty}} \dd l\\
				&\leq 3 + \sup_{l\in [s,t]}\| \nabla P^{}(l,\cdot)\|_{W^{1,\infty}} \int_s^t  \max\left\{\|\Pi^{}(l,s;\cdot) \|_{W^{1,\infty}},1\right\} \dd l\\
				&\leq 3 + C_2 \int_s^t  \max\{\|\Pi^{}(l,s;\cdot) \|_{W^{1,\infty}},1\} \dd l，
			\end{aligned}
		\end{equation*}
		Thanks to Gr\"onwall's inequality,  we have
		\begin{equation}\label{eq:boundforPi}
			\sup_{t,s\in [0,\tau]}\| \Pi^{}(t,s;\cdot) \|_{W^{1,\infty}}\leq 3 e^{\tau C_2}. 
		\end{equation}
		Substituting the \eqref{eq:boundforPi} into \eqref{eq:w1b} yields
		\begin{equation*}
			\|w^{1}(t,\cdot) \|_{W^{1,\infty}} \leq \| u_0\|_{W^{1,\infty}} e^{\tau \left(\tilde{F}+\kappa\right)}\Big(1+\tau \tilde{F}\max\{C_1,1\}+3\tau \kappa e^{\tau C_2}\Big) . 
		\end{equation*} 
		Since $ C_1 \geq 4 \| u_0\|_{W^{1,\infty}} $, we can choose $ \tau\leq \min\left\{\frac{\ln 2}{\tilde F +\kappa}, \frac{1}{\tilde{F}\max\{C_1,1\}+3\kappa e^{C_2}},1 \right\} $ and we obtain 
		\begin{equation}\label{eq:w1bb}
			\sup_{t\in [0, \tau]}\|w^{1}(t,\cdot) \|_{W^{1,\infty}} \leq C_1.
		\end{equation}
		Thus, Equation \eqref{condition1aw} holds.\medskip\\	
		Let us now check that $ w^{1} $ satisfies  \eqref{condition1bw}. Let $ \chi[u]:=u e^u $, we remark  that $ |e^u-1|\leq ue^u =\chi [u] $ for all $ u\geq 0 $. We have
		\begin{align}
			|w^{1}(t,x)-u_0(x)| &\leq |u_0(x)| \left| \exp\left\{\int_0^t F(w^{}(l,x))+P^{}(l,\Pi^{}(l,0,x))\dd l \right\}-1 \right|\notag\\
			&\leq \|u_0 \|_{C^0} \chi \left[\int_0^t \|F(w^{}(l,\cdot)) \|_{C^0}+\|P^{}(l,\Pi^{}(l,0,\cdot))\|_{C^0}\dd l \right]\notag\\
			&\leq  \|u_0 \|_{C^0} \chi \left[\tau \tilde{F}+\tau\sup_{l\in [0,\tau]}\|P^{}(l,\cdot)\|_{C^0} \right]\notag\\
			&\leq  \|u_0 \|_{C^0} \chi \left[\tau \tilde{F}+\tau\kappa \right],\label{eq:w1a}
		\end{align}
		where $ \tilde{F} = \sup_{u\in [0,\kappa]} \left\{|F(u)|+|F'(u)| \right\} $. 
		From \eqref{eq:w1a} we have
		\begin{equation}\label{eq:w1aa}
			\sup_{t\in [0,\tau]}\|w^{1}(t,\cdot)-u_0(\cdot)\|_{C^0} \leq \|u_0 \|_{C^0} \chi \left[\tau \tilde{F}+\tau\kappa \right].
		\end{equation}		
		Since $ \lim_{u\to 0}\chi[u]=0 $,  it suffice to take $ \tau $ small enough to ensure \eqref{condition1bw}.\medskip\\
		
		\noindent\textbf{Step 2.}
		Next we verify \eqref{condition1aP} and \eqref{condition1bP} for $P^1$ where $P^1$ is defined as the second component of \eqref{eq:fixpoint}. We show that there exists $ \tau  $ small enough such that for any $ (w^{},P^{})\in \tilde{X}^\tau\times \tilde{Y}^\tau $
		\[ \sup_{t\in [0,\tau]}\|P^{1} (t,\cdot)\|_{W^{2,\infty}} \leq C_2. \]
		Thanks to the Schauder estimate \cite[Theorem 6.30]{Gilbarg2001}, there exists a constant $C$ depending only on $ \Omega$ such that 
		\begin{equation*}
			\|P^{1}(t,\cdot)\|_{ C^{2,\frac{1}{2}}}\leq C\| w^{1}(t,\Pi^{}(0,t;\cdot))\|_{C^{0,\frac{1}{2}} }.
		\end{equation*}
		
		Recalling  $ \sup_{t\in[0,\tau]}\|\Pi^{}(0,t;\cdot)\|_{W^{1,\infty}}\leq 3 e^{\tau C_2} $ as a consequence of \eqref{eq:boundforPi}, we have
		\begin{equation*}
			\begin{aligned}
				\|P^{1} (t,\cdot)\|_{W^{2,\infty}} &\leq \|P^{1}(t,\cdot)\|_{C^{2,\frac{1}{2}}} \\
				&\leq C\| w^{1}(t,\Pi^{}(0,t;\cdot))\|_{C^{0,\frac{1}{2}}}\\
				&\leq C\| w^{1}(t,\Pi^{}(0,t;\cdot))\|_{W^{1,\infty}}\\
				&\leq C\| w^{1}(t,\cdot)\|_{W^{1,\infty}}  \max\{\|\Pi^{}(0,t;\cdot)\|_{W^{1,\infty}},1\}\\
				&\leq 3\,C\,C_1e^{\tau C_2}.
			\end{aligned}
		\end{equation*}
		We can now define 
		\begin{equation}\label{eq:def_C2}
		C_2=6\,C\,C_1,
		\end{equation} 
		which only depends on $ \Omega $ and $ \| u_0\|_{W^{1,\infty}} $. Finally, we let $ \tau \leq (\ln2) / C_2 $ and we have
		\[ \|P^{1} (t,\cdot)\|_{W^{2,\infty}} \leq 6\,C\,C_1 = C_2 .\]
		In particular, we have shown \eqref{condition1aP}.\medskip\\
		Next we prove \eqref{condition1bP}. Since $\Omega$ is a two-dimensional unit disk, using Morrey's inequality \cite[Chapter 5. Theorem 6]{Evans1998}, we have
		\begin{equation*}
			\quad \|P^{1}(t,\cdot)-P_0(\cdot) \|_{C^{1,\frac{1}{2}}} 
			\leq C\|P^{1}(t,\cdot)-P_0(\cdot) \|_{W^{2,4}},
		\end{equation*}
		where $C$ is a constant depending only on $\Omega$. For the sake of simplicity, we use the same notation $C$ for a universal constant depending only on $\Omega$ in the following estimates.
		Moreover, by  the classical elliptic estimates we have
		\[ \|P^{1}(t,\cdot)-P_0(\cdot) \|_{W^{2,4}} \leq C \|w^{1}(t,\Pi^{}(0,t;\cdot))-u_0(\cdot) \|_{L^{4}}. \]
		This implies that
		\begin{align*}\label{eq:P2a}
			\|P^{1}(t,\cdot)-P_0(\cdot) \|_{C^1}&\leq C \|w^{1}(t,\Pi^{}(0,t;\cdot))-u_0(\cdot) \|_{L^{4}}\\
			&\leq C \|w^{1}(t,\Pi^{}(0,t;\cdot))-u_0(\cdot) \|_{C^{0}}\\
			&\leq  C \|w^{1}(t,\Pi^{}(0,t;\cdot))-w^{1}(t,\cdot) \|_{C^{0}} + C \|w^{1}(t,\cdot)-u_0(\cdot) \|_{C^{0}}\\
			&\leq 	C \| w^{1}\|_{W^{1,\infty}} \| \Pi^{}(0,t;\cdot) -\cdot\|_{C^0} +C  \|w^{1}(t,\cdot)-u_0(\cdot) \|_{C^{0}}\\
			&\leq 	C\; C_1 \| \Pi^{}(0,t;\cdot) -\cdot\|_{C^0} +C  \|w^{1}(t,\cdot)-u_0(\cdot) \|_{C^{0}}\\
			&\leq 	C\; C_1\; \tau\;\sup_{t\in[0,\tau]}\|\nabla P^{}(t,\cdot)\|_{C^0} +C  \|w^{1}(t,\cdot)-u_0(\cdot) \|_{C^{0}}\\
			&\leq 	C\; C_1\; \tau\; \kappa +C  \|w^{1}(t,\cdot)-u_0(\cdot) \|_{C^{0}}\\
			&\leq  C\; C_1\; \tau\; \kappa +C  \|u_0 \|_{C^0} \chi \left[\tau \tilde{F}+\tau\kappa \right],
		\end{align*}
		where we have used \eqref{eq:w1aa} for the last inequality . We can conclude
		\[ \sup_{t\in[0,\tau]}\|P^{1}(t,\cdot)-P_0(\cdot) \|_{C^1} \to 0, \quad \tau \to 0. \]
		Thus, it suffice to take $ \tau $ small enough to ensure the neighborhood condition \eqref{condition1bP}.\medskip\\
		
		\noindent\textbf{Step 3. Contraction mapping}
		In order to verify \eqref{condition2}, we let $ \begin{pmatrix}
		\tilde{w}^{}\\
		\tilde{P}^{}
		\end{pmatrix},\begin{pmatrix}
		w^{}\\
		P^{}
		\end{pmatrix} \in \overline{B_{\tilde Z^\tau}}\left(\begin{pmatrix}
		w^0\\
		P^0
		\end{pmatrix},r\right) $. We observe that
		\begin{align*}
			\left|\tilde{w}^{1}(t,x)-w^{1}(t,x)\right|&= \bigg|u_0(x)\exp \Big(\int_{0}^{t}F(w^{}(l,x))+P^{}(l,\Pi^{}(l,0;x))\dd l\Big)\\
			&\quad-u_0(x)\exp \Big(\int_{0}^{t}F(\tilde w^{}(l,x))+\tilde P^{}(l,\tilde \Pi^{}(l,0;x))\dd l\Big)\bigg| 
		\end{align*}
		Due to the classical inequality $|e^x-e^y|\leq e^{x+y}|x-y|$  which holds for any $x,y\in \R$, we deduce
		\begin{align}
			\big|\tilde{w}^{1}(t,x)&-w^{1}(t,x)\big| \notag\\
			&\leq \|u_0 \|_{C^0} e^{2\tau (\tilde{F}+\kappa)} \bigg[ \int_0^t\| F(\tilde{w}^{}(l,\cdot)) -F(w^{}(l,\cdot))\|_{C^0} \dd l\notag\\
			&\quad + \int_0^t \|\tilde{P}^{}(l,\tilde{\Pi}^{}(l,0;\cdot)) -P^{}(l,\Pi^{}(l,0;\cdot))\|_{C^0} \dd l\bigg]\notag\\
			&\leq \|u_0 \|_{C^0} e^{2\tau (\tilde{F}+\kappa)} \bigg[ \tau \tilde{F}\sup_{l\in[0,\tau]} \| \tilde{w}^{}(l,\cdot) -w^{}(l,\cdot) \|_{C^0}\notag \\
			&\quad + \tau \sup_{l\in[0,\tau]}\| \tilde{P}^{}(l,\tilde{\Pi}^{}(l,0;\cdot)) -P^{}(l,\tilde{\Pi}^{}(l,0;\cdot)) \|_{C^0}\notag\\
			&\quad+ \tau \sup_{l\in[0,\tau]}\| P^{}(l,\tilde{\Pi}^{}(l,0;\cdot))  -P^{}(l,\Pi^{}(l,0;\cdot)) \|_{C^0}\bigg]\notag\\
			&\leq \|u_0 \|_{C^0} e^{2\tau (\tilde{F}+\kappa)} \Big[ \tau \tilde{F} \sup_{l\in[0,\tau]}\| \tilde{w}^{}(l,\cdot) -w^{}(l,\cdot) \|_{C^0} + \tau \sup_{l\in[0,\tau]}\| \tilde{P}^{}(l,\cdot) -P^{}(l,\cdot) \|_{C^0}\notag\\
			&\quad+ \tau \sup_{l\in[0,\tau]}\| P^{}(l,\cdot)\|_{W^{1,\infty}} \sup_{l\in[0,\tau]}\| \tilde{\Pi}^{}(l,0;\cdot)  -\Pi^{}(l,0;\cdot) \|_{C^0}\Big]\notag\\
			&\leq \tau\|u_0 \|_{C^0} e^{2\tau (\tilde{F}+\kappa)} \Big[  \tilde{F} \| \tilde{w}^{} -w^{} \|_{X^\tau} + \| \tilde{P}^{} -P^{} \|_{Y^\tau}\notag\\
			&\quad + \;C_2\;\sup_{l\in[0,\tau]}\| \tilde{\Pi}^{}(l,0;\cdot)  -\Pi^{}(l,0;\cdot) \|_{C^0}\Big]\label{eq:w2}
		\end{align}
		To estimate $\sup_{l\in[0,\tau]}\| \tilde{\Pi}^{}(l,0;\cdot)  -\Pi^{}(l,0;\cdot) \|_{C^0}$ in \eqref{eq:w2}, we claim that 
		\begin{equation}\label{eq:Pidiff}
			\sup_{t,s\in[0,\tau]}\| \tilde{\Pi}^{}(t,s;\cdot)  -\Pi^{}(t,s;\cdot) \|_{C^0} \leq  \tau e^{\tau C_2} \sup_{t\in[0,\tau]} \| \tilde{P}^{}(l,\cdot) -P^{}(l,\cdot) \|_{C^1} 
		\end{equation}
		Indeed, we can obtain that
		\begin{align*}
			\left|  \tilde{\Pi}^{}(t,s;x)  -\Pi^{}(t,s;x) \right| &= \left|\int_s^t \nabla\tilde{P}^{}(l,\tilde{\Pi}^{}(l,s;x)) -\nabla P^{}(l,\Pi^{}(l,s;x))\dd l  \right|\\
			&\leq \int_s^t \|\nabla\tilde{P}^{}(l,\tilde{\Pi}^{}(l,s;\cdot)) -\nabla P^{}(l,\tilde{\Pi}^{}(l,s;\cdot)) \|_{C^0} \dd l\\
			&\quad +  \int_s^t \|\nabla P^{}(l,\tilde{\Pi}^{}(l,s;\cdot)) -\nabla P^{}(l,\Pi^{}(l,s;\cdot)) \|_{C^0} \dd l\\
			&\leq \tau \sup_{l\in[0,\tau]}\|\nabla\tilde{P}^{}(l,\tilde{\Pi}^{}(l,s;\cdot)) -\nabla P^{}(l,\tilde{\Pi}^{}(l,s;\cdot)) \|_{C^0}\\
			& \quad+\sup_{l\in[0,\tau]}\|\nabla P^{}(l,\cdot) \|_{W^{1,\infty}}\int_s^t \|\tilde{\Pi}^{}(l,s;\cdot) -\Pi^{}(l,s;\cdot) \|_{C^0} \dd l.
		\end{align*}
		This leads to 
		\begin{align*}
			\sup_{t,s\in[0,\tau]}\| \tilde{\Pi}^{}(t,s;\cdot)  -\Pi^{}(t,s;\cdot) \|_{C^0} &\leq  \tau \sup_{l\in[0,\tau]}\|\tilde{P}^{}(l,\cdot) - P^{}(l,\cdot) \|_{C^1}\\
			&\quad+C_2\int_s^t \|\tilde{\Pi}^{}(l,s;\cdot) -\Pi^{}(l,s;\cdot) \|_{C^0} \dd l 
		\end{align*}
		Again due to Gr\"onwall's inequality, we  conclude that  \eqref{eq:Pidiff} holds.
		
		Inserting \eqref{eq:Pidiff} into \eqref{eq:w2} we have
		\begin{align}
			\sup_{t\in[0,\tau]}&\left\|\tilde{w}^{1}(t,\cdot)-w^{1}(t,\cdot)\right\|_{C^0}\notag\\
			&\leq \|u_0 \|_{C^0} e^{2\tau (\tilde{F}+\kappa)} \Big[ \tau \tilde{F} \| \tilde{w}^{} -w^{} \|_{X^\tau} + \tau\| \tilde{P}^{} -P^{} \|_{Y^\tau}+ \tau^2 \;C_2\; e^{\tau C_2}\| \tilde{P}^{} -P^{} \|_{Y^\tau}  \Big]\notag\\
			&\leq \tau \|u_0 \|_{C^0} e^{2\tau (\tilde{F}+\kappa)} \Big[  \tilde{F} \| \tilde{w}^{} -w^{} \|_{X^\tau} +\left(1+\tau \;C_2\; e^{\tau C_2}\right) \| \tilde{P}^{} -P^{} \|_{Y^\tau} \Big]\notag\\
			&\leq L_1(\tau)  \Big[  \| \tilde{w}^{} -w^{} \|_{X^\tau} +\| \tilde{P}^{} -P^{} \|_{Y^\tau} \Big]\label{eq:contraction_w}
		\end{align}
		where we set 
		\[ L_1(\tau):= \tau \|u_0 \|_{C^0} e^{2\tau (\tilde{F}+\kappa)} \left( \tilde{F} +\left(1+\tau \;C_2\; e^{\tau C_2}\right)\right)   \]
		and $ L_1(\tau)\to 0 $ as $ \tau \to 0 $. 
		
		Next we prove the contraction  property for $ \|\tilde{P}^{1}-P^{1}\|_{Y^\tau}  $. As before, applying the same argument of Morrey's inequality and the  classical elliptic estimates, we can deduce
		\begin{equation*}
			\begin{aligned}
				\|\tilde{P}^{1}(t,\cdot)-P^{1}(t,\cdot) \|_{C^1}
				&\leq C \|\tilde{w}^{1}(t,\tilde{\Pi}^{}(0,t;\cdot))-w^{1}(t,\Pi^{}(0,t;\cdot)) \|_{L^{4}}\\
				&\leq C \|\tilde{w}^{1}(t,\tilde{\Pi}^{}(0,t;\cdot))-w^{1}(t,\Pi^{}(0,t;\cdot)) \|_{C^{0}}\\
				&\leq   C \|\tilde{w}^{1}(t,\tilde{\Pi}^{}(0,t;\cdot))-w^{1}(t,\tilde{\Pi}^{}(0,t;\cdot)) \|_{C^{0}}\\
				&\quad +C \|w^{1}(t,\tilde{\Pi}^{}(0,t;\cdot))-w^{1}(t,\Pi^{}(0,t;\cdot))\|_{C^{0}} \\
				&\leq C \|\tilde{w}^{1}(t,\cdot)- w^{1}(t,\cdot)\|_{C^{0}}+C \| w^{1}\|_{W^{1,\infty}} \|\tilde{\Pi}^{}(0,t;\cdot) - \Pi^{}(0,t;\cdot)\|_{C^0}\\
				&\leq C \|\tilde{w}^{1}(t,\cdot)- w^{1}(t,\cdot)\|_{C^{0}}+	C\; C_1 \|\tilde{\Pi}^{}(0,t;\cdot) - \Pi^{}(0,t;\cdot)\|_{C^0}\\
				&\leq C  \|\tilde{w}^{1}(t,\cdot)- w^{1}(t,\cdot)\|_{C^{0}}+C\; C_1\; \tau\;e^{\tau C_2}\sup_{t\in[0,\tau]}\|\tilde{P}^{}(t,\cdot)-P^{}(t,\cdot)\|_{C^1} ,
			\end{aligned}
		\end{equation*}
		where we used \eqref{eq:Pidiff} in the last inequality and $C$ is a constant depending only on $\Omega$.
		Defining  $ L_2(\tau) := 	C\; C_1\; \tau\;e^{\tau C_2}$ and together with \eqref{eq:contraction_w} we obtain
		\begin{equation}\label{eq:contraction_P}
			\sup_{t\in[0,\tau]}\|\tilde{P}^{1}(t,\cdot)-P^{1}(t,\cdot) \|_{C^1} \leq C\,L_1(\tau)\,\Big[  \| \tilde{w}^{} -w^{} \|_{X^\tau} +\| \tilde{P}^{} -P^{} \|_{Y^\tau} \Big] + L_2(\tau) \| \tilde{P}^{} -P^{} \|_{Y^\tau} 
		\end{equation}
		Combing with \eqref{eq:contraction_w} and \eqref{eq:contraction_P} we deduce 
		\begin{equation}\label{eq:contraction}
			\|\tilde{w}^{1}-w^{1} \|_{X^\tau}+\|\tilde{P}^{1}-P^{1} \|_{Y^\tau}\leq \big(C\,L_1(\tau) +L_2(\tau)\big)\Big[  \| \tilde{w}^{} -w^{} \|_{X^\tau} +\| \tilde{P}^{} -P^{} \|_{Y^\tau} \Big],
		\end{equation}
		where $ L_i(\tau) \to 0,\;i=1,2 $ as $ \tau \to 0 $. If $ \tau $ is small enough, this implies \eqref{condition2} for some $\theta\in (0,1)$. 
		Since $ \tilde Z^\tau $ is complete metric space for the distance induced by the norm $\left(\|\cdot \|_{X^\tau},\|\cdot \|_{Y^\tau}\right)$ in $ Z_{\tau} $, the result follows by the classical Banach fixed point theorem.
	\end{proof}
	
	\subsection{Parameter fitting}\label{App:para}
	From the work in \cite{Pasquier2011}, MCF-7 and MCF-7/Doxo cells are cultured at $ 10^5 $ initial cell number separately in $ 60\times 15 $ mm cell dish with or without doxorubicine. We use the cell  proliferation data followed every 12 hours during six days to fit the parameters of the following ordinary differential equation
	\begin{equation}\label{eq:ode}
		\begin{cases}
			\dfrac{du_i}{dt} =u_i(b_i -a_{ii} u_i)-\delta_i u_i\quad i=1,2.\\
			u_i(0)= u_{i,0}.
		\end{cases}
	\end{equation}
	Here we use $ u_1 $ to represent the MCF-7 (sensitive to drug) and $ u_2 $ to  represent the MCF-7/Doxo  (resistant to drug) and 
	$b_i>0$ is the growth rate $ \delta_i $ is the extra mortality rate caused by drug (doxorubicine) treatment and $a_{ii}>0$ is a coefficient which controls the number of saturation.

	In the work \cite{Sutherland1983}  cell proliferation kinetics for MCF-7 is studied over 11 days in 150 $ cm^2 $ flask. Following an inoculation of $ 3 \times 10^5 $ cells at day $ 0 $,  a maximum cell density of $ 8 $ to $ 9 \times 10^7 $ cells/flask was reached at day $ 11 $. Therefore, we assume the saturation number for each species in $ 60 \times 15 $ mm (surface of $ 21.5\, cm^2$) dish satisfies  
	\[\frac{b_i}{a_{ii}}\approx 9\times10^7 \times \frac{21.5\, cm^2}{150\, cm^2}= 1.29 \times 10^7,\quad i=1,2.  \]
	
	By fixing the saturation number, we first estimate the growth rate $ b_i $ of each species under zero drug concentration, namely $ \delta_i=0 $. We divide the cell number by $ u_{i,0}=10^5 $ (the initial cell number) and rescale the parameters as follows  
	\begin{equation}\label{eq:rescalling}
		\tilde{u}_i=\frac{u_i}{10^5},\quad \tilde{a}_i=a_{ii}\times10^5,\quad \tilde{b}_i=b_i.
	\end{equation}
	
	As seen in Figure \ref{FIG_fitting1}, without treatment, MCF-7 and MCF-7/Doxo displayed very similar growth rates, 0.6420 and 0.6359 per day, respectively.
	\begin{figure}[H]
		\begin{center}
\includegraphics[width=0.5\textwidth]{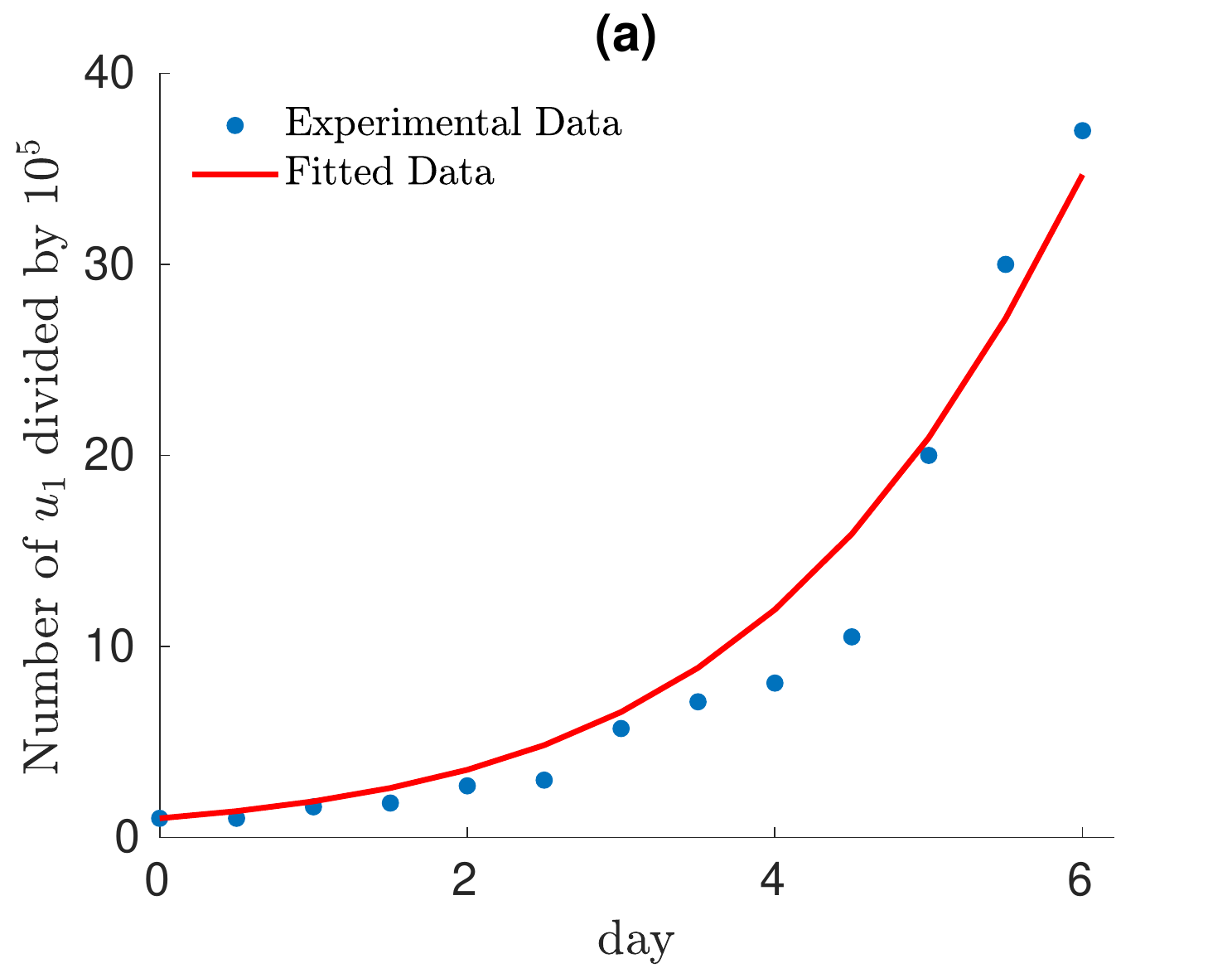}\includegraphics[width=0.5\textwidth]{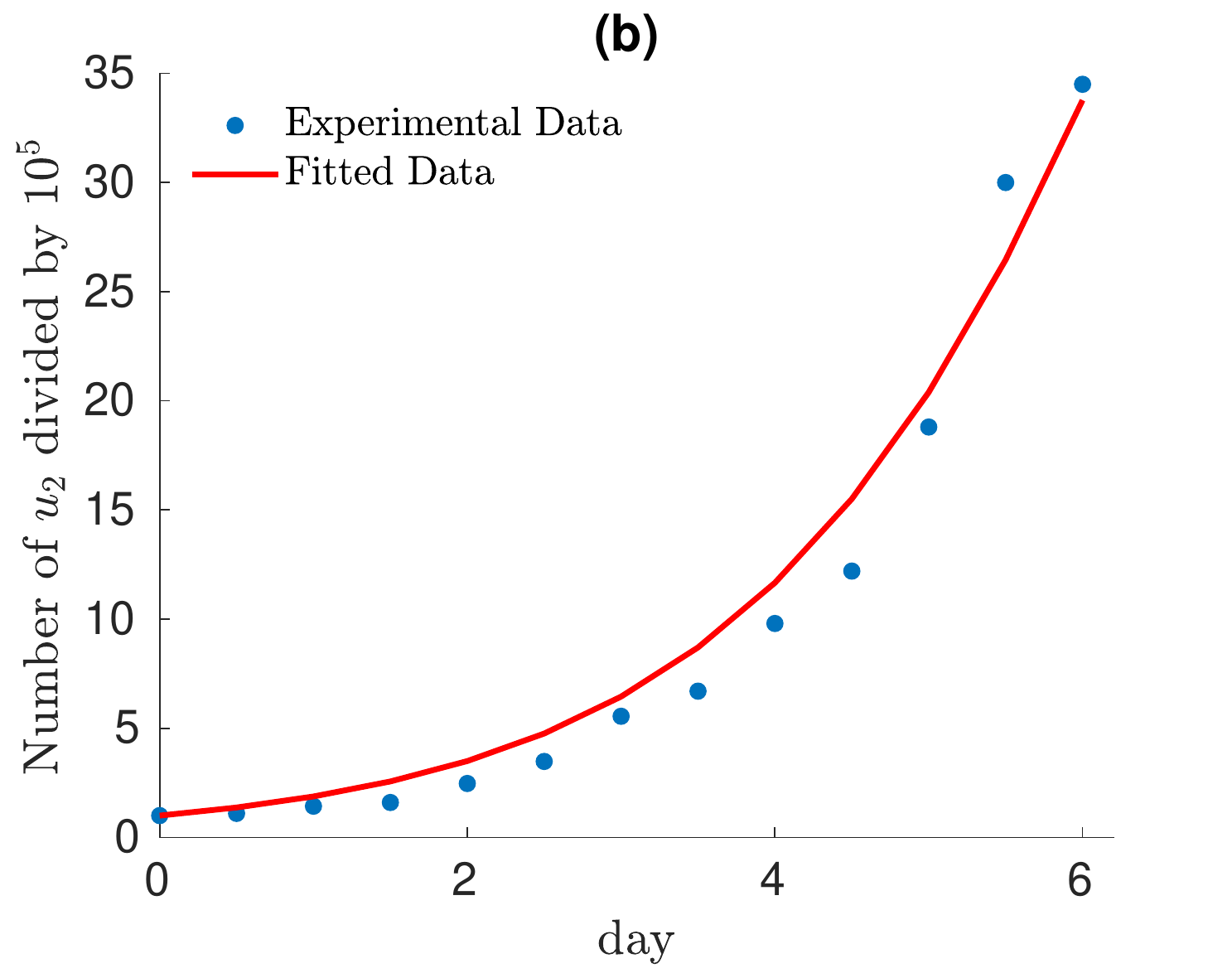}
		\end{center}
		\caption{\textit{Fitting for the parameters {\rm(}under rescaling \eqref{eq:rescalling}{\rm)} in model \eqref{eq:ode}. We plot the experimental data (dots in (a)) of MCF-7 (sensitive to drug) and (dots in (b)) MCF-7/Doxo (resistant to drug) with no drug concentration over 6 days. We obtain an estimation of the growth rates $ b_1= 0.6420, b_2 =0.6359$ and $ a_{11}=0.0050,a_{22}=0.0049 $.}}
		\label{FIG_fitting1}
	\end{figure}
	By fixing the parameters 
	\begin{equation}\label{eq:para_growth}
		b_1=0.6420,\,  a_{11}=0.0050, \quad b_2=0.6359,\, a_{22}=0.0049,
	\end{equation}
	we consider different scenarios with the drug concentration varies from $ 0.1\, \mu M $ to $ 10\,\mu M $ (see Figure \ref{FIG_fitting2}) and we estimate the extra mortality rate $ \delta_i $ for each population due to doxorubicine (see Table \ref{TABLE3}).  
	\begin{figure}[H]
		\begin{center}
			\includegraphics[width=0.5\textwidth]{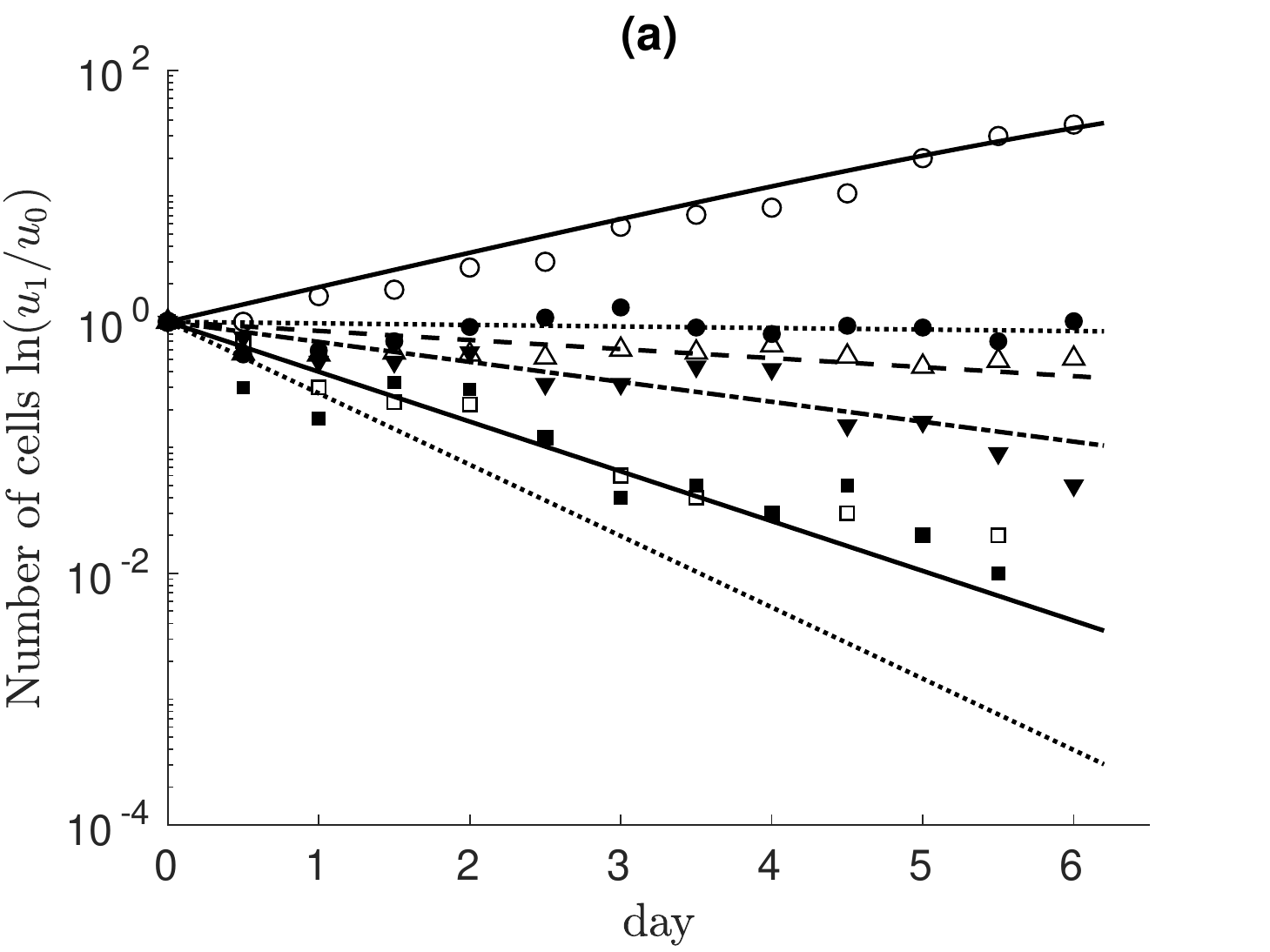}\includegraphics[width=0.5\textwidth]{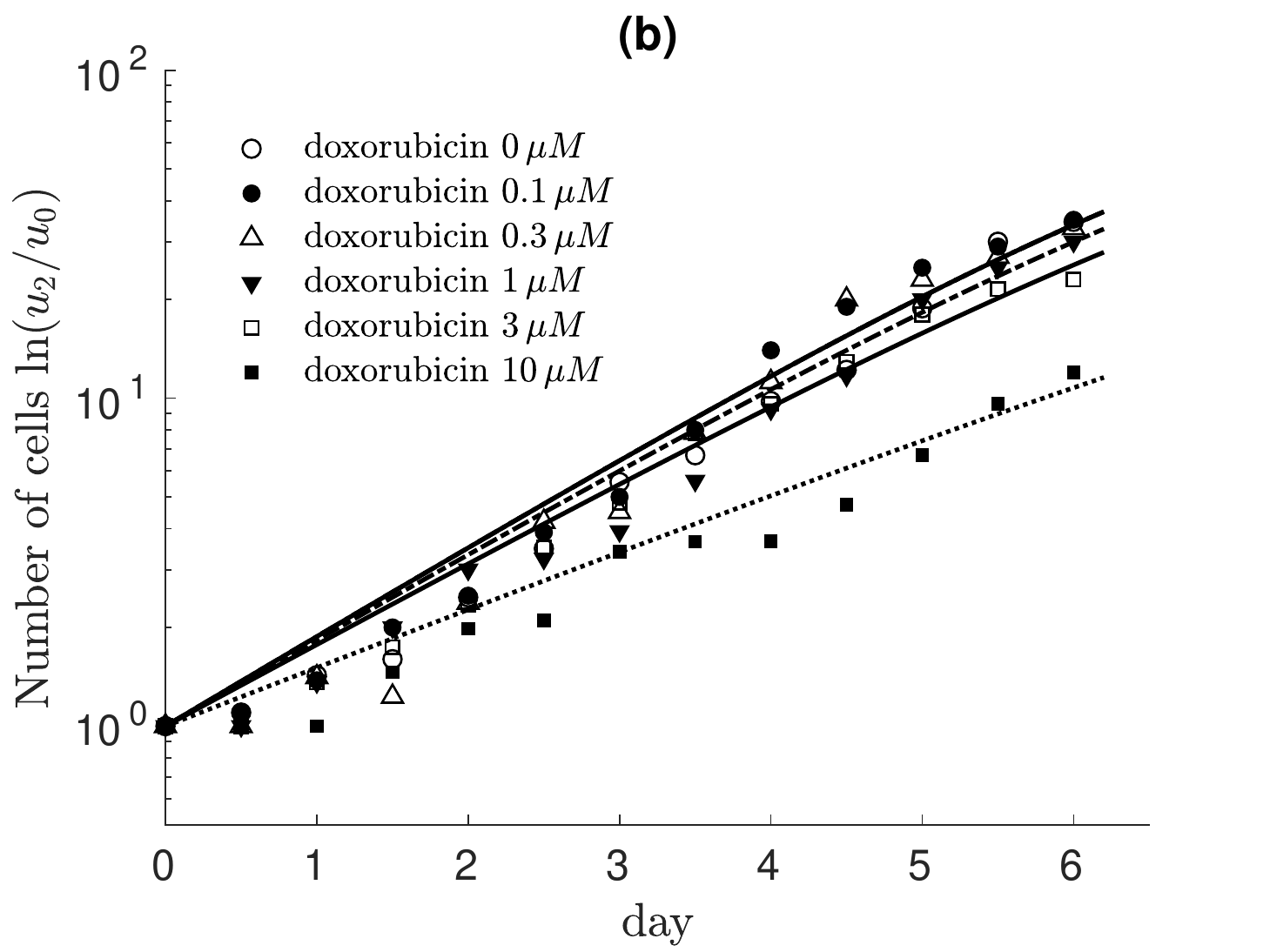}
		\end{center}
		\caption{\textit{{Fitting for the growth curves of MCF-7 (a) and MCF-7/Doxo (b) under different drug concentrations in model \eqref{eq:ode} over 6 days. Cells were grown in the absence or presence of doxorubicine (0.1 to 10 $ \mu M $, corresponding symbols given in the legend in (b)) and counted every 12 hours in a Malassez chamber. Cell counts are expressed as the logarithm of the cell numbers ($ u_i $) divided by the cell number at day 0 ($ u_{i,0} $). We fix the growth rate $ b_i $ and $ a_{ii},\,i=1,2 $ as in \eqref{eq:para_growth}.}}}
		\label{FIG_fitting2}
	\end{figure}
	

	\begin{table}[H]\centering
		\begin{tabular}{ccccccc}
			\doubleRule
			\textbf{Drug concentration} ($\mu M$) & 0 & 0.1 & 0.3 & 1 & 3 & 10 \\ \hline
			\textbf{Extra mortality} $ \delta_{1} $ (day$^{-1} $)  & 0 & 0.6619   & 0.8109   & 1.0118 & 1.5585 & 1.9545  \\ 
			\textbf{Extra mortality} $ \delta_{2} $ (day$^{-1} $)  & 0 & 0   & 0   & 0.0246 & 0.0569 & 0.2192  \\
			\doubleRule
		\end{tabular}
		\caption{\textit{List of the estimation of extra mortality rate  $ \delta_1 $  for the sensitive cell and $ \delta_2 $ for the resistant cell under different concentrations of doxorubicine.}}\label{TABLE3}
	\end{table}
	
	\subsection{Numerical Scheme}\label{App:NumScheme}
	For simplicity, we give the numerical scheme for the following one species and one dimensional model 
	\begin{equation}\label{4.1}
		\begin{cases}
			\begin{aligned}
				&\partial_t u+d\,\partial_x\left(u\partial_x P  \right)=f(u) \\
				&(I-\chi \Delta)P(t,x)=u(t,x)
			\end{aligned}
			&\text{in }(0,T]\times[-L,L]\\
			\partial_x P(t,\pm L) =0 &\text{on } [0,T].
		\end{cases}
	\end{equation}
	The numerical method used is based on finite volume method. We refer to 	\cite{Leveque2002,Toro2013} for more results about this subject. 
	Our numerical scheme reads as follows
	\begin{equation}\label{4.2}
		\begin{aligned}
			&u^{n+1}_i=u^n_i -d\, \frac{\Delta t}{\Delta x}\bigg(\phi(u_{i+1}^n,u_{i}^n)-\phi(u_{i}^n,u_{i-1}^n)\bigg)+\Delta t \, f(u_i^n),\\
			&i=1,2,...,M,\ n=0,1,2,...,N,
		\end{aligned}
	\end{equation}
	with the flux $ \phi(u_{i+1}^n,u_{i}^n) $ defined as
	\begin{equation}\label{4.3}
		\phi(u_{i+1}^n,u_{i}^n)=(v_{i+\frac{1}{2}}^n)^+u_{i}^n -(v_{i+\frac{1}{2}}^n)^-u_{i+1}^n=\begin{cases}
			v_{i+\frac{1}{2}}^n u_{i}^n, & v_{i+\frac{1}{2}}^n\geq 0,\\
			v_{i+\frac{1}{2}}^n u_{i+1}^n, &v_{i+\frac{1}{2}}^n<0.
		\end{cases} 
	\end{equation}
	and
	\begin{equation}\label{4.3bis}
		v_{i+\frac{1}{2}}^n=-\dfrac{l_{i+1}^n-l_i^n}{\Delta x},\ i=0,1,2,\cdots,M,
	\end{equation}
	where we define
	\begin{equation*}
		L^n:=(I-\chi A)^{-1}U^n,\quad \ n=0,1,2,...,N,\quad L_i^n = \big(l_i^n\big)_{M\times 1}\quad
		U^n = \big(u_i^n\big)_{M\times 1}.
	\end{equation*}	
	where $ \chi $ is a constant and $ A=(a_{i,j})_{M\times M} $ is the usual linear diffusion matrix with Neumann boundary condition. Therefore, since the Neumann boundary condition corresponds to a no flux boundary condition, we impose
	\begin{equation}\label{4.4}
		\begin{aligned}
			\phi(u_{1}^n,u_{0}^n)=0,\\
			\phi(u_{M+1}^n,u_{M}^n)=0.
		\end{aligned}
	\end{equation}
	which corresponds to $ l_0=l_1 $ and $ l_{M+1}=l_M $. 
	
	The numerical scheme at the boundary becomes 
	\begin{align*}
		u^{n+1}_1&=u^n_1 -d\, \frac{\Delta t}{\Delta x}\phi(u_{2}^n,u_{1}^n)+\Delta t \, f(u_1^n),\\
		u^{n+1}_M&=u^n_M+d\, \frac{\Delta t}{\Delta x}\phi(u_{M}^n,u_{M-1}^n)+\Delta t \, f(u_M^n).
	\end{align*}
	By this boundary condition, we have the conservation of mass for Equation \eqref{4.1} when the reaction term $ f\equiv 0 $.

\end{document}